\documentclass[11pt]{amsart}
\usepackage[margin=1in,letterpaper]{geometry}

\usepackage{amssymb,amsfonts,amsmath,bbm,mathrsfs,stmaryrd,wasysym}

\usepackage{xcolor}
\usepackage[pdftitle={},
             pdfauthor={Theo Johnson-Freyd},
             pdfproducer={pdfLaTeX},
             pdfpagemode=None,
             bookmarksopen=true
             bookmarksnumbered=true]{hyperref}

\newcommand\arXiv[1]{\href{http://arxiv.org/abs/#1}{\nolinkurl{arXiv:#1}}}
\newcommand\MRnumber[1]{\href{http://www.ams.org/mathscinet-getitem?mr=#1}{\nolinkurl{MR#1}}}
\newcommand\DOI[1]{\href{http://dx.doi.org/#1}{\nolinkurl{DOI:#1}}}
\newcommand\MAILTO[1]{\href{mailto:#1}{\nolinkurl{#1}}}

\usepackage{amsthm}

\newtheorem{theorem}{Theorem}

\newtheorem{proposition}[subsection]{Proposition}

\newtheorem{corollary}[subsection]{Corollary}
\newtheorem{lemma}[subsection]{Lemma}
\newtheorem{metacorollary}[subsection]{Conclusion}

\theoremstyle{definition}

\newtheorem{definition}[subsection]{Definition}
\newtheorem{example}[subsection]{Example}

\newtheorem{remark}[subsection]{Remark}

\renewcommand\qedhere{\hfill\qedsymbol}

\usepackage{tikz}
\usetikzlibrary{arrows,decorations.pathmorphing,decorations.pathreplacing,decorations.markings,shapes.geometric,through,fit,shapes.symbols}
\tikzset{
    dot/.style={circle,draw,fill,inner sep=1.25pt},
    intdot/.style={circle,draw,fill=gray,inner sep=2pt},
    opendot/.style={circle,draw,inner sep=1pt},
    onearrow/.style={postaction={decorate}, decoration={markings,mark=at position .6 with {\arrow[draw,line width=1pt]{<}}}},
    inversearrow/.style={postaction={decorate}, decoration={markings,mark=at position .45 with {\arrow[draw,line width=1pt]{>}}}},
    twoarrows/.style={draw, postaction={decorate}, decoration={markings,mark=at position .35 with {\arrow[draw,line width=1pt]{<}},mark=at position .75 with {\arrow[draw,line width=1pt]{>}}}},
    twoarrowsempty/.style={postaction={decorate}, decoration={markings,mark=at position .3 with {\arrow[draw,line width=1pt]{<}},mark=at position .7 with {\arrow[draw,line width=1pt]{>}}}},
    inversetwoarrows/.style={draw, postaction={decorate}, decoration={markings,mark=at position .35 with {\arrow[draw,line width=1pt]{>}},mark=at position .7 with {\arrow[draw,line width=1pt]{<}}}},
    squiggly/.style={draw, decorate,decoration={snake,amplitude=.3mm,segment length=2mm}},
    fastsquiggly/.style={draw, decorate,decoration={snake,amplitude=.3mm,segment length=1mm}},
    inversesquiggly/.style={draw, decorate,decoration={snake,amplitude=.2mm,segment length=2mm},postaction={decorate,decoration={markings,mark=at position .45 with {\arrow[draw,line width=1pt]{<}}}}},
    tensor/.style={draw,double,double distance=1.5pt},
}

\newcommand\comult{\,\tikz[baseline=(basepoint)]{ 
    \path (0,0) coordinate (basepoint) (0,4pt) node[dot] {};
    \draw[](-6pt,-4pt) -- (0,4pt);
    \draw[](6pt,-4pt) -- (0,4pt);
    \draw[](0,4pt) -- (0,12pt);
  }\,}
\newcommand\mult{\,\tikz[baseline=(basepoint)]{ 
    \path (0,0) coordinate (basepoint) (0,4pt) node[dot] {};
    \draw[](0,-4pt) -- (0,4pt);
    \draw[](0,4pt) -- (-6pt,12pt);
    \draw[](0,4pt) -- (6pt,12pt);
  }\,}
  
\newcommand\graphB{\begin{tikzpicture}[baseline=(basepoint)]
    \path (0,12pt) coordinate (basepoint) (0pt,5pt) node[dot] {} (8pt,12pt) node[dot] {} (-4pt,18pt) node[dot] {} (4pt,25pt) node[dot]{};
    \draw (0pt,-2pt) -- (0pt,5pt);
    \draw[] (0pt,5pt) -- (8pt,12pt);
    \draw[] (8pt,12pt) -- (-4pt,18pt);
    \draw[] (0pt,5pt) .. controls +(-6pt,6pt) and +(-6pt,-6pt) .. (-4pt,18pt);
    \draw[] (-4pt,18pt) -- (4pt,25pt);
    \draw[] (8pt,12pt) .. controls +(6pt,6pt) and +(6pt,-6pt) .. (4pt,25pt);
    \draw (4pt,25pt) -- (4pt,32pt);
  \end{tikzpicture}}

\newcommand\graphD{\begin{tikzpicture}[baseline=(basepoint)]
    \path (0,5pt) coordinate (basepoint) (0pt,2pt) node[dot] {} (8pt,12pt) node[dot] {} (-4pt,17pt) node[dot] {};
    \draw (0pt,-6pt) -- (0pt,2pt);
    \draw[] (0pt,0pt) -- (8pt,12pt);
    \draw[] (8pt,12pt) -- (-4pt,17pt);
    \draw[] (0pt,0pt) .. controls +(-6pt,6pt) and +(-6pt,-6pt) .. (-4pt,17pt);
    \draw (8pt,12pt) -- (12pt,24pt);
    \draw (-4pt,17pt) -- (-4pt,24pt);
  \end{tikzpicture}\,}

\newcommand\graphDbox{\,\tikz[baseline=(main.base)]{
    \node[draw,rectangle,inner sep=2pt] (main) {\graphD};
    \draw (main.south) -- ++(0,-8pt);
    \draw (main.north) ++(-6pt,0) -- ++(0,8pt);
    \draw (main.north) ++(6pt,0) -- ++(0,8pt);
  }\,}
  
\newcommand\graphBbox{\,\tikz[baseline=(main.base)]{
    \node[draw,rectangle,inner sep=2pt] (main) {
      \graphB
    };
    \draw (main.south) -- ++(0,-8pt);
    \draw (main.north) -- ++(0,8pt);
  }\,}

\newcommand\graphA{\,\begin{tikzpicture}[baseline=(basepoint),xscale=-1]
    \path (0,-13pt) coordinate (basepoint) (0pt,-1pt) node[dot] {} (8pt,-12pt) node[dot] {} (-4pt,-17pt) node[dot] {};
    \draw (0pt,6pt) -- (0pt,-2pt);
    \draw[] (0pt,0pt) -- (8pt,-12pt);
    \draw[] (8pt,-12pt) -- (-4pt,-17pt);
    \draw[] (0pt,0pt) .. controls +(-6pt,-6pt) and +(-6pt,6pt) .. (-4pt,-17pt);
    \draw (8pt,-12pt) -- (12pt,-24pt);
    \draw (-4pt,-17pt) -- (-4pt,-24pt);
  \end{tikzpicture}}

\newcommand\graphAbox{\,\tikz[baseline=(main.base)]{
    \node[draw,rectangle,inner sep=2pt] (main) {%
    \graphA%
  };
    \draw (main.north) -- ++(0,8pt);
    \draw (main.south) ++(-6pt,0) -- ++(0,-8pt);
    \draw (main.south) ++(6pt,0) -- ++(0,-8pt);
  }\,}

\newcommand\graphI{\,\tikz[baseline=(basepoint)]{
    \path (0,0pt) coordinate (basepoint) (0,1pt) node[dot] {} (0,7pt) node[dot] {};
    \draw (0,-4pt) -- (0,0pt); \draw (0,8pt) -- (0,12pt);
    \draw (0,1pt) .. controls +(4pt,1pt) and +(4pt,-1pt) .. (0,7pt);
    \draw (0,1pt) .. controls +(-4pt,1pt) and +(-4pt,-1pt) .. (0,7pt);
  }\,}

\newcommand\define[1]{{\em #1}}
\newcommand\cat[1]{\textmd{\textsc{#1}}}

\newcommand\Ff{{\mathcal F}}

\newcommand\bN{{\mathbb N}}

\newcommand\bR{{\mathbb R}}

\newcommand\bS{{\mathbb S}}

\newcommand\bZ{{\mathbb Z}}

\renewcommand{\d}{{\mathrm d}}
\newcommand\B{{\mathbf B}}

\newcommand\C{{\mathrm C}}

\newcommand\st{{\textrm{ s.t.\ }}}

\DeclareMathOperator{\End}{End}

\newcommand{\DGVect}{\cat{DGVect}}

\DeclareMathOperator{\Homology}{H}
\renewcommand\H\Homology

\newcommand{\id}{\mathrm{id}}

\DeclareMathOperator{\Frob}{Frob}

\DeclareMathOperator{\qloc}{Qloc}

\DeclareMathOperator{\di}{di}
\DeclareMathOperator{\pr}{pr}
\DeclareMathOperator{\op}{op}

\newcommand\h{\mathbf{h}}
\newcommand\sh{\mathbf{sh}}

\DeclareMathOperator{\diag}{diag}

\newcommand\isom{\overset{\sim}{\to}}
\newcommand\mono{\hookrightarrow}
\newcommand\epi{\twoheadrightarrow}
\newcommand\onto\epi

\newcommand\<\langle
\renewcommand\>\rangle

\newcommand\cover{\overset\sim\epi}

\renewcommand\[{\llbracket}

\newcommand\shriek{{\textnormal{\textexclamdown}}}

\newcommand\want{\overset{\mathrm{want}}=}
\newcommand\dwant{\overset{\mathrm{want}}:}
\newcommand\wantspacing{\overset{\phantom{\mathrm{want}}}=}

\renewcommand\partial\d

\title{Tree- versus graph-level quasilocal Poincar\texorpdfstring{\'e}{e} duality on \texorpdfstring{\ensuremath{S^{1}}}{S^1}}
\author{Theo Johnson-Freyd}

\begin{document}

\begin{abstract}
Among its many corollaries, Poincar\'e duality implies that the de Rham cohomology of a compact oriented manifold is a shifted commutative Frobenius algebra --- a commutative Frobenius algebra in which the comultiplication has cohomological degree equal to the dimension of the manifold.  We study the question of whether this structure lifts to a ``homotopy'' shifted commutative Frobenius algebra structure at the cochain level.  To make this question nontrivial, we impose a mild locality-type condition that we call ``quasilocality'': strict locality at the cochain level is unreasonable, but it is reasonable to ask for homotopically-constant families of operations that become local ``in the limit.''
To make the question concrete, we 
take the manifold  to be the one-dimensional circle.

The answer to whether  a quasilocal homotopy-Frobenius algebra structure exists turns out to depend on the choice of context in which to do homotopy algebra.  There are two reasonable worlds in which to study structures (like Frobenius algebras) that involve many-to-many operations: one can work at ``tree level,'' 
which in the case of Frobenius algebras seems to correspond roughly to the world of equivalences of homotopy modules of homotopy commutative algebras;
or one can work at ``graph level,'' corresponding to the world of PROPs.  For the tree-level version of our question, the answer is the unsurprising ``Yes, such a structure exists'' --- indeed, it is unique up to a contractible space of choices.  But for the graph-level version, the answer is the surprising ``No, such a structure does not exist.''  Most of the paper consists of proving this nonexistence,
which is controlled by the numerical value of a certain obstruction, and we compute this value explicitly via a sequence of integrals.
\end{abstract}

\maketitle

\tableofcontents

\section{Introduction: In pursuit of quasilocal cochain-level algebraic structures}

Let $X$ be a  smooth compact manifold, with smooth de Rham complex $\Omega^{\bullet}(X)$ and de Rham cohomology $\H^{\bullet}(X)$. Given some algebraic structure on $\H^{\bullet}(X)$, it is reasonable to look for related structures on $\Omega^{\bullet}(X)$.  For example, the commutative algebra structure on $\H^{\bullet}(X)$ lifts to the wedge product on $\Omega^{\bullet}(X)$.  The latter includes \emph{more data} about the topology of $X$ than does the former.  Indeed, given a choice of quasiisomorphism $\Omega^{\bullet}(X) \isom \H^{\bullet}(X)$, homotopy transfer theory moves the commutative algebra structure on $\Omega^{\bullet}(X)$ to a homotopy-commutative algebra structure  on $\H^{\bullet}(X)$ in which the homotopies imposing associativity and higher coherences, far from being zero, are precisely the \define{Massey products} (see e.g.~\cite{Vallette2012}).

The example of the wedge product on $\Omega^{\bullet}(X)$ is misleading in one way: in general, there is no reason to expect algebraic structures at the cochain level to be ``strict.''  We are then faced with two questions:
\begin{enumerate}
\item \label{question1} What type of algebraic structure on $\Omega^{\bullet}(X)$ constitutes a ``lift'' of a given algebraic structure on $\H^{\bullet}(X)$?  
\item \label{question2} Does a lift of a given structure exist, and how unique is it?
\end{enumerate}

An answer to question~(\ref{question1}) comes from the  model category structure on dg operads (or dioperads or properads or PROPs, depending on the type of algebraic structure considered; definitions are reviewed in Section~\ref{section.diproperads})~\cite{MR2954392}.  Let $P$ be a dg operad and $\h P \overset\sim\epi P$ some cofibrant replacement of~$P$.  If $\h P$ acts on $\Omega^{\bullet}(X)$ then so does any other cofibrant replacement (and in a canonical way, up to a contractible space of choices), and $\h P$-actions can be transferred along quasiisomorphisms.  By definition, a \define{homotopy action} of $P$ is an action of $\h P$.  (An object equipped with a $P$-action is also called a \define{$P$-algebra}.)  If $P$ has trivial differential then $\H^{\bullet}(\h P) = P$.  Thus, by taking cohomology, any homotopy action of $P$ on $\Omega^{\bullet}(X)$ induces an action of $P$ on $\H^{\bullet}(X)$.  We will say that this $\h P$-action on $\Omega^{\bullet}(X)$ \define{lifts to cochain level} the action of $P$ on $\H^{\bullet}(X)$.  See Remark~\ref{remark.modelcategorystructure} and Definition~\ref{defn.action}.

With this definition, question~(\ref{question2}) has a somewhat trivial answer.  Any action of $P$ on $\H^{\bullet}(X)$ together with a quasiisomorphism $\Omega^{\bullet}(X) \isom \H^{\bullet}(X)$ induces, again by homotopy transfer theory, an action of $\h P$ on $\Omega^{\bullet}(X)$, which contains no further data than the original homology-level structure.    

Inspection of examples provides a qualitative reason to reject that trivial answer to question~(\ref{question2}): the algebraic structures produced by transferring up from cohomology to cochains tend to include operations that look highly ``nonlocal.''  An operation on $\Omega^{\bullet}(X)$ is \define{local} if for any collection of inputs, the support of the output of the operation is always contained within the intersections of the supports of the inputs.  More generally, for a set $U \subseteq X \times Y$, we will say that a linear map $f:\Omega^{\bullet}(X) \to \Omega^{\bullet}(Y)$ has \define{support in $U$} if for any $X' \subseteq X$ and $\alpha \in \Omega^{\bullet}(X)$ with  support in $X'$, $f(\alpha)$ is supported in $\{y\in Y \st \exists x\in X' \text{ with } (x,y) \in U\}$.  Then local operations are precisely the ones whose support is in the diagonal.

Such strict locality is unreasonable to expect in general, if one wants operations to be defined on arbitrary inputs --- consider the case of intersecting singular chains rather than smooth de Rham forms, where one can define non-transverse intersections by making very small but non-zero perturbations to the inputs --- but it is reasonable to ask for operations that can be made ``as close to local as desired.''
To make sense of this, one can work with smooth families of operations parameterized by $\epsilon \in \bR_{>0}$ that ``become local'' as $\epsilon \to 0$.  Of course, one should make sure that as an operation is ``made close to local,'' it doesn't change except in homotopically-trivial ways.  Put another way, the given family should be ``homotopically constant.''
The following definition is one way of formalizing this.

\begin{definition} \label{defn.quasilocal}
      Given cochain complexes $V^{\bullet}$ and $W^{\bullet}$ of nuclear topological vector spaces, let $\hom(V^{\bullet},W^{\bullet})$ denote the cochain complex of all continuous linear maps $V \to W$ with its usual differential and the topology of bounded convergence, and $V^{\bullet} \otimes W^{\bullet}$ the projective tensor product (see Remarks~\ref{remark.nuclearTVS} and~\ref{remark.signs}).
  Given a manifold $Y$ and a cochain complex $W^{\bullet}$ of nuclear topological vector spaces with differential $\partial_{W}$, let $\Omega^{\bullet}(Y;W) = \Omega^{\bullet}(Y) \otimes W^{\bullet}$ denote the complex of $W$-valued de Rham forms on $Y$, with its usual differential $\d + \partial_{W}$.
  
  Fix $m,n \in \bN$.  Let $\diag : X \mono X^{\times m}\times X^{\times n}$ denote the diagonal embedding.
  The space $\qloc(X)(m,n)$ of \define{quasilocal operations on $\Omega^{\bullet}(X)$ with $m$ inputs and $n$ outputs} is the subcomplex of $\Omega^{\bullet}(\bR_{>0}; \hom( \Omega^{\bullet}(X)^{\otimes m} , \Omega^{\bullet}(X)^{\otimes n})\bigr)$ consisting of those de Rham forms $f$ with the following property: for any open neighborhood $U \supseteq \diag(X)$, there exists $\epsilon_{U} \in \bR_{>0}$ such that the restriction $f|_{(0,\epsilon_{U})}$ is valued in those operations 
  $\Omega^{\bullet}(X)^{\otimes m} \to \Omega^{\bullet}(X)^{\otimes n}$ with support in $U$.
\end{definition}

To justify Definition~\ref{defn.quasilocal}, note that
 a closed  element  $f\in\Omega^{\bullet}(\bR_{>0}; \hom( \Omega^{\bullet}(X)^{\otimes m} , \Omega^{\bullet}(X)^{\otimes n})\bigr)$ of total degree-$k$
 consists of a smooth family $\epsilon \mapsto f(\epsilon)$, $\epsilon \in \bR_{>0}$, of $[\d,-]$-closed degree-$k$ maps $\Omega^{\bullet}(X)^{\otimes m} \to \Omega^{\bullet}(X)^{\otimes n}$, and also a $1$-form $f'(\epsilon)\d \epsilon$ valued in degree-$(k-1)$ maps such that $\int_{\epsilon_{1}}^{\epsilon_{2}}f'(\epsilon)\d \epsilon$ is a homotopy between $f(\epsilon_{1})$ and $f(\epsilon_{2})$.  It is in this sense that    $\Omega^{\bullet}(\bR_{>0}; \hom( \Omega^{\bullet}(X)^{\otimes m} , \Omega^{\bullet}(X)^{\otimes n})\bigr)$ consists of ``homotopically constant'' operations.

Moreover, the inclusion  $\hom( \Omega^{\bullet}(X)^{\otimes m} , \Omega^{\bullet}(X)^{\otimes n}) \hookrightarrow \Omega^{\bullet}\bigl(\bR_{>0}; \hom( \Omega^{\bullet}(X)^{\otimes m} , \Omega^{\bullet}(X)^{\otimes n})\bigr)$,
sending each operation to the corresponding actually-constant family,
 is a quasiisomorphism, and so for the purpose of finding homotopy algebra structures, we expect the two complexes to be interchangeable, with the difference being that for homotopically-constant families it makes sense to talk about quasilocality.
In particular, closed  operations determine well-defined classes in $\H^{\bullet} \bigl(\hom( \Omega^{\bullet}(X)^{\otimes m} , \Omega^{\bullet}(X)^{\otimes n})\bigr)$, and hence well-defined maps $\H^{\bullet}(X)^{\otimes m} \to \H^{\bullet}(X)^{\otimes n}$.

\begin{remark}\label{remark.nuclearTVS}
The smooth de Rham complex $\Omega^{\bullet}(X)$ is a complex of nuclear Fr\'echet spaces; in general, nuclear topological vector spaces have a well-behaved linear algebra (for a classic reference, see~\cite{MR0225131}).
For any manifolds $Y$ and $Z$, 
the inclusion 
of the algebraic tensor product
$\Omega^{\bullet}(Y) \otimes \Omega^{\bullet}(Z) \mono \Omega^{\bullet}(Y\times Z)$ is a quasiisomorphism and identifies the latter as the completion of the former to the {projective} tensor product.  
 The complex $\hom\bigl(\Omega^{\bullet}(Y) , \Omega^{\bullet}(Z)\bigr)$ 
 of continuous linear maps consists of
  those linear maps with integral kernels on $Y \times Z$ that are distributional in the $Y$-direction and smooth in the $Z$-direction.  When $Y$ is compact and oriented, $\hom\bigl(\Omega^{\bullet}(Y) , \Omega^{\bullet}(Z)\bigr)$ naturally embeds quasiisomorphically into $\Omega^{\bullet}_{\mathrm{dist}}(Y\times Z)[\dim Y]$, where $\Omega^{\bullet}_{\mathrm{dist}}$ denotes the complex of distributional de Rham forms, and this embedding intertwines the two notions of support.

In the remainder of this paper, it will be convenient to identify $\Omega^{\bullet}(X)^{\otimes m}$ with $\Omega^{\bullet}(X^{\times m})$.  More generally,  the nuclear topology on the space $\hom\bigl(\Omega^{\bullet}(X)^{\otimes m},\Omega^{\bullet}(X)^{\otimes n}\bigr)$ of $m$-to-$n$ chain-level operations induces a nuclear topology on the space $\Omega^{\bullet}\bigl(\bR_{>0}; \hom( \Omega^{\bullet}(X)^{\otimes m} , \Omega^{\bullet}(X)^{\otimes n})\bigr)$ of homotopically-constant operations, identifying the latter as the subcomplex of $\Omega^{\bullet}(\bR_{>0} \times X^{\times m} \times X^{\times n})[m(\dim X)]$ consisting of those forms which are smooth in the $\bR_{>0}$ direction, distributional in the first $m$ $X$-directions, and smooth in the last $n$ $X$-directions.  Thus we can think about quasilocal operations quite geometrically.
\end{remark}

With Definition~\ref{defn.quasilocal} in hand, we can provide an answer to question~(\ref{question1}) that rules out the trivial answer to question~(\ref{question2}); c.f.\ Definition~\ref{defn.quasilocallift.redux}:

\begin{definition} \label{defn.quasilocallift}
  Let $P$ be an operad (or generalization thereof) acting on $\H^{\bullet}(X)$ and $\h P$ some cofibrant replacement of $P$.  A \define{quasilocal lift of the $P$-action to the cochain level} is an action of $\h P$ on $\Omega^{\bullet}(X)$ inducing the given $P$-action on cohomology in which all operations are quasilocal.
\end{definition}

In this paper, we focus on one of the simplest compact manifolds and set $X = S^{1}$.  Then $\H^{\bullet}(S^{1})$ has a \define{one-shifted commutative Frobenius algebra structure} --- a commutative Frobenius algebra structure in which the comultiplication has cohomological degree $1$ --- and we address question~(\ref{question2}) for this structure.  (In fact, we will ignore the unit and counit, and try to lift only the ``non-unital, non-counital'' Frobenius algebra structure; see Remark~\ref{remark.openunit}.)  Frobenius algebras involve many-to-many operations and so are not controlled by operads, but rather by generalizations thereof.  There are two reasonable generalizations, called \define{dioperads} and \define{properads}, which we review in Definition~\ref{defn.diproperad}.
Our main results are:

\begin{theorem}\label{mainthm.dioperads}
  There is a homotopically-unique quasilocal lift of the 1-shifted Frobenius algebra structure on $\H^{\bullet}(S^{1})$ to the cochain level if ``homotopy Frobenius algebra'' is interpreted in the sense of {dioperads}.
\end{theorem}

\begin{theorem}\label{mainthm.properads}
  There does not exist a quasilocal lift of the 1-shifted Frobenius algebra structure on $\H^{\bullet}(S^{1})$ to the cochain level if ``homotopy Frobenius algebra'' is interpreted in the sense of {properads}.
\end{theorem}

  Before getting to the proofs we assemble the main ingredients.  First, in Section~\ref{section.diproperads} we  discuss dioperads and properads and calculate the cohomology of the space of quasilocal operations.  In Section~\ref{section.Koszul} we discuss Koszul duality and calculate (the beginning of) a presentation of the notion of ``homotopy Frobenius algebra'' in both the dioperadic and properadic senses.  Section~\ref{section.mainthm1} contains the proof of Theorem~\ref{mainthm.dioperads}; it also recalls some obstruction theory that seems to be fairly standard but does not seem to be well-recorded in the literature.  The heart of the paper is Section~\ref{section.mainthm2}, which consists of a proof of Theorem~\ref{mainthm.properads}.  The proof is calculation-intensive: we will show that a certain explicit obstruction fails to vanish, and this requires evaluating certain explicit integrals.  In Section~\ref{section.conclusion}, we provide an entirely different series of calculations, 
  this time working not with smooth de Rham forms but with cellular cochains for a fine cell devision of $S^{1}$, 
  and arrive (in Theorem~\ref{discrete thm.properad}) at the same obstruction.  We draw from Theorems~\ref{mainthm.properads} and~\ref{discrete thm.properad} an important general restriction on cochain-level homotopy Frobenius algebras in Conclusion~\ref{meta}: they are necessarily not strict.

In Theorem~\ref{mainthm.dioperads}, by ``homotopically-unique'' we mean that the space of such lifts is contractible; see Remark~\ref{remark.modelcategorystructure}.
According to~\cite{MR2320654}, a homotopy action in the sense of properads is the same as an action in the sense of PROPs, which are the most general way to control algebraic structures with many-to-many operations.  
Thus Theorem~\ref{mainthm.properads} shows that sometimes the best answer to question~(\ref{question2}) is the surprising ``No, a lift does not exist.''
On the other hand, Theorems~\ref{mainthm.dioperads} and~\ref{mainthm.properads} together illustrate that the choice of whether to work with properads or dioperads really does matter (see also~\cite[Theorem~47]{MR2560406}). 
 As remarked in~\cite{MR1960128}, dioperads seem closely related to the ``cyclic operads'' of~\cite{MR1358617},  another context in which such ``cochain-level Poincar\'e duality'' problems can be posed.
 The dioperadic notion of ``homotopy Frobenius algebra'' also seems closely related to notions in terms of \emph{operadic} homotopy algebras and their homotopy modules.  
 But this author is not aware of a paper making either relation precise.

\begin{remark}
Theorem~\ref{mainthm.dioperads} is also proved in~\cite{PoissonAKSZ} for $S^{1}$ replaced by an arbitrary oriented manifold.  The present paper includes the proof (just for $S^{1}$) in order to be self contained, and also so that we can directly compare the  result with Theorem~\ref{mainthm.properads}.  An indirect argument for Theorem~\ref{mainthm.properads} is also given in~\cite{PoissonAKSZ}, where the problem of quasilocal homotopy Frobenius algebra structures is related to the problem of wheel-free deformation quantization, which is obstructed \cite{Merkulov07,Dito13,MR3268756}.

Note that the version of $\qloc(X)$ given in~\cite{PoissonAKSZ} is less elegant than Definition~\ref{defn.quasilocal}; the two versions are not isomorphic, but with some work can be shown to be quasiisomorphic.  The reader wishing to generalize Definition~\ref{defn.quasilocal} to non-compact manifolds should also keep in mind that without compactness, one-parameter families of operations will generally fail to eventually live in arbitrarily narrow neighborhoods of $\diag(X)$; one must replace $\bR_{>0}$ with an infinite-dimensional parameter space. 
%
\end{remark}

\section{Dioperads and properads, including \texorpdfstring{$\Frob_{1}$}{Frob_1} and \texorpdfstring{$\qloc$}{qloc}}

\label{section.diproperads}

Dioperads were introduced in~\cite{MR1960128} and properads in~\cite{MR2320654}.  Both provide frameworks in which to axiomatize algebraic structures with many-to-many operations.  Our primary references for the theory of dioperads and properads, including their homotopy theory and Koszul duality, are~\cite{MR2320654,MR2560406,MR2572248}.  The main goal of this section is to recall  in Definition~\ref{defn.diproperad} one of many equivalent presentations of these (di/pr)operads.  
We will also formally introduce the (di/pr)operads $\Frob_{1}$ of one-shifted Frobenius algebras and $\qloc(S^{1})$ of quasilocal operations on $S^{1}$, and we will calculate the cohomology of the latter.
Koszul duality will be discussed Section~\ref{section.Koszul}.

\begin{definition}
Let $\cat{DGVect}$ denote the category of cochain complexes over $\bR$, with cohomological conventions (differentials increase degree).  Cochain complexes may be supported in both positive and negative degrees.  We will be inconsistent about whether to put a raised bullet on an object or not, and understand $V = V^{\bullet}$.

 We use the usual Koszul sign rules --- the canonical isomorphism $V \otimes W \cong W \otimes V$, for $V,W\in \DGVect$, sends $v \otimes w \mapsto (-1)^{(\deg v)(\deg w)}w\otimes v$ if $v$ and $w$ are homogeneous --- and the word ``commutative'' is always relative to the Koszul signs.  Let $[n]$ denote the one-dimensional graded vector space satisfying $\dim [n]^{{-n}} = 1$ and $\dim [n]^{{\bullet}} = 0$ if $\bullet\neq -n$, and  shift chain complexes by $V[n] = V \otimes [n]$, so that $V[n]^{\bullet} \cong V^{\bullet + n}$; as discussed further in Remark~\ref{remark.signs}, this introduces signs to formulas involving homogeneous elements.  

  Let $\bS$ denote the groupoid of finite sets and bijections.  An \define{$\bS$-bimodule} is a functor $P : \bS^{\op}\times \bS \to \DGVect$.  Thus, the data of an $\bS$-bimodule is a collection of cochain complexes $P(m,n)$ for $(m,n)\in \bN^{2}$, along with, for each $(m,n)$, an action on $P(m,n)$ of $\bS_{m}^{\op} \times \bS_{n}$, where $\bS_{n}$ denotes the symmetric group on $n$ letters.
\end{definition}

\begin{remark}\label{remark.signs}
  Sign conventions are annoying.  Since the forgetful functor $\cat{DGVect} \to \cat{Vect}$ is faithful and naturally monoidal, it is standard to work with homogeneous elements, which are precisely the elements in the
  image of a complex under  the forgetful functor.  But the forgetful functor is not symmetric, and this is the root cause of the fundamental fact that \emph{there is no perfect sign convention for homogeneous elements}.  There are many mediocre conventions, some better than others, and precise formulas depend on some choice.  Because one can never be sure another author is using precisely the same conventions, it is dangerous to quote other authors' formulas; for this reason, \cite{SerreHowToWrite} argues that one should distrust most of the literature in homological algebra.  On the other hand, the symmetric structure on $\cat{DGVect}$ is uniquely determined by the condition that the braiding $[1]\otimes[1] \to [1]\otimes[1]$ be $-\id$.  It follows that any categorical statement which itself is convention-independent, if carefully proven in one convention, really is universally true.
  
  This paper attempts to use conventions that are simultaneously standard and  clean.  The absolutely cleanest convention maintaining usual syntax is to change the semantics of formulas, working not with homogeneous elements but with ``generalized elements'' (see~\cite{DM1999}).  Unfortunately, most calculations in this paper really do involve particular homogeneous elements, and so we will use the more usual conventions, which we  record here:
  
  Fix a basis element $\iota \in [1]$ (so that $\deg \iota = -1$).  For a homogeneous vector $v\in V^{\bullet}$, the corresponding vector in $V[1]^{\bullet-1}$ will be denoted $v \otimes \iota = v\iota$.  (If we wrote the suspension functor $V \mapsto V[1]$ on the left instead as $V \mapsto \Sigma V$, we would name the suspended vector $\varsigma v$.)  Upon choosing an isomorphism $[1]\otimes [1] \cong [2]$, we can set the basis vector of $[2]$ to be $\iota^{2}$.  (Note that no such choice is compatible with the braiding $[1]\otimes[1] \to [1]\otimes [1]$.)  Given complexes $V$ and $W$, the isomorphism $(V \otimes W)[2] \cong V[1] \otimes W[1]$ on homogeneous elements is $(v\otimes w)\iota^{2} \mapsto (-1)^{\deg w}(v\iota \otimes w\iota)$, since an $\iota$ had to move past the $w$.  Since $[1]$ has trivial differential, if the complex $V$ has differential $\d_{V}$, then the complex $V[1]$ satisfies $\d_{V[1]}(v\iota) = (\d_{V} v) \iota + (-1)^{\deg v}v(\d_{[1]} \iota) = (\d_{V} v)\iota$.  The differential on $V\otimes W$ is $\d_{V\otimes W}(v\otimes w) = \d_{V} v \otimes w + (-1)^{\deg v}v \otimes \d_{W} w$.  This formula follows from the expression $\d_{V\otimes W} = \d_{V} \otimes \id_{W} + \id_{V} \otimes \d_{W}$ along with the following more general convention: if $f: V \to V'$ and $g : W \to W'$ are homogeneous linear maps, then $(f\otimes g)(v\otimes w) = (-1)^{(\deg g)(\deg v)}f(v) \otimes g(w)$, since the $g$ and $v$ had to switch spots.  Note in particular that $f(v\iota) = f(v)\iota$.  We let $\iota^{-1}$ denote the basis vector of $[-1]$ that satisfies $\iota  \iota^{-1} = 1\in [0] = \bR$.
  A manifestation of the nonexistence of a perfect sign convention is that, once $\iota  \iota^{-1} = 1$ is decided, both choices for $\iota^{-1}\iota = \pm 1$ lead to problems;  fortunately, we will not need that latter multiplication.
\end{remark}

\begin{definition} \label{defn.diproperad}
  A \define{nonunital properad} is an $\bS$-bimodule along with, for every tuple of finite sets $\mathbf{m}_{1},\mathbf{m}_{2},\mathbf{n}_{1},\mathbf{n}_{2},\mathbf{k}$ with $\mathbf{k}$ nonempty, a \define{composition} map:
  $$
  \begin{tikzpicture}[baseline=(basepoint),yscale=1.25,xscale=-1]
  \path(0,.25) coordinate (basepoint);
  \draw (-.5,.5) node[circle,draw,inner sep=3,fill=gray] (delta1) {};
  \draw (.5,0) node[circle,draw,inner sep=3,fill=gray] (delta2) {};
  \draw[onearrow]  (-.25,-.5) .. controls +(0,.25) and +(.25,-.5) .. (delta1);
  \draw[onearrow] (-.75,-.5) .. controls +(0,.25) and +(-.25,-.25) .. (delta1);
  \draw[onearrow]  (.25,-.5) .. controls +(0,.25) and +(-.25,-.25) .. (delta2);
  \draw[onearrow] (.75,-.5) .. controls +(0,.25) and +(.25,-.25) .. (delta2);
  \draw[onearrow] (delta1) .. controls +(-.25,.25) and +(0,-.25) .. (-.85,1.1);
  \draw[onearrow] (delta1) .. controls +(.25,.25) and +(0,-.25) .. (-.15,1.1);
  \draw[onearrow] (delta2) .. controls +(-.25,.5) and +(0,-.25) .. (.15,1.1);
  \draw[onearrow] (delta2) .. controls +(.25,.25) and +(0,-.25) .. (.85,1.1);
  \draw[] (-.5,1) node {$\scriptstyle \dots$} (.5,1) node {$\scriptstyle \dots$};
  \draw[] (-.5,-.35) node {$\scriptstyle \dots$} (.5,-.35) node {$\scriptstyle \dots$};
  \draw[] (delta2) ..controls +(-.5,.125) and +(.25,-.25) .. (delta1);
  \draw[] (delta2) ..controls +(-.25,.25) and +(.5,-.125) .. (delta1);
  \draw[] (-.05,.25) node[anchor=base] {$\scriptscriptstyle \cdots$};
    \draw[decorate,decoration={brace,mirror}] (-.9,1.15) -- node[auto,swap] {$\scriptstyle \mathbf m_{2}$} (-.1,1.15);
    \draw[decorate,decoration={brace,mirror}] (.1,1.15) -- node[auto,swap] {$\scriptstyle \mathbf m_{1}$} (.9,1.15);
    \draw[decorate,decoration={brace,mirror}] (-.2,-.55) -- node[auto,swap] {$\scriptstyle \mathbf n_{2}$} (-.8,-.55);
    \draw[decorate,decoration={brace,mirror}] (.8,-.55) -- node[auto,swap] {$\scriptstyle \mathbf n_{1}$} (.2,-.55);
    \draw[decorate,decoration={brace,mirror}] (.25,.35) -- node[anchor=north] {$\scriptstyle \mathbf k$} (-.2,.15);
\end{tikzpicture}
:
   P\bigl(\mathbf m_{1} \sqcup \mathbf k,\mathbf n_{1}\bigr) \otimes  P\bigl(\mathbf m_{2},\mathbf k \sqcup \mathbf n_{2} \bigr) \to P\bigl(\mathbf m_{1} \sqcup \mathbf m_{2}, \mathbf n_{1} \sqcup \mathbf n_{2} \bigr) $$
   The order is chosen to be compatible with the usual order of composition of functions.
   
   The composition maps should 
   intertwine the various $\bS$-actions
   in the obvious way from the picture.  Moreover, we demand an \define{associativity} condition, which is actually four conditions for the types of connected directed  graphs with three vertices and no directed cycles:
  $$
  
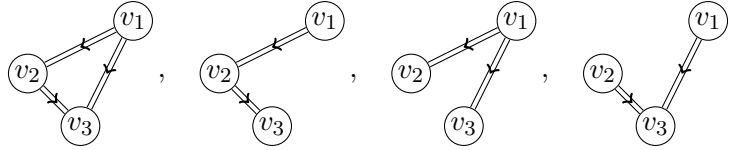
\begin{tikzpicture}[baseline=(v2.base),xscale=-1]
    \path (0,0) node[draw,circle,inner sep = 1pt] (v1) {$v_{3}$};
    \path (.7,.7) node[draw,circle,inner sep = 1pt] (v2) {$v_{2}$};
    \path (-.7,1.4) node[draw,circle,inner sep = 1pt] (v3) {$v_{1}$};
    \draw[tensor,onearrow] (v1) -- (v2);
    \draw[tensor,onearrow] (v1) -- (v3);
    \draw[tensor,onearrow] (v2) -- (v3);
  \end{tikzpicture}\,,\quad
  \begin{tikzpicture}[baseline=(v2.base),xscale=-1]
    \path (0,0) node[draw,circle,inner sep = 1pt] (v1) {$v_{3}$};
    \path (.7,.7) node[draw,circle,inner sep = 1pt] (v2) {$v_{2}$};
    \path (-.7,1.4) node[draw,circle,inner sep = 1pt] (v3) {$v_{1}$};
    \draw[tensor,onearrow] (v1) -- (v2);
    \draw[tensor,onearrow] (v2) -- (v3);
  \end{tikzpicture}\,,\quad
  \begin{tikzpicture}[baseline=(v2.base),xscale=-1]
    \path (0,0) node[draw,circle,inner sep = 1pt] (v1) {$v_{3}$};
    \path (.7,.7) node[draw,circle,inner sep = 1pt] (v2) {$v_{2}$};
    \path (-.7,1.4) node[draw,circle,inner sep = 1pt] (v3) {$v_{1}$};
    \draw[tensor,onearrow] (v1) -- (v3);
    \draw[tensor,onearrow] (v2) -- (v3);
  \end{tikzpicture}\,,\quad
  \begin{tikzpicture}[baseline=(v2.base),xscale=-1]
    \path (0,0) node[draw,circle,inner sep = 1pt] (v1) {$v_{3}$};
    \path (.7,.7) node[draw,circle,inner sep = 1pt] (v2) {$v_{2}$};
    \path (-.7,1.4) node[draw,circle,inner sep = 1pt] (v3) {$v_{1}$};
    \draw[tensor,onearrow] (v1) -- (v2);
    \draw[tensor,onearrow] (v1) -- (v3);
  \end{tikzpicture}
  $$
  The reader is invited to spell out the details of the associativity equations; note that for the last two, one must reverse the order of two factors in a tensor product, and this introduces signs.
  
  A \define{nonunital dioperad} is as above, but the only compositions that are defined are when $\mathbf k$ is a set of size $1$:
  $$
  \begin{tikzpicture}[baseline=(basepoint),yscale=1.25,xscale=-1]
  \path(0,.25) coordinate (basepoint);
  \draw (-.5,.5) node[circle,draw,inner sep=3,fill=gray] (delta1) {};
  \draw (.5,0) node[circle,draw,inner sep=3,fill=gray] (delta2) {};
  \draw[onearrow]  (-.25,-.5) .. controls +(0,.25) and +(.25,-.5) .. (delta1);
  \draw[onearrow] (-.75,-.5) .. controls +(0,.25) and +(-.25,-.25) .. (delta1);
  \draw[onearrow]  (.25,-.5) .. controls +(0,.25) and +(-.25,-.25) .. (delta2);
  \draw[onearrow] (.75,-.5) .. controls +(0,.25) and +(.25,-.25) .. (delta2);
  \draw[onearrow] (delta1) .. controls +(-.25,.25) and +(0,-.25) .. (-.85,1.1);
  \draw[onearrow] (delta1) .. controls +(.25,.25) and +(0,-.25) .. (-.15,1.1);
  \draw[onearrow] (delta2) .. controls +(-.25,.5) and +(0,-.25) .. (.15,1.1);
  \draw[onearrow] (delta2) .. controls +(.25,.25) and +(0,-.25) .. (.85,1.1);
  \draw[] (-.5,1) node {$\scriptstyle \dots$} (.5,1) node {$\scriptstyle \dots$};
  \draw[] (-.5,-.35) node {$\scriptstyle \dots$} (.5,-.35) node {$\scriptstyle \dots$};
  \draw[onearrow] (delta2) -- (delta1);
    \draw[decorate,decoration={brace,mirror}] (-.9,1.15) -- node[auto,swap] {$\scriptstyle \mathbf m_{2}$} (-.1,1.15);
    \draw[decorate,decoration={brace,mirror}] (.1,1.15) -- node[auto,swap] {$\scriptstyle \mathbf m_{1}$} (.9,1.15);
    \draw[decorate,decoration={brace,mirror}] (-.2,-.55) -- node[auto,swap] {$\scriptstyle \mathbf n_{2}$} (-.8,-.55);
    \draw[decorate,decoration={brace,mirror}] (.8,-.55) -- node[auto,swap] {$\scriptstyle \mathbf n_{1}$} (.2,-.55);
\end{tikzpicture}
:
   P\bigl(\mathbf m_{1} \sqcup \{*\},\mathbf n_{1}\bigr) \otimes P\bigl(\mathbf m_{2}, \{*\} \sqcup\mathbf n_{2}\bigr)  \to P\bigl(\mathbf m_{1} \sqcup \mathbf m_{2}, \mathbf n_{1} \sqcup \mathbf n_{2} \bigr) $$
   There are associativity axioms for each of the following types of diagrams:
  $$
  
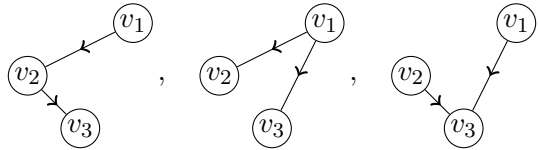
\begin{tikzpicture}[baseline=(v2.base),xscale=-1]
    \path (0,0) node[draw,circle,inner sep = 1pt] (v1) {$v_{3}$};
    \path (.7,.7) node[draw,circle,inner sep = 1pt] (v2) {$v_{2}$};
    \path (-.7,1.4) node[draw,circle,inner sep = 1pt] (v3) {$v_{1}$};
    \draw[,onearrow] (v1) -- (v2);
    \draw[,onearrow] (v2) -- (v3);
  \end{tikzpicture}\,,\quad
  \begin{tikzpicture}[baseline=(v2.base),xscale=-1]
    \path (0,0) node[draw,circle,inner sep = 1pt] (v1) {$v_{3}$};
    \path (.7,.7) node[draw,circle,inner sep = 1pt] (v2) {$v_{2}$};
    \path (-.7,1.4) node[draw,circle,inner sep = 1pt] (v3) {$v_{1}$};
    \draw[,onearrow] (v1) -- (v3);
    \draw[,onearrow] (v2) -- (v3);
  \end{tikzpicture}\,,\quad
  \begin{tikzpicture}[baseline=(v2.base),xscale=-1]
    \path (0,0) node[draw,circle,inner sep = 1pt] (v1) {$v_{3}$};
    \path (.7,.7) node[draw,circle,inner sep = 1pt] (v2) {$v_{2}$};
    \path (-.7,1.4) node[draw,circle,inner sep = 1pt] (v3) {$v_{1}$};
    \draw[,onearrow] (v1) -- (v2);
    \draw[,onearrow] (v1) -- (v3);
  \end{tikzpicture}
  $$
  
  A \define{nonunital coproperad} has instead a \define{decomposition} map $\Delta: P\bigl(\mathbf m_{1} \sqcup \mathbf m_{2}, \mathbf n_{1} \sqcup \mathbf n_{2} \bigr) \to P\bigl(\mathbf m_{1},\mathbf k \sqcup \mathbf n_{1} \bigr) \otimes P\bigl(\mathbf m_{2} \sqcup \mathbf k,\mathbf n_{2}\bigr) $ for each tuple $(\mathbf m_{1},\mathbf m_{2},\mathbf n_{1},\mathbf n_{2},\mathbf k)$ with $\mathbf k \neq \emptyset$, 
  intertwining the $\bS$-actions and
  satisfying coassociativity axioms.  A \define{nonunital codioperad} similarly has decomposition maps whenever $\mathbf k = \{*\}$.

\end{definition}

  We will not use PROPs in this paper, but mention them for completeness: in addition to the properadic compositions, they allow compositions for disconnected graphs.  
  We henceforth drop the word ``non(co)unital,'' understanding in the sequel that (co)(di/pr)operads may be non(co)unital.

\begin{remark}\label{remark.modelcategorystructure}
  The category of (di/pr)operads has a model category structure in which the weak equivalences are the quasiisomorphisms, and the fibrations are the surjections~\cite[Appendix~A]{MR2572248}.  
  General cofibrations are a bit far to characterize (beyond their definition in terms of the left lifting property), but it is known~\cite[Corollary~40]{MR2572248} that \define{quasifree} (di/pr)operads --- the ones that would be free if one were to forget their differentials --- are cofibrant, and this class of cofibrants suffices for most purposes.
  
  Abstract nonsense of model categories guarantees that if $\h_{1}P \cover P$ and $\h_{2}P \cover P$ are any two cofibrant replacements of the same (di/pr)operad $P$, then the space of maps $\h_{1}P \to \h_{2}P$ covering the identity on $P$ is contractible, and in particular $\h_{1}P$ and $\h_{2}P$ are quasiisomorphic. 

By definition, if $P$ and $Q$ are (di/pr)operads, the \define{space of maps $P \to Q$} is the simplicial set whose $k$-simplices are homomorphisms $P \to Q \otimes \bR[\Delta^{k}]$, where $\bR[\Delta^{k}] = \bR[t_{0},\dots,t_{k},\d t_{0},\dots,\d t_{k}] / \bigl\<\sum t_{i} = 1,\, \sum \d t_{i} = 0\bigr\>$ is the commutative dg algebra of polynomial de Rham forms on the $k$-dimensional simplex.  If $P$ is cofibrant, then this simplicial set satisfies the Kan horn-filling condition.  By convention, a space is \define{contractible} if it has the homotopy type of $\{*\}$, which in particular includes that it is nonempty.  Some further remarks about the space of maps between two (di/pr)operads are in the discussion of the basic facts of obstruction theory at the start of Section~\ref{section.mainthm1}.

There is a forgetful functor from properads to dioperads, whose left adjoint defines the \define{universal enveloping properad} of a dioperad.  The reader should be warned that these functors are known  to be not exact~\cite[Theorem~47]{MR2560406}.
For comparison, in \cite{MR2320654} it is shown that the forgetful functor from PROPs to properads and its adjoint constructing the universal enveloping PROP of a properad are exact.
\end{remark}

\begin{definition} \label{defn.action}
  For any $V \in \DGVect$, the (di/pr)operad $\End(V)$ satisfies $\End(V)(\mathbf m,\mathbf n) = \hom(V^{\otimes \mathbf m},V^{\otimes \mathbf n})$.  An \define{action} of a (di/pr)operad $P$ on $V$ is a homomorphism $P \to \End(V)$.  If $V$ is equipped with an action of $P$, then we will call $V$ a \define{$P$-algebra}.
  
  A \define{homotopy $P$-action on $V$} is an action on $V$ by any cofibrant replacement $\h P \cover P$.  By  Remark~\ref{remark.modelcategorystructure}, the choice of cofibrant replacement is irrelevant up to contractible spaces of choices.  Since the universal enveloping and forgetful functors between dioperads and properads are not exact, the notion of ``homotopy $P$-action'' depends on whether $P$ is treated as a properad or a dioperad.  On the other hand, since the universal enveloping and forgetful functors between properads and PROPs are exact,  properadic homotopy actions always extend to PROPic homotopy actions.
\end{definition}

We will focus our attention on actions of the $1$-shifted Frobenius (di/pr)operad:

\begin{definition}\label{defn.frob1}
  The  (di/pr)operad $\Frob_{1}$ of \define{open and coopen $1$-shifted commutative Frobenius algebras} satisfies:
  $$ \dim \Frob_{1}(m,n)^{{\bullet}} = \begin{cases} 1, & m,n > 0,\ (m,n)\neq (1,1), \text{ and } \bullet = n-1, \\ 0, & mn \leq 1 \text{ or } \bullet \neq n-1.\end{cases} $$
  The $\bS_{m}$ action on $\Frob_{1}(m,n)$ is trivial, whereas $\bS_{n}$ acts via the sign representation.  
  Let $e_{m,n}$ denote a basis element of $\Frob_{1}(m,n)^{n-1}$.  For degree reasons, the only non-zero compositions correspond to trees and are of the form $e_{m_{1},n_{1}}\otimes e_{m_{2},n_{2}} \mapsto (\#)e_{m_{1} + m_{2} - 1, n_{1} + n_{2} - 1}$ for some coefficients $(\#)$.  To set those numbers, we declare:
  $$   
  \begin{tikzpicture}[baseline=(basepoint),yscale=1.25,xscale=-1]
  \path(0,.25) coordinate (basepoint);
  \draw (-.5,.5) node[circle,draw,inner sep=3,fill=gray] (delta1) {};
  \draw (.5,0) node[circle,draw,inner sep=3,fill=gray] (delta2) {};
  \draw[onearrow]  (-.25,-.5) .. controls +(0,.25) and +(.25,-.5) .. (delta1);
  \draw[onearrow] (-.75,-.5) .. controls +(0,.25) and +(-.25,-.25) .. (delta1);
  \draw[onearrow]  (.25,-.5) .. controls +(0,.25) and +(-.25,-.25) .. (delta2);
  \draw[onearrow] (.75,-.5) .. controls +(0,.25) and +(.25,-.25) .. (delta2);
  \draw[onearrow] (delta1) .. controls +(-.25,.25) and +(0,-.25) .. (-.85,1.1);
  \draw[onearrow] (delta1) .. controls +(.25,.25) and +(0,-.25) .. (-.15,1.1);
  \draw[onearrow] (delta2) .. controls +(-.25,.5) and +(0,-.25) .. (.15,1.1);
  \draw[onearrow] (delta2) .. controls +(.25,.25) and +(0,-.25) .. (.85,1.1);
  \draw[] (-.5,1) node {$\scriptstyle \dots$} (.5,1) node {$\scriptstyle \dots$};
  \draw[] (-.5,-.35) node {$\scriptstyle \dots$} (.5,-.35) node {$\scriptstyle \dots$};
  \draw[onearrow] (delta2) -- (delta1);
    \draw[decorate,decoration={brace,mirror}] (-.9,1.15) -- node[auto,swap] {$\scriptstyle  m_{2}$} (-.1,1.15);
    \draw[decorate,decoration={brace,mirror}] (.1,1.15) -- node[auto,swap] {$\scriptstyle  m_{1}$} (.9,1.15);
    \draw[decorate,decoration={brace,mirror}] (-.2,-.55) -- node[auto,swap] {$\scriptstyle  n_{2}$} (-.8,-.55);
    \draw[decorate,decoration={brace,mirror}] (.8,-.55) -- node[auto,swap] {$\scriptstyle  n_{1}$} (.2,-.55);
\end{tikzpicture}
:
  e_{m_{1}+1,n_{1}} \otimes e_{m_{2},n_{2}+1} \mapsto e_{m_{1}+m_{2},n_{1}+n_{2}}$$
\end{definition}  
  Rather than writing the basis vectors $e_{m,n}$, it will be convenient to adopt a more graphical notation.  We set
  $$ e_{m,n} = \tikz[baseline=(basepoint)] {
    \path (0,0) coordinate (basepoint) (0,4pt) node[dot] {}  
    (0,-7pt) node{$\scriptstyle \dots$} (0,15pt) node{$\scriptstyle \dots$};
    \draw[onearrow] (-8pt,-8pt) -- (0,4pt);
    \draw[onearrow] (8pt,-8pt) -- (0,4pt);
    \draw[onearrow] (0,4pt) -- (-8pt,16pt);
    \draw[onearrow] (0,4pt) -- (8pt,16pt);
    \draw[decorate,decoration=brace] (-9pt,18pt) -- node[auto] {$\scriptstyle m$} (9pt,18pt);
    \draw[decorate,decoration=brace] (9pt,-10pt) -- node[auto] {$\scriptstyle n$} (-9pt,-10pt);
   }, $$
   whence the above multiplication rule becomes:
$$  \begin{tikzpicture}[baseline=(basepoint),yscale=1.25]
  \path(0,.1) coordinate (basepoint);
  \draw (.5,.5) node[dot] (delta1) {};
  \draw (-.5,0) node[dot] (delta2) {};
  \draw[onearrow]  (.25,-.5) .. controls +(0,.25) and +(-.25,-.5) .. (delta1);
  \draw[onearrow] (.75,-.5) .. controls +(0,.25) and +(.25,-.25) .. (delta1);
  \draw[onearrow]  (-.25,-.5) .. controls +(0,.25) and +(.25,-.25) .. (delta2);
  \draw[onearrow] (-.75,-.5) .. controls +(0,.25) and +(-.25,-.25) .. (delta2);
  \draw[onearrow] (delta1) .. controls +(.25,.25) and +(0,-.25) .. (.85,1.1);
  \draw[onearrow] (delta1) .. controls +(-.25,.25) and +(0,-.25) .. (.15,1.1);
  \draw[onearrow] (delta2) .. controls +(.25,.5) and +(0,-.25) .. (-.15,1.1);
  \draw[onearrow] (delta2) .. controls +(-.25,.25) and +(0,-.25) .. (-.85,1.1);
  \draw[] (-.45,1) node {$\scriptstyle \dots$} (.55,1) node {$\scriptstyle \dots$};
  \draw[] (-.45,-.35) node {$\scriptstyle \dots$} (.55,-.35) node {$\scriptstyle \dots$};
  \draw[onearrow] (delta2) -- (delta1);
    \draw[decorate,decoration=brace] (-.9,1.15) -- node[auto] {$\scriptstyle m_{1}$} (-.1,1.15);
    \draw[decorate,decoration=brace]  (.1,1.15) -- node[auto] {$\scriptstyle m_{2}$} (.9,1.15);
    \draw[decorate,decoration=brace] (-.2,-.55) -- node[auto] {$\scriptstyle n_{1}$} (-.8,-.55);
    \draw[decorate,decoration=brace]  (.8,-.55) -- node[auto] {$\scriptstyle n_{2}$} (.2,-.55);
 \end{tikzpicture}
   = 
   \tikz[baseline=(basepoint)] {
    \path (0,0) coordinate (basepoint) (0,4pt) node[dot] {}  
    (0,-7pt) node{$\scriptstyle \dots$} (0,15pt) node{$\scriptstyle \dots$};
    \draw[onearrow] (-8pt,-8pt) -- (0,4pt);
    \draw[onearrow] (8pt,-8pt) -- (0,4pt);
    \draw[onearrow] (0,4pt) -- (-8pt,16pt);
    \draw[onearrow] (0,4pt) -- (8pt,16pt);
    \draw[decorate,decoration=brace] (-9pt,18pt) -- node[auto] {$\scriptstyle m_{1}+m_{2}$} (9pt,18pt);
    \draw[decorate,decoration=brace] (9pt,-10pt) -- node[auto] {$\scriptstyle n_{1}+n_{2}$} (-9pt,-10pt);
   } 
     , 
     \quad\quad
     \begin{tikzpicture}[baseline=(basepoint),yscale=1.25]
  \path(0,.1) coordinate (basepoint);
  \draw (.5,.5) node[dot] (delta1) {};
  \draw (-.5,0) node[dot] (delta2) {};
  \draw[onearrow]  (.25,-.5) .. controls +(0,.25) and +(-.25,-.5) .. (delta1);
  \draw[onearrow] (.75,-.5) .. controls +(0,.25) and +(.25,-.25) .. (delta1);
  \draw[onearrow]  (-.25,-.5) .. controls +(0,.25) and +(.25,-.25) .. (delta2);
  \draw[onearrow] (-.75,-.5) .. controls +(0,.25) and +(-.25,-.25) .. (delta2);
  \draw[onearrow] (delta1) .. controls +(.25,.25) and +(0,-.25) .. (.85,1.1);
  \draw[onearrow] (delta1) .. controls +(-.25,.25) and +(0,-.25) .. (.15,1.1);
  \draw[onearrow] (delta2) .. controls +(.25,.5) and +(0,-.25) .. (-.15,1.1);
  \draw[onearrow] (delta2) .. controls +(-.25,.25) and +(0,-.25) .. (-.85,1.1);
  \draw[] (-.45,1) node {$\scriptstyle \dots$} (.55,1) node {$\scriptstyle \dots$};
  \draw[] (-.45,-.35) node {$\scriptstyle \dots$} (.55,-.35) node {$\scriptstyle \dots$};
  \draw[] (delta2) ..controls +(.5,.125) and +(-.25,-.25) .. (delta1);
  \draw[] (delta2) ..controls +(.25,.25) and +(-.5,-.125) .. (delta1);
  \draw[] (.1,.25) node[anchor=base] {$\scriptscriptstyle \cdots$};
    \draw[decorate,decoration=brace] (-.9,1.15) -- node[auto] {$\scriptstyle m_{1}$} (-.1,1.15);
    \draw[decorate,decoration=brace]  (.1,1.15) -- node[auto] {$\scriptstyle m_{2}$} (.9,1.15);
    \draw[decorate,decoration=brace] (-.2,-.55) -- node[auto] {$\scriptstyle n_{1}$} (-.8,-.55);
    \draw[decorate,decoration=brace]  (.8,-.55) -- node[auto] {$\scriptstyle n_{2}$} (.2,-.55);
    \draw[decorate,decoration=brace] (.25,.15) -- node[anchor=north] {$\scriptstyle  k$} (-.2,.35);
 \end{tikzpicture}
 = 0 \text{ if } k > 1.
$$

   Other compositions can be computed using the symmetric group actions.  Permutations of the incoming strands act trivially, but permutations of the outgoing strands lead to signs.  For example, tracking the actions of both $\bS_{n_{1}+1}$ on the top vertex and $\bS_{n_{1}+n_{2}}$ on the resulting composition gives:
   $$
 \begin{tikzpicture}[baseline=(basepoint),yscale=1.25]
  \path(0,.1) coordinate (basepoint);
  \draw (-.5,.5) node[dot] (delta1) {};
  \draw (.5,0) node[dot] (delta2) {};
  \draw[onearrow]  (-.25,-.5) .. controls +(0,.25) and +(.25,-.5) .. (delta1);
  \draw[onearrow] (-.75,-.5) .. controls +(0,.25) and +(-.25,-.25) .. (delta1);
  \draw[onearrow]  (.25,-.5) .. controls +(0,.25) and +(-.25,-.25) .. (delta2);
  \draw[onearrow] (.75,-.5) .. controls +(0,.25) and +(.25,-.25) .. (delta2);
  \draw[onearrow] (delta1) .. controls +(-.25,.25) and +(0,-.25) .. (-.85,1.1);
  \draw[onearrow] (delta1) .. controls +(.25,.25) and +(0,-.25) .. (-.15,1.1);
  \draw[onearrow] (delta2) .. controls +(-.25,.5) and +(0,-.25) .. (.15,1.1);
  \draw[onearrow] (delta2) .. controls +(.25,.25) and +(0,-.25) .. (.85,1.1);
  \draw[] (-.45,1) node {$\scriptstyle \dots$} (.55,1) node {$\scriptstyle \dots$};
  \draw[] (-.45,-.35) node {$\scriptstyle \dots$} (.55,-.35) node {$\scriptstyle \dots$};
  \draw[onearrow] (delta2) -- (delta1);
    \draw[decorate,decoration=brace] (-.9,1.15) -- node[auto] {$\scriptstyle m_{1}$} (-.1,1.15);
    \draw[decorate,decoration=brace] (.1,1.15) -- node[auto] {$\scriptstyle m_{2}$} (.9,1.15);
    \draw[decorate,decoration=brace] (-.2,-.55) -- node[auto] {$\scriptstyle n_{1}$} (-.8,-.55);
    \draw[decorate,decoration=brace] (.8,-.55) -- node[auto] {$\scriptstyle n_{2}$} (.2,-.55);
 \end{tikzpicture}
   = 
   (-1)^{n_{1}(n_{2}-1)}
   \tikz[baseline=(basepoint)] {
    \path (0,0) coordinate (basepoint) (0,4pt) node[dot] {}  
    (0,-7pt) node{$\scriptstyle \dots$} (0,15pt) node{$\scriptstyle \dots$};
    \draw[onearrow] (-8pt,-8pt) -- (0,4pt);
    \draw[onearrow] (8pt,-8pt) -- (0,4pt);
    \draw[onearrow] (0,4pt) -- (-8pt,16pt);
    \draw[onearrow] (0,4pt) -- (8pt,16pt);
    \draw[decorate,decoration=brace] (-9pt,18pt) -- node[auto] {$\scriptstyle m_{1}+m_{2}$} (9pt,18pt);
    \draw[decorate,decoration=brace] (9pt,-10pt) -- node[auto] {$\scriptstyle n_{1}+n_{2}$} (-9pt,-10pt);
   }   $$
   
   The reader should check that these signs satisfy the appropriate associativity from Definition~\ref{defn.diproperad}.  In the conventions from Remark~\ref{remark.signs}, signs occur whenever two vertices both with an even number of outputs switch heights.
   
\begin{example}
  Our goal is to study the Frobenius algebra structure on $\H^{\bullet}(S^{1})$, and so it is worth checking that the signs in Definition~\ref{defn.frob1} reproduce the usual ones.  The wedge multiplication is commutative and associative, and so setting $\mult = \wedge$ does work:
  $$   \,\tikz[baseline=(basepoint)]{ 
    \path (0,3pt) coordinate (basepoint) (0,1pt) node[dot] {} (-8pt,11pt) node[dot] {};
    \draw[](0,-9pt) -- (0,1pt);
    \draw[](0,1pt) -- (-8pt,11pt);
    \draw[](-8pt,11pt) -- (-12pt,21pt);
    \draw[](-8pt,11pt) -- (0pt,21pt);
    \draw[](0,1pt) -- (12pt,21pt);
  }\,
  =
  \,\tikz[baseline=(basepoint)]{ 
    \path (0,3pt) coordinate (basepoint) (0,6pt) node[dot] {} ;
    \draw[](0,-9pt) -- (0,6pt);
    \draw[](0,6pt) -- (-12pt,21pt);
    \draw[](0,6pt) -- (0pt,21pt);
    \draw[](0,6pt) -- (12pt,21pt);
  }\,
  =
  \,\tikz[baseline=(basepoint)]{ 
    \path (0,3pt) coordinate (basepoint) (0,1pt) node[dot] {} (8pt,11pt) node[dot] {};
    \draw[](0,-9pt) -- (0,1pt);
    \draw[](0,1pt) -- (8pt,11pt);
    \draw[](0,1pt) -- (-12pt,21pt);
    \draw[](8pt,11pt) -- (0pt,21pt);
    \draw[](8pt,11pt) -- (12pt,21pt);
  }\, : \H^{\bullet}(S^{1})^{\otimes 3} \to \H^{\bullet}(S^{1})$$
  
  What about comultiplication?  Let $1\in \H^{0}(S^{1})$ denote the monoidal unit and $\omega \in \H^{1}(S^{1})$ the class of any volume form of total volume $1$.  Then the comultiplication should satisfy $\Delta(\omega) = \omega \otimes \omega \in \H^{\bullet}(S^{1}) \otimes \H^{\bullet}(S^{1})$, and so
  $$ \Delta(\omega) \otimes \omega = \omega \otimes \omega \otimes \omega = \omega \otimes \Delta(\omega). $$
  The left-hand side is nothing but $\bigl((\Delta \otimes \id)\circ \Delta\bigr) (\omega)$.  But the right-hand side is $-\bigl( (\id \otimes \Delta) \circ \Delta\bigr)(\omega) = -(\id \otimes \Delta)(\omega \otimes\omega)$ in the conventions from Remark~\ref{remark.signs}, since the $\Delta$ must pass an $\omega$, and both are of odd degree.  This explains the sign in
  $$  
  \,\tikz[baseline=(basepoint)]{ 
    \path (0,3pt) coordinate (basepoint) (-8pt,1pt) node[dot] {} (0,11pt) node[dot] {};
    \draw[](12pt,-9pt) -- (0,11pt);
    \draw[](-12pt,-9pt) -- (-8pt,1pt);
    \draw[](0pt,-9pt) -- (-8pt,1pt);
    \draw[](-8pt,1pt) -- (0,11pt);
    \draw[](0,11pt) -- (0,21pt);
  }\,  =
  \,\tikz[baseline=(basepoint)]{ 
    \path (0,3pt) coordinate (basepoint) (0,6pt) node[dot] {};
    \draw[](-12pt,-9pt) -- (0,6pt);
    \draw[](0pt,-9pt) -- (0,6pt);
    \draw[](12pt,-9pt) -- (0,6pt);
    \draw[](0,6pt) -- (0,21pt);
  }\,
  = -
  \,\tikz[baseline=(basepoint)]{ 
    \path (0,3pt) coordinate (basepoint) (8pt,1pt) node[dot] {} (0,11pt) node[dot] {};
    \draw[](-12pt,-9pt) -- (0,11pt);
    \draw[](0pt,-9pt) -- (8pt,1pt);
    \draw[](12pt,-9pt) -- (8pt,1pt);
    \draw[](8pt,1pt) -- (0,11pt);
    \draw[](0,11pt) -- (0,21pt);
  }\,
.
  $$
  
  What about $\Delta(1)$?  There is no universal way to communicate signs: the map sending $\omega \mapsto \omega \otimes \omega$ and $1 \mapsto 1\otimes\omega + \omega \otimes 1$ is the same data as a map that instead sends $1 \mapsto \pm(1\otimes \omega - \omega \otimes 1)$, since one may always compose such a map by some sign factors that depend on the degrees of the inputs.  To maintain the coassociativity above, we declare $\Delta(1) = -1 \otimes \omega + \omega \otimes 1$, as then
  \begin{align*}
    \bigl( (\id \otimes \Delta) \circ \Delta\bigr)(1) & = (\id \otimes \Delta)(-1 \otimes \omega + \omega \otimes 1) = -1 \otimes \omega \otimes \omega - \omega \otimes (-1 \otimes \omega  + \omega \otimes 1), \\
    \bigl( (\Delta \otimes \id) \circ \Delta\bigr)(1) & = (\Delta \otimes \id) (-1 \otimes \omega + \omega \otimes 1) = -(-1 \otimes \omega + \omega \otimes 1) \otimes \omega + \omega \otimes \omega \otimes 1.
  \end{align*}
  
  Finally, let us check the Frobenius axiom.  We have claimed:
  $$   \,\tikz[baseline=(basepoint)]{ 
    \path (0,3pt) coordinate (basepoint) (0,1pt) node[dot] {} (0,11pt) node[dot] {};
    \draw[](-8pt,-9pt) -- (0,1pt);
    \draw[](8pt,-9pt) -- (0,1pt);
    \draw[](0,1pt) -- (0,11pt);
    \draw[](0,11pt) -- (-8pt,21pt);
    \draw[](0,11pt) -- (8pt,21pt);
  }\,
  =
  \,\tikz[baseline=(basepoint)]{ 
    \path (0,3pt) coordinate (basepoint) (0,6pt) node[dot] {} ;
    \draw[](-8pt,-9pt) -- (0,6pt);
    \draw[](8pt,-9pt) -- (0,6pt);
    \draw[](0,6pt) -- (-8pt,21pt);
    \draw[](0,6pt) -- (8pt,21pt);
  }\,
  =
  \,\tikz[baseline=(basepoint)]{ 
    \path (0,3pt) coordinate (basepoint) (8pt,1pt) node[dot] {} (-8pt,11pt) node[dot] {};
    \draw[](-8pt,-9pt) -- (-8pt,11pt);
    \draw[](8pt,-9pt) -- (8pt,1pt);
    \draw[](8pt,1pt) -- (-8pt,11pt);
    \draw[](-8pt,11pt) -- (-8pt,21pt);
    \draw[](8pt,1pt) -- (8pt,21pt);
  }\,.
  $$
  Both sides vanish when evaluated at $\omega \otimes \omega$ for degree reasons.  For the other possible inputs, the left-hand side is:
   $$ \begin{array} {ccccl}
  1\otimes 1 & \overset\wedge\mapsto &1& \overset\Delta\mapsto & -1 \otimes \omega  + \omega \otimes 1, \\
   1\otimes \omega & \overset\wedge\mapsto &\omega& \overset\Delta\mapsto& \omega \otimes \omega, \\
   \omega \otimes 1 & \overset\wedge\mapsto & \omega& \overset\Delta\mapsto& \omega \otimes \omega .\end{array}$$
  The right-hand side is $-(\id \otimes \wedge) \circ (\Delta \otimes \id)$, which gives:
  $$ \begin{array}{ccccl}
    1\otimes 1 & \overset{\Delta \otimes \id}\mapsto & (-1 \otimes \omega  + \omega \otimes 1)\otimes 1 & \overset{\id \otimes \wedge} \mapsto & -1 \otimes \omega + \omega \otimes 1, \\
    1 \otimes \omega & \overset{\Delta \otimes \id}\mapsto & (-1 \otimes \omega  + \omega \otimes 1)\otimes \omega & \overset{\id \otimes \wedge} \mapsto & 0 + \omega \otimes \omega, \\
    \omega \otimes 1 & \overset{\Delta \otimes \id}\mapsto & \omega \otimes \omega \otimes 1 & \overset{\id \otimes \wedge} \mapsto & \omega \otimes \omega.
  \end{array} $$
  
\end{example}

\begin{remark}\label{remark.openunit}
  The words ``open'' and ``coopen'' in Definition~\ref{defn.frob1} denote that our Frobenius algebras need not have units or counits.  To restore these, one must modify Definition~\ref{defn.diproperad} to require unital (di/pr)operads, and then in Definition~\ref{defn.frob1} simply set $\Frob_{1}(m,n)^{n-1} = \bR$ for all $m,n$.  These changes require more complicated versions of Koszulity and the (co)bar construction than we present in Section~\ref{section.Koszul}; the interested reader may consult~\cite{MR2993002} for details.
\end{remark}

The other (di/pr)operad that will play a role in this paper is $\qloc$ from Definition~\ref{defn.quasilocal}:

\begin{lemma}
  The cochain complexes $\qloc(S^{1})(-,-)$ from Definition~\ref{defn.quasilocal} are naturally a subproperad (and hence a subdioperad) of $\Omega^{\bullet}\bigl(\bR_{>0}; \End(\Omega^{\bullet}(S^{1}))\bigr)$.
\end{lemma}

\begin{proof}
   Note that for any dg algebraic object --- algebra, operad, properad --- $P$ and any manifold~$Y$, $\Omega^{\bullet}(Y;P)$ is a dg algebraic object of the same type as $P$: composition in $\Omega^{\bullet}(Y;P)$ is a combination of wedge product in $\Omega^{\bullet}(Y)$ and composition in $P$.  
   In particular, $\Omega^{\bullet}\bigl(\bR_{>0}; \End(\Omega^{\bullet}(S^{1}))\bigr)$ really is a properad.  Moreover,
   the complexes $\qloc(S^{1})(-,-)$ are clearly a sub-$\bS$-bimodule of $\Omega^{\bullet}\bigl(\bR_{>0}; \End(\Omega^{\bullet}(S^{1}))\bigr)$.  
   
   It thus suffices to check that $\qloc(S^{1})$ is closed under properadic composition.  This follows from the triangle inequality along with the following characterization of when an element $f\in \Omega^{\bullet}\bigl(\bR_{>0}; \End(\Omega^{\bullet}(S^{1}))\bigr)$ is in $\qloc(S^{1})(-,-)$: for each $\ell > 0$, there exists $\epsilon>0$ such that $f|_{(0,\epsilon)}$ takes values in those operations supported in an $\ell$-radius neighborhood of $\diag(S^{1})$.
\end{proof}

\begin{remark}
  It is important that we are using properadic composition, and not the more general ``disconnected'' compositions allowed in PROPs, as for those we would have no triangle to call on.  Indeed, $\End(\Omega^{\bullet}(S^{1}))$ is a PROP, but $\qloc(S^{1})$ is not a sub-PROP, only a sub-properad.
\end{remark}

Note that the inclusion $\End(\Omega^{\bullet}(S^{1})) \to \Omega^{\bullet}\bigl(\bR_{>0}; \End(\Omega^{\bullet}(S^{1}))\bigr)$ as the constant $0$-forms is a quasiisomorphism.  It has as a homotopy inverse the map $\Omega^{\bullet}\bigl(\bR_{>0}; \End(\Omega^{\bullet}(S^{1}))\bigr) \overset\sim\to \End(\Omega^{\bullet}(S^{1}))$ that ignores any $1$-form data and evaluates each $0$-form at $1\in \bR_{>0}$.  Thus up to homotopy of (di/pr)operads, we may identify $\End(\Omega^{\bullet}(S^{1}))$ and $\Omega^{\bullet}\bigl(\bR_{>0}; \End(\Omega^{\bullet}(S^{1}))\bigr)$.  We can then clarify Definition~\ref{defn.quasilocallift}:

\begin{definition}\label{defn.quasilocallift.redux}
  A \define{quasilocal action} of a cofibrant (di/pr)operad $P$ on $\Omega^{\bullet}(S^{1})$ is  
  an action of $P$ on $\Omega^{\bullet}(S^{1})$
  equipped with a factorization
  through a homomorphism $$P \to \qloc(S^{1}) \mono \Omega^{\bullet}\bigl(\bR_{>0}; \End(\Omega^{\bullet}(S^{1}))\bigr) \overset\sim\to \End(\Omega^{\bullet}(S^{1})).$$
\end{definition}

\begin{proposition} \label{prop.Hqloc}
  The space of quasilocal operations has the following cohomology:
  $$ \dim \H^{\bullet}\qloc(m,n) = \begin{cases} 1 & \bullet = n-1 \text{ or } n \\ 0 & \text{otherwise.} \end{cases} $$
\end{proposition}
\begin{proof}
Recall (Definition~\ref{defn.quasilocal}) that $\diag: S^{1} \to (S^{1})^{\times(m+n)}$ denotes the diagonal embedding.

  Let $\ell : \bR_{>0} \to \bR_{>0}$ be a smooth monotonic function with $\lim_{\epsilon \to 0}\ell(\epsilon) = 0$.  Let $\qloc_{\ell}(m,n) \subseteq \Omega^{\bullet}\bigl(\bR_{>0}; \End(\Omega^{\bullet}(S^{1}))(m,n)\bigr) \subseteq \Omega^{\bullet}_{\mathrm{dist}}\bigl(\bR_{>0} \times (S^{1})^{\times (m+n)} \bigr)[m]$ denote the subcomplex whose integral kernels are supported in the set 
  $$ U_{\ell} = \bigl\{ (\epsilon,\vec x) \in \bR \times (S^{1})^{\times (m+n)} : \text{$\vec x$ is within  distance $\ell(\epsilon)$ of $\diag(S^{1})$} \bigr\}. $$
 For a fixed function $\ell$, the $\bS$-bimodules $\qloc_{\ell}(m,n)$ do not package together into a properad.  But the set of such functions $\ell$ is partially ordered by $\leq$, and $\qloc(m,n)$ is the filtered colimit 
  $$ \qloc(m,n) = \bigcup_{\ell} \qloc_{\ell}(m,n). $$
  Since for $\ell' \leq \ell$ the map $\qloc_{\ell'}(m,n) \to \qloc_{\ell}(m,n)$ is an inclusion, and since filtered colimits of cochain complexes along inclusions are homotopy filtered colimits, we see that
  $$ \H^{\bullet}\qloc(m,n) = \operatorname{hocolim}_{\ell} \limits \H^{\bullet}\qloc_{\ell}(m,n). $$
  It thus suffices to show that $\H^{\bullet}\qloc_{\ell}(m,n) \cong \H^{\bullet}(S^{1})[1-n]$ for each $\ell$ and that the inclusions $\qloc_{\ell'}(m,n) \hookrightarrow \qloc_{\ell}(m,n)$ are quasiisomorphisms.
  
  Fix the function $\ell$.  Then the set $U_{\ell} \subseteq \bR_{>0} \times (S^{1})^{\times(m+n)}$ looks like
  $$ U_{\ell} \quad = \quad
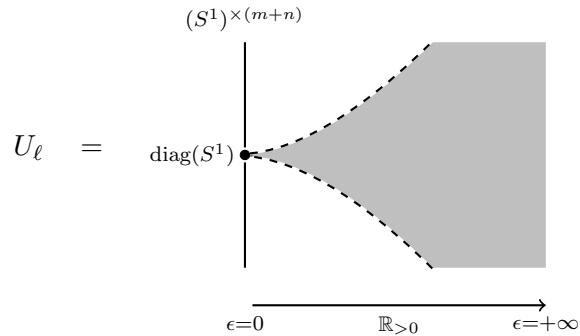
  \begin{tikzpicture}[baseline=(diag)]
    \path (0,0) node[dot] (diag) {} node[anchor=east] {$\scriptstyle \diag(S^{1})$} ;
    \draw[thick] (0,.1) -- (0,1.5) node[anchor=south] {$\scriptstyle (S^{1})^{\times(m+n)}$} ;
    \draw[thick] (0,-.1) -- (0,-1.5);
    \draw[thick,->] (0,-2) node[anchor=north] {$\scriptstyle \epsilon = 0$} (.1,-2) -- node[auto,swap] {$\scriptstyle \bR_{>0}$} (4,-2) node[anchor=north] {$\scriptstyle \epsilon = +\infty$};
    \draw[ultra thick,dashed] (diag) .. controls +(1,.1) and +(-.5,-.5) .. (2.5,1.5);
    \draw[ultra thick,dashed] (2.5,-1.5) .. controls  +(-.5,.5) and +(1,-.1) .. (diag);
    \path[fill=black!25] (diag) .. controls +(1,.1) and +(-.5,-.5) .. (2.5,1.5) -- (4,1.5) -- (4,-1.5) -- (2.5,-1.5) .. controls  +(-.5,.5) and +(1,-.1) .. (diag);
  \end{tikzpicture}
  $$
  where the dashed lines are the boundary of the neighborhood of $\diag(S^{1})$ of radius $\ell(\epsilon)$.  Recall from Remark~\ref{remark.nuclearTVS} that $\End(\Omega^{\bullet}\bigr(S^{1})\bigl)(m,n) \cong \Omega^{\bullet}\bigl((S^{1})^{\times(m+n)}\bigr)[m]$ with some regularity conditions that don't affect cohomology calculations.  It follows that
  \begin{equation}
  \label{eqn.a shift by m for qloc-l}
   \qloc_{\ell}(m,n)  \cong \Omega^{\bullet}\left(
  \begin{tikzpicture}[baseline=(diag)]
    \path (0,0) node[dot] (diag) {} node[anchor=east] {$\scriptstyle \diag(S^{1})$} ;
    \draw[thick] (0,.1) -- (0,1.5) node[anchor=south] {$\scriptstyle (S^{1})^{\times(m+n)}$} ;
    \draw[thick] (0,-.1) -- (0,-1.5);
    \draw[thick,->] (0,-2) node[anchor=north] {$\scriptstyle \epsilon = 0$} (.1,-2) -- node[auto,swap] {$\scriptstyle \bR_{>0}$} (4,-2) node[anchor=north] {$\scriptstyle \epsilon = +\infty$};
    \draw[ultra thick,dashed] (diag) .. controls +(1,.1) and +(-.5,-.5) .. (2.5,1.5);
    \draw[ultra thick,dashed] (2.5,-1.5) .. controls  +(-.5,.5) and +(1,-.1) .. (diag);
    \path[fill=black!25] (diag) .. controls +(1,.1) and +(-.5,-.5) .. (2.5,1.5) -- (4,1.5) -- (4,-1.5) -- (2.5,-1.5) .. controls  +(-.5,.5) and +(1,-.1) .. (diag);
  \end{tikzpicture}
  \right)[m],\end{equation}
  where the picture means the complex of de Rham forms on the entire $\bR_{>0} \times (S^{1})^{\times(m+n)}$ that vanish in the white regions.  
  
  The complex of forms on $\bR_{>0}$ that vanish below some cutoff is exact:
  $$ \Omega^{\bullet}\left( 
  \begin{tikzpicture}[baseline=(basing)]
    \draw[thick,dotted] (0,-3pt) coordinate (basing) (0,0) node[anchor=north] {$\scriptstyle \epsilon = 0$} (.1,0) -- 
     node[auto,swap,pos=.75] {$\scriptstyle \bR_{>0}$}
     (2.55,0);
    \draw[thick,->] (2.55,0) --
      (4,0) node[anchor=north] {$\scriptstyle \epsilon = +\infty$};  
    \draw[very thick,dashed] (2.5,-.25) -- (2.6,.25);
  \end{tikzpicture}
  \right) \simeq 0.$$
  We may witness this exactness by an explicit homotopy, depending smoothly on the location of the cutoff.  Applying this outside of a tubular neighborhood of $\diag(S^{1})$ provides an explicit deformation retraction of complexes witnessing the following quasiisomorphism:
  $$ \Omega^{\bullet}\left(
  \begin{tikzpicture}[baseline=(diag)]
    \path (0,0) node[dot] (diag) {} node[anchor=east] {$\scriptstyle \diag(S^{1})$} ;
    \draw[thick] (0,.1) -- (0,1.5) node[anchor=south] {$\scriptstyle (S^{1})^{\times(m+n)}$} ;
    \draw[thick] (0,-.1) -- (0,-1.5);
    \draw[thick,->] (0,-2) node[anchor=north] {$\scriptstyle \epsilon = 0$} (.1,-2) -- node[auto,swap] {$\scriptstyle \bR_{>0}$} (4,-2) node[anchor=north] {$\scriptstyle \epsilon = +\infty$};
    \draw[ultra thick,dashed] (diag) .. controls +(1,.1) and +(-.5,-.5) .. (2.5,1.5);
    \draw[ultra thick,dashed] (2.5,-1.5) .. controls  +(-.5,.5) and +(1,-.1) .. (diag);
    \path[fill=black!25] (diag) .. controls +(1,.1) and +(-.5,-.5) .. (2.5,1.5) -- (4,1.5) -- (4,-1.5) -- (2.5,-1.5) .. controls  +(-.5,.5) and +(1,-.1) .. (diag);
  \end{tikzpicture}
  \right)
  \quad \simeq \quad
  \Omega^{\bullet}\left(
  \begin{tikzpicture}[baseline=(diag)]
    \path (0,0) node[dot] (diag) {} node[anchor=east] {$\scriptstyle \diag(S^{1})$} ;
    \draw[thick] (0,.1) -- (0,1.5) node[anchor=south] {$\scriptstyle (S^{1})^{\times(m+n)}$} ;
    \draw[thick] (0,-.1) -- (0,-1.5);
    \draw[thick,->] (0,-2) node[anchor=north] {$\scriptstyle \epsilon = 0$} (.1,-2) -- node[auto,swap] {$\scriptstyle \bR_{>0}$} (4,-2) node[anchor=north] {$\scriptstyle \epsilon = +\infty$};
    \draw[ultra thick,dashed] (diag) .. controls +(1,.1) and +(-2,-0) .. (4,.75);
    \draw[ultra thick,dashed] (4,-.75) .. controls  +(-2,0) and +(1,-.1) .. (diag);
    \path[fill=black!25] (diag) .. controls +(1,.1) and +(-2,-0) .. (4,.75) -- (4,-.75) .. controls  +(-2,0) and +(1,-.1) .. (diag);
  \end{tikzpicture}
  \right).
  $$
  But the latter is just the complex of forms on a tubular neighborhood of $\bR_{>0} \times \diag(S^{1}) \hookrightarrow \bR_{>0} \times (S^{1})^{\times(m+n)}$ that have compact support in the fibers of the tubular neighborhood.  Such a tubular neighborhood has codimension $m+n-1$, and so
  $$
  \Omega^{\bullet}\left(
  \begin{tikzpicture}[baseline=(diag)]
    \path (0,0) node[dot] (diag) {} node[anchor=east] {$\scriptstyle \diag(S^{1})$} ;
    \draw[thick] (0,.1) -- (0,1.5) node[anchor=south] {$\scriptstyle (S^{1})^{\times(m+n)}$} ;
    \draw[thick] (0,-.1) -- (0,-1.5);
    \draw[thick,->] (0,-2) node[anchor=north] {$\scriptstyle \epsilon = 0$} (.1,-2) -- node[auto,swap] {$\scriptstyle \bR_{>0}$} (4,-2) node[anchor=north] {$\scriptstyle \epsilon = +\infty$};
    \draw[ultra thick,dashed] (diag) .. controls +(1,.1) and +(-2,-0) .. (4,.75);
    \draw[ultra thick,dashed] (4,-.75) .. controls  +(-2,0) and +(1,-.1) .. (diag);
    \path[fill=black!25] (diag) .. controls +(1,.1) and +(-2,-0) .. (4,.75) -- (4,-.75) .. controls  +(-2,0) and +(1,-.1) .. (diag);
  \end{tikzpicture}
  \right)
  \simeq
  \Omega^{\bullet}\bigl(\bR_{>0} \times \diag(S^{1})\bigr)[1-m-n]
  \simeq
  \Omega^{\bullet}(S^{1})[1-m-n]
  .
  $$
  Combining with the shift by $[m]$ from equation~(\ref{eqn.a shift by m for qloc-l}) gives the desired answer.
  
  Finally, note that if $\ell' \leq \ell$, then applying the homotopy above for both $U_{\ell'}$ and $U_{\ell}$ produces an inclusion of the form
  $$ 
  \Omega^{\bullet}\left(
  \begin{tikzpicture}[baseline=(diag)]
    \path (0,0) node[dot] (diag) {} node[anchor=east] {$\scriptstyle \diag(S^{1})$} ;
    \draw[thick] (0,.1) -- (0,1.5) node[anchor=south] {$\scriptstyle (S^{1})^{\times(m+n)}$} ;
    \draw[thick] (0,-.1) -- (0,-1.5);
    \draw[thick,->] (0,-2) node[anchor=north] {$\scriptstyle \epsilon = 0$} (.1,-2) -- node[auto,swap] {$\scriptstyle \bR_{>0}$} (4,-2) node[anchor=north] {$\scriptstyle \epsilon = +\infty$};
    \draw[ultra thick,dashed] (diag) .. controls +(2,.1) and +(-2,-0) .. (4,.25);
    \draw[ultra thick,dashed] (4,-.25) .. controls  +(-2,0) and +(2,-.1) .. (diag);
    \path[fill=black!25] (diag) .. controls +(2,.1) and +(-2,-0) .. (4,.25) -- (4,-.25) .. controls  +(-2,0) and +(2,-.1) .. (diag);
  \end{tikzpicture}
  \right)
  \quad \hookrightarrow \quad
  \Omega^{\bullet}\left(
  \begin{tikzpicture}[baseline=(diag)]
    \path (0,0) node[dot] (diag) {} node[anchor=east] {$\scriptstyle \diag(S^{1})$} ;
    \draw[thick] (0,.1) -- (0,1.5) node[anchor=south] {$\scriptstyle (S^{1})^{\times(m+n)}$} ;
    \draw[thick] (0,-.1) -- (0,-1.5);
    \draw[thick,->] (0,-2) node[anchor=north] {$\scriptstyle \epsilon = 0$} (.1,-2) -- node[auto,swap] {$\scriptstyle \bR_{>0}$} (4,-2) node[anchor=north] {$\scriptstyle \epsilon = +\infty$};
    \draw[ultra thick,dashed] (diag) .. controls +(1,.1) and +(-2,-0) .. (4,.75);
    \draw[ultra thick,dashed] (4,-.75) .. controls  +(-2,0) and +(1,-.1) .. (diag);
    \path[fill=black!25] (diag) .. controls +(1,.1) and +(-2,-0) .. (4,.75) -- (4,-.75) .. controls  +(-2,0) and +(1,-.1) .. (diag);
  \end{tikzpicture}
  \right).
  $$
But this is just an inclusion of one tubular neighborhood around $\bR_{>0} \times \diag(S^{1})$ into another, and hence a quasiisomorphism.
\end{proof}

\section{Koszulity of \texorpdfstring{$\Frob_{1}$}{Frob_1}} \label{section.Koszul}

As far as this paper is concerned, the \emph{raison d'\^etre} of Koszul duality theory is to provide small cofibrant replacements of objects of interest.
We will briefly recall enough of the theory for our purposes.

  For any $\bS$-bimodule $T$,  denote by $\Ff(T)$ the free (di/pr)operad generated by $T$.  (This is not much of an abuse of notation, as the universal enveloping properad of the free dioperad generated by $T$ is the free properad generated by $T$.)  Note that $\Ff(T)$ is also a co(di/pr)operad.  We let $\Ff^{(k)}(T)$ denote the sub-$\bS$-bimodule of $\Ff(T)$ that transforms with weight $k$ under the canonical $\bR^{\times}$-action on $T$. In particular, $\Ff^{(1)}(T) = T$, and $\Ff(T) = \bigoplus_{k\geq 1}\Ff^{(k)}(T)$.  
    
\begin{definition} \label{defn.quaddual}
  A \define{quadratic} (di/pr)operad is a (di/pr)operad presented as $P = \Ff(T) / \<R\>$, where $T$ is an $\bS$-bimodule and $R \subseteq \Ff^{(2)}(T)$ is a sub-$\bS$-bimodule generating the ideal $\<R\>$.
  
  The \define{quadratic dual} $P^{\shriek}$ of a quadratic (di/pr)operad $P$ is the maximal graded sub-co(di/pr)operad of $\Ff(T[1])$ whose intersection with $\Ff^{(2)}(T[1])$ is precisely $R[2]\subseteq \Ff^{(2)}(T[1])$.
\end{definition}

Note that Remark~\ref{remark.signs} introduces some signs when comparing homogeneous elements of~$R$ and~$R[2]$.  Indeed, suppose that for $t_{1},t_{2} \in T$, some particular composition $  \begin{tikzpicture}[baseline=(basepoint),yscale=.5,xscale=-.5]
  \path(0,.25) coordinate (basepoint);
  \draw (-.5,.5) node[circle,draw,inner sep=2,fill=gray] (delta1) {};
  \draw (.5,.1) node[circle,draw,inner sep=2,fill=gray] (delta2) {};
  \draw[onearrow]  (-.25,-.5) -- (delta1);
  \draw[onearrow] (-.75,-.5) -- (delta1);
  \draw[onearrow]  (.25,-.5) -- (delta2);
  \draw[onearrow] (.75,-.5) -- (delta2);
  \draw[onearrow] (delta1) -- (-.85,1.1);
  \draw[onearrow] (delta1) -- (-.15,1.1);
  \draw[onearrow] (delta2) -- (.15,1.1);
  \draw[onearrow] (delta2) -- (.85,1.1);
  \draw[]  (.5,1) node {$\scriptstyle \dots$} (-.5,-.35) node {$\scriptstyle \dots$} ;
  \draw[] (delta2) ..controls +(-.5,.125) and +(.25,-.25) .. (delta1);
  \draw[] (delta2) ..controls +(-.25,.25) and +(.5,-.125) .. (delta1);
  \draw[] (-.05,.25) node[anchor=base] {$\scriptscriptstyle \cdots$};
\end{tikzpicture}
(t_{1},t_{2}) \in R \vspace{-6pt}$.  Then $R[2]$ contains $  \begin{tikzpicture}[baseline=(basepoint),yscale=.5,xscale=-.5]
  \path(0,.25) coordinate (basepoint);
  \draw (-.5,.5) node[circle,draw,inner sep=2,fill=gray] (delta1) {};
  \draw (.5,.1) node[circle,draw,inner sep=2,fill=gray] (delta2) {};
  \draw[onearrow]  (-.25,-.5) -- (delta1);
  \draw[onearrow] (-.75,-.5) -- (delta1);
  \draw[onearrow]  (.25,-.5) -- (delta2);
  \draw[onearrow] (.75,-.5) -- (delta2);
  \draw[onearrow] (delta1) -- (-.85,1.1);
  \draw[onearrow] (delta1) -- (-.15,1.1);
  \draw[onearrow] (delta2) -- (.15,1.1);
  \draw[onearrow] (delta2) -- (.85,1.1);
  \draw[]  (.5,1) node {$\scriptstyle \dots$} (-.5,-.35) node {$\scriptstyle \dots$} ;
  \draw[] (delta2) ..controls +(-.5,.125) and +(.25,-.25) .. (delta1);
  \draw[] (delta2) ..controls +(-.25,.25) and +(.5,-.125) .. (delta1);
  \draw[] (-.05,.25) node[anchor=base] {$\scriptscriptstyle \cdots$};
\end{tikzpicture}
(t_{1},t_{2})\iota^{2} = (-1)^{\deg t_{2}}  \begin{tikzpicture}[baseline=(basepoint),yscale=.5,xscale=-.5]
  \path(0,.25) coordinate (basepoint);
  \draw (-.5,.5) node[circle,draw,inner sep=2,fill=gray] (delta1) {};
  \draw (.5,.1) node[circle,draw,inner sep=2,fill=gray] (delta2) {};
  \draw[onearrow]  (-.25,-.5) -- (delta1);
  \draw[onearrow] (-.75,-.5) -- (delta1);
  \draw[onearrow]  (.25,-.5) -- (delta2);
  \draw[onearrow] (.75,-.5) -- (delta2);
  \draw[onearrow] (delta1) -- (-.85,1.1);
  \draw[onearrow] (delta1) -- (-.15,1.1);
  \draw[onearrow] (delta2) -- (.15,1.1);
  \draw[onearrow] (delta2) -- (.85,1.1);
  \draw[]  (.5,1) node {$\scriptstyle \dots$} (-.5,-.35) node {$\scriptstyle \dots$} ;
  \draw[] (delta2) ..controls +(-.5,.125) and +(.25,-.25) .. (delta1);
  \draw[] (delta2) ..controls +(-.25,.25) and +(.5,-.125) .. (delta1);
  \draw[] (-.05,.25) node[anchor=base] {$\scriptscriptstyle \cdots$};
\end{tikzpicture}
(t_{1}\iota,t_{2}\iota)$.  Note also that we have decided to draw $t_{2}$ at the top.  This explains some signs that will arise later on.

\begin{definition}
  Let $Q$ be any co(di/pr)operad, such that for each $m_{1},m_{2},n_{1},n_{2}$, there are only finitely many $k\in \bN$ for which the decomposition map $ Q\bigl( m_{1} \sqcup  m_{2},  n_{1} \sqcup  n_{2} \bigr) \to Q\bigl( m_{1}, k \sqcup  n_{1} \bigr) \otimes Q\bigl( m_{2} \sqcup  k, n_{2}\bigr) $ is nonzero.
  The \define{cobar construction} applied to $Q$ produces the (di/pr)operad $\B Q = \Ff(Q[-1])$, with the differential extending the degree-$1$ map $\sum (\text{decompositions}) : Q[-1] \to \Ff^{(2)}(Q[-1])$.  Coassiciativity is equivalent to the differential squaring to $0$.
\end{definition}

\begin{remark}\label{remark.signs2}
Given $q\in Q$, we will draw the corresponding generator $q\iota^{-1} \in Q[-1] \subseteq \B Q$ with a box, like $\tikz[baseline=(q.base)] \node[draw] (q) {$q$};$.  A clean convention is to declare that
 if the decomposition $\Delta(q)$ contains $  \begin{tikzpicture}[baseline=(basepoint),yscale=.5,xscale=-.5]
  \path(0,.25) coordinate (basepoint);
  \draw (-.5,.5) node[circle,draw,inner sep=2,fill=gray] (delta1) {};
  \draw (.5,.1) node[circle,draw,inner sep=2,fill=gray] (delta2) {};
  \draw[onearrow]  (-.25,-.5) -- (delta1);
  \draw[onearrow] (-.75,-.5) -- (delta1);
  \draw[onearrow]  (.25,-.5) -- (delta2);
  \draw[onearrow] (.75,-.5) -- (delta2);
  \draw[onearrow] (delta1) -- (-.85,1.1);
  \draw[onearrow] (delta1) -- (-.15,1.1);
  \draw[onearrow] (delta2) -- (.15,1.1);
  \draw[onearrow] (delta2) -- (.85,1.1);
  \draw[]  (.5,1) node {$\scriptstyle \dots$} (-.5,-.35) node {$\scriptstyle \dots$} ;
  \draw[] (delta2) ..controls +(-.5,.125) and +(.25,-.25) .. (delta1);
  \draw[] (delta2) ..controls +(-.25,.25) and +(.5,-.125) .. (delta1);
  \draw[] (-.05,.25) node[anchor=base] {$\scriptscriptstyle \cdots$};
\end{tikzpicture}
(q_{1},q_{2})$, then $\partial (q\iota^{-1})$ contains $(-1)^{\deg q_{1}}  \begin{tikzpicture}[baseline=(basepoint),yscale=.5,xscale=-.5]
  \path(0,.25) coordinate (basepoint);
  \draw (-.5,.5) node[circle,draw,inner sep=2,fill=gray] (delta1) {};
  \draw (.5,.1) node[circle,draw,inner sep=2,fill=gray] (delta2) {};
  \draw[onearrow]  (-.25,-.5) -- (delta1);
  \draw[onearrow] (-.75,-.5) -- (delta1);
  \draw[onearrow]  (.25,-.5) -- (delta2);
  \draw[onearrow] (.75,-.5) -- (delta2);
  \draw[onearrow] (delta1) -- (-.85,1.1);
  \draw[onearrow] (delta1) -- (-.15,1.1);
  \draw[onearrow] (delta2) -- (.15,1.1);
  \draw[onearrow] (delta2) -- (.85,1.1);
  \draw[]  (.5,1) node {$\scriptstyle \dots$} (-.5,-.35) node {$\scriptstyle \dots$} ;
  \draw[] (delta2) ..controls +(-.5,.125) and +(.25,-.25) .. (delta1);
  \draw[] (delta2) ..controls +(-.25,.25) and +(.5,-.125) .. (delta1);
  \draw[] (-.05,.25) node[anchor=base] {$\scriptscriptstyle \cdots$};
\end{tikzpicture}
(q_{1}\iota^{-1},q_{2}\iota^{-1})$, or:\vspace{-12pt}
$$
 \partial\tikz[baseline=(q.base)] \node[draw] (q) {$q$}; \ni 
 (-1)^{\deg q_{1}}
   \begin{tikzpicture}[baseline=(basepoint),yscale=1.25,xscale=-1]
  \path(0,.25) coordinate (basepoint);
  \draw (-.5,.5) node[draw] (delta1) {$q_{2}$};
  \draw (.5,0) node[draw] (delta2) {$q_{1}$};
  \draw[onearrow]  (-.25,-.5) .. controls +(0,.25) and +(.25,-.5) .. (delta1);
  \draw[onearrow] (-.75,-.5) .. controls +(0,.25) and +(-.25,-.25) .. (delta1);
  \draw[onearrow]  (.25,-.5) .. controls +(0,.25) and +(-.25,-.25) .. (delta2);
  \draw[onearrow] (.75,-.5) .. controls +(0,.25) and +(.25,-.25) .. (delta2);
  \draw[onearrow] (delta1) .. controls +(-.25,.25) and +(0,-.25) .. (-.85,1.1);
  \draw[onearrow] (delta1) .. controls +(.25,.25) and +(0,-.25) .. (-.15,1.1);
  \draw[onearrow] (delta2) .. controls +(-.25,.5) and +(0,-.25) .. (.15,1.1);
  \draw[onearrow] (delta2) .. controls +(.25,.25) and +(0,-.25) .. (.85,1.1);
  \draw[] (-.5,1) node {$\scriptstyle \dots$} (.5,1) node {$\scriptstyle \dots$};
  \draw[] (-.5,-.35) node {$\scriptstyle \dots$} (.5,-.35) node {$\scriptstyle \dots$};
  \draw[] (delta2) ..controls +(-.5,.125) and +(.25,-.25) .. (delta1);
  \draw[] (delta2) ..controls +(-.25,.25) and +(.5,-.125) .. (delta1);
  \draw[] (-.05,.25) node[anchor=base] {$\scriptscriptstyle \cdots$};
\end{tikzpicture}
$$
Note that the sign corresponds to the unboxed degree of the \emph{bottom} vertex.

To confirm this choice, we can check that $\partial^2 = 0$.  Suppose that  $\Delta^2(q)$ has a summand of the following form:
$$ \Delta^2(q) \ni   
  \;\begin{tikzpicture}[baseline=(q2.base),xscale=-1,yscale=1.5]
    \path (0,0) node (q1) {$q_{1}$};
    \path (.7,.7) node (q2) {$q_{2}$};
    \path (-.7,1.4) node (q3) {$q_{3}$};
    \draw[tensor,onearrow] (q1) -- (q2);
    \draw[tensor,onearrow] (q1) -- (q3);
    \draw[tensor,onearrow] (q2) -- (q3);
  \end{tikzpicture}\; $$
  Then
  $$ \partial\tikz[baseline=(q.base)] \node[draw] (q) {$q$}; \ni
   (-1)^{\deg q_1}
  \;\begin{tikzpicture}[baseline=(q2.base),xscale=-1,yscale=1.5]
    \path (0,0) node[draw] (q1) {$q_{1}$};
    \path (.7,.7) node (q2) {$q_{2}$};
    \path (-.7,1.4) node (q3) {$q_{3}$};
    \draw[tensor,onearrow] (q2) -- (q3);
    \draw (q2.south west) --  coordinate (q23)  (q3.north east |- q2.south west) -- (q3.north east) --  (q3.north east -| q2.south west) -- cycle;
    \draw[tensor,onearrow] (q1) -- (q23);
  \end{tikzpicture}\;
  + (-1)^{\deg q_1+ \deg q_2}
  \;\begin{tikzpicture}[baseline=(q2.base),xscale=-1,yscale=1.5]
    \path (0,0) node (q1) {$q_{1}$};
    \path (.7,.7) node (q2) {$q_{2}$};
    \path (-.7,1.4) node[draw] (q3) {$q_{3}$};
    \draw[tensor,onearrow] (q1) -- (q2);
    \draw (q2.north west) --  coordinate (q12)  (q1.south east |- q2.north west) -- (q1.south east) --  (q1.south east -| q2.north west) -- cycle;
    \draw[tensor,onearrow] (q12) -- (q3);
  \end{tikzpicture}\;
  $$
  and so
  \begin{multline*}
  \partial^2\tikz[baseline=(q.base)] \node[draw] (q) {$q$}; \ni
  (-1)^{\deg \tikz[baseline=(q.base)] \node[draw,inner sep=1pt] (q) {$\scriptstyle q_1$};}
   (-1)^{\deg q_1}
  \;\begin{tikzpicture}[baseline=(q2.base),xscale=-1,yscale=1.5]
    \path (0,0) node[draw] (q1) {$q_{1}$};
    \path (.7,.7) node (q2) {$q_{2}$};
    \path (-.7,1.4) node (q3) {$q_{3}$};
    \draw[tensor,onearrow] (q2) -- (q3);
    \draw (q2.south west) --  coordinate (q23)  (q3.north east |- q2.south west) -- (q3.north east) --  (q3.north east -| q2.south west) -- node[anchor=east] {$\partial$} (q2.south west) ;
    \draw[tensor,onearrow] (q1) -- (q23);
  \end{tikzpicture}\;
  + (-1)^{\deg q_1+ \deg q_2}
  \;\begin{tikzpicture}[baseline=(q2.base),xscale=-1,yscale=1.5]
    \path (0,0) node (q1) {$q_{1}$};
    \path (.7,.7) node (q2) {$q_{2}$};
    \path (-.7,1.4) node[draw] (q3) {$q_{3}$};
    \draw[tensor,onearrow] (q1) -- (q2);
    \draw (q2.north west) --  coordinate (q12)  (q1.south east |- q2.north west) -- (q1.south east) --  (q1.south east -| q2.north west) -- node[anchor=east] {$\partial$} (q2.north west);
    \draw[tensor,onearrow] (q12) -- (q3);
  \end{tikzpicture}\;
  \\ \ni
  (-1)^{\deg q_1 + 1}
   (-1)^{\deg q_1}
  (-1)^{\deg q_2}
  \;\begin{tikzpicture}[baseline=(q2.base),xscale=-1,yscale=1.5]
    \path (0,0) node[draw] (q1) {$q_{1}$};
    \path (.7,.7) node[draw] (q2) {$q_{2}$};
    \path (-.7,1.4) node[draw] (q3) {$q_{3}$};
    \draw[tensor,onearrow] (q1) -- (q2);
    \draw[tensor,onearrow] (q1) -- (q3);
    \draw[tensor,onearrow] (q2) -- (q3);
  \end{tikzpicture}\;
  +(-1)^{\deg q_1+ \deg q_2}(-1)^{\deg q_1}
  \;\begin{tikzpicture}[baseline=(q2.base),xscale=-1,yscale=1.5]
    \path (0,0) node[draw] (q1) {$q_{1}$};
    \path (.7,.7) node[draw] (q2) {$q_{2}$};
    \path (-.7,1.4) node[draw] (q3) {$q_{3}$};
    \draw[tensor,onearrow] (q1) -- (q2);
    \draw[tensor,onearrow] (q1) -- (q3);
    \draw[tensor,onearrow] (q2) -- (q3);
  \end{tikzpicture}\;
  = 0.
  \end{multline*}
  Similar calculations apply to the other types of associativity from Definition~\ref{defn.diproperad}.
\end{remark}

One of the many uses of the cobar construction is to organize the inductive computation of quadratic duals:

\begin{lemma} \label{lemma.cobarinductive}
  Let $P = \Ff(T)/\<R\>$ be a quadratic (di/pr)operad and $P^{\shriek}$ its quadratic dual.  The canonical $\bR^{\times}$ action on $T$ extends to a grading by positive integers of both $P$ and $P^{\shriek}$.  For $x\in \Ff^{(k)}(T[1])$, let $x\iota^{-1} \in \Ff^{(k)}(T[1])[-1]$ denote the corresponding generator of $\B\Ff(T[1])$, and $\partial(x\iota^{-1})$ its differential therein.  Suppose that $k\geq 3$.  Then $x\in P^{\shriek}$ if and only if  $\partial(x\iota^{-1}) \in \B P^{\shriek} \subseteq \B\Ff(T[1])$.
\end{lemma}

\begin{proof}
As in more familiar cases, a co(di/pr)operad $Q$ is entirely determined by its cobar construction $\B Q$ along with the inclusion $Q[-1] \mono \B Q$ of (non-dg) $\bS$-bimodules~\cite{MR2320654}.  In particular, the different binary decompositions of an element $x\in Q$ correspond (without cancelations) to the summands appearing in $\partial(x\iota^{-1})$, where $x\iota^{-1} \in Q[-1]$ is the corresponding element of the shifted $\bS$-bimodule.

Suppose that $x\in \Ff^{(k)}(T[1])$ with $k\geq 3$.  Definition~\ref{defn.quaddual} then says that $x\in P^{\shriek}$ if and only if all binary decompositions of $x$ are elements of $\Ff^{(2)}(P^{\shriek})$.  (The binary decompositions of $x$ consist of elements of individual weights, for the $\bR^{\times}$ action on $T$, each less than $k$, and so whether or not they are in $P^{\shriek}$ has been determined by induction; in weights $1$ and $2$ one has $T[1]$ and $R[2]$ respectively.)   
By unpacking the definition of the cobar construction, one sees that
this happens if and only if $\partial(x \iota^{-1}) \in \B(P^{\shriek})$.
\end{proof}

The (di/pr)operad $\B Q$ is an example of a \define{quasifree} (di/pr)operad, which more generally is any dg (di/pr)operad which would be free if one were to forget its differential.

\begin{lemma}\label{lemma.kos}
  Let $P$ be a  quadratic (di/pr)operad, with generators $T$ and relations $R$. Then $\B P^{\shriek}$ is cofibrant and fibers over $P$.
\end{lemma}
\begin{proof}
  Cofibrancy follows from~\cite[Corollary~40]{MR2572248}.  To define the map $\B P^{\shriek} \to P$, it suffices to define the action on generators $P^{\shriek}[-1] \subseteq \Ff(T[1])[-1]$.  We declare that $T[1][-1] \subseteq P^{\shriek}[-1]$ maps by the identity to $T \subseteq P$, and that all other generators map to $0$.  To check that this is well-defined, it suffices to check that the derivatives of the generators in $\Ff^{(2)}(T[1])[-1]$ get mapped to $0$.  But these generators are precisely a copy of $R[1]$, and differentiating and mapping gives the (vanishing) image of $R$ in $P = \Ff(T)/\<R\>$.
\end{proof}

\begin{definition}
  A quadratic (di/pr)operad is \define{Koszul} if the canonical fibration $\B P^{\shriek} \epi P$ from Lemma~\ref{lemma.kos} is acyclic, in which case $\B P^{\shriek}$ is a cofibrant replacement of $P$. 
  When $P$ is Koszul, we let $\sh P = \B P^{\shriek}$,  and call $\sh P$-algebras \define{strong homotopy $P$-algebras}.
\end{definition}

For any (di/pr)operad $P$ satisfying some mild finite-dimensionality assumptions, the (di/pr)operad $\B((\B P^{*})^{*})$ is always a cofibrant replacement of $P$.  (The second dual should be taken relative to the grading induced by the $\bR^{\times}$ action on the $\bS$-bimodule $P$.)  The point is that $\B P^{\shriek}$ is generally much smaller than $\B((\B P^{*})^{*})$, and hence more manageable.

The main result of this section says that $\Frob_{1}$ is Koszul, as both a dioperad and as a properad:

\begin{proposition}
  The (di/pr)operad $\Frob_{1}$ of open and coopen commutative Frobenius algebras has the following quadratic presentation, with respect to which it is Koszul.  The generating $\bS$-bimodule $T$ is spanned by:
    $$\underbrace{
  \,\tikz[baseline=(basepoint)]{ 
    \path (0,0) coordinate (basepoint) (0,4pt) node[dot] {};
    \draw[onearrow](0,-6pt) -- (0,4pt);
    \draw[onearrow](0,4pt) -- (-8pt,14pt);
    \draw[onearrow](0,4pt) -- (8pt,14pt);
  }\,
  = 
  \,\tikz[baseline=(basepoint)]{ 
    \path (0,0) coordinate (basepoint) (0,4pt) node[dot] {};
    \draw[onearrow](0,-6pt) -- (0,4pt);
    \draw[](0,4pt) .. controls +(8pt,4pt) and +(8pt,-4pt) .. (-8pt,14pt);
    \draw[](0,4pt) .. controls +(-8pt,4pt) and +(-8pt,-4pt) .. (8pt,14pt);
  }\,
  ,}_{\text{cohomological degree $0$}} \quad 
  \underbrace{\,\tikz[baseline=(basepoint)]{ 
    \path (0,0) coordinate (basepoint) (0,4pt) node[dot] {};
    \draw[onearrow](-8pt,-6pt) -- (0,4pt);
    \draw[onearrow](8pt,-6pt) -- (0,4pt);
    \draw[onearrow](0,4pt) -- (0,14pt);
  }\,
  = -
  \,\tikz[baseline=(basepoint)]{ 
    \path (0,0) coordinate (basepoint) (0,4pt) node[dot] {};
    \draw[] (-8pt,-6pt) .. controls +(8pt,4pt) and +(8pt,-4pt) .. (0,4pt);
    \draw[] (8pt,-6pt) .. controls +(-8pt,4pt) and +(-8pt,-4pt) .. (0,4pt);
    \draw[onearrow](0,4pt) -- (0,14pt);
  }\,}_{\text{cohomological degree $1$}}
  $$
  A basis for the relations $R$ is:
  \begin{gather*}
  \,
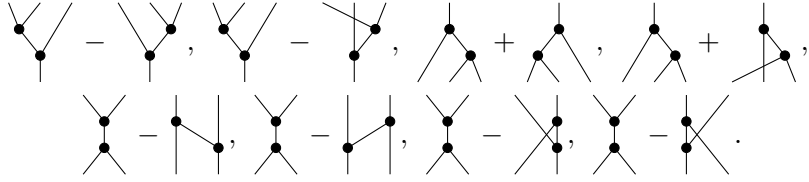
\tikz[baseline=(basepoint)]{ 
    \path (0,3pt) coordinate (basepoint) (0,1pt) node[dot] {} (-8pt,11pt) node[dot] {};
    \draw[](0,-9pt) -- (0,1pt);
    \draw[](0,1pt) -- (-8pt,11pt);
    \draw[](-8pt,11pt) -- (-12pt,21pt);
    \draw[](-8pt,11pt) -- (0pt,21pt);
    \draw[](0,1pt) -- (12pt,21pt);
  }\,
  -
  \,\tikz[baseline=(basepoint)]{ 
    \path (0,3pt) coordinate (basepoint) (0,1pt) node[dot] {} (8pt,11pt) node[dot] {};
    \draw[](0,-9pt) -- (0,1pt);
    \draw[](0,1pt) -- (8pt,11pt);
    \draw[](0,1pt) -- (-12pt,21pt);
    \draw[](8pt,11pt) -- (0pt,21pt);
    \draw[](8pt,11pt) -- (12pt,21pt);
  }\,
  ,\; 
  \,\tikz[baseline=(basepoint)]{ 
    \path (0,3pt) coordinate (basepoint) (0,1pt) node[dot] {} (-8pt,11pt) node[dot] {};
    \draw[](0,-9pt) -- (0,1pt);
    \draw[](0,1pt) -- (-8pt,11pt);
    \draw[](-8pt,11pt) -- (-12pt,21pt);
    \draw[](-8pt,11pt) -- (0pt,21pt);
    \draw[](0,1pt) -- (12pt,21pt);
  }\,
  -
  \,\tikz[baseline=(basepoint)]{ 
    \path (0,3pt) coordinate (basepoint) (0,1pt) node[dot] {} (8pt,11pt) node[dot] {};
    \draw[](0,-9pt) -- (0,1pt);
    \draw[](0,1pt) -- (8pt,11pt);
    \draw[](0,1pt) -- (0pt,21pt);
    \draw[](8pt,11pt) -- (-12pt,21pt);
    \draw[](8pt,11pt) -- (12pt,21pt);
  }\,
  ,\;
  \,\tikz[baseline=(basepoint)]{ 
    \path (0,3pt) coordinate (basepoint) (8pt,1pt) node[dot] {} (0,11pt) node[dot] {};
    \draw[](-12pt,-9pt) -- (0,11pt);
    \draw[](0pt,-9pt) -- (8pt,1pt);
    \draw[](12pt,-9pt) -- (8pt,1pt);
    \draw[](8pt,1pt) -- (0,11pt);
    \draw[](0,11pt) -- (0,21pt);
  }\,
  +
  \,\tikz[baseline=(basepoint)]{ 
    \path (0,3pt) coordinate (basepoint) (-8pt,1pt) node[dot] {} (0,11pt) node[dot] {};
    \draw[](12pt,-9pt) -- (0,11pt);
    \draw[](-12pt,-9pt) -- (-8pt,1pt);
    \draw[](0pt,-9pt) -- (-8pt,1pt);
    \draw[](-8pt,1pt) -- (0,11pt);
    \draw[](0,11pt) -- (0,21pt);
  }\,
  ,\; 
  \,\tikz[baseline=(basepoint)]{ 
    \path (0,3pt) coordinate (basepoint) (8pt,1pt) node[dot] {} (0,11pt) node[dot] {};
    \draw[](-12pt,-9pt) -- (0,11pt);
    \draw[](0pt,-9pt) -- (8pt,1pt);
    \draw[](12pt,-9pt) -- (8pt,1pt);
    \draw[](8pt,1pt) -- (0,11pt);
    \draw[](0,11pt) -- (0,21pt);
  }\,
  +
  \,\tikz[baseline=(basepoint)]{ 
    \path (0,3pt) coordinate (basepoint) (8pt,1pt) node[dot] {} (0,11pt) node[dot] {};
    \draw[](0pt,-9pt) -- (0,11pt);
    \draw[](12pt,-9pt) -- (8pt,1pt);
    \draw[](-12pt,-9pt) -- (8pt,1pt);
    \draw[](8pt,1pt) -- (0,11pt);
    \draw[](0,11pt) -- (0,21pt);
  }\,
  ,\; 
  \\
  \,\tikz[baseline=(basepoint)]{ 
    \path (0,3pt) coordinate (basepoint) (0,1pt) node[dot] {} (0,11pt) node[dot] {};
    \draw[](-8pt,-9pt) -- (0,1pt);
    \draw[](8pt,-9pt) -- (0,1pt);
    \draw[](0,1pt) -- (0,11pt);
    \draw[](0,11pt) -- (-8pt,21pt);
    \draw[](0,11pt) -- (8pt,21pt);
  }\,
  -
  \,\tikz[baseline=(basepoint)]{ 
    \path (0,3pt) coordinate (basepoint) (8pt,1pt) node[dot] {} (-8pt,11pt) node[dot] {};
    \draw[](-8pt,-9pt) -- (-8pt,11pt);
    \draw[](8pt,-9pt) -- (8pt,1pt);
    \draw[](8pt,1pt) -- (-8pt,11pt);
    \draw[](-8pt,11pt) -- (-8pt,21pt);
    \draw[](8pt,1pt) -- (8pt,21pt);
  }\,
  ,\;
  \,\tikz[baseline=(basepoint)]{ 
    \path (0,3pt) coordinate (basepoint) (0,1pt) node[dot] {} (0,11pt) node[dot] {};
    \draw[](-8pt,-9pt) -- (0,1pt);
    \draw[](8pt,-9pt) -- (0,1pt);
    \draw[](0,1pt) -- (0,11pt);
    \draw[](0,11pt) -- (-8pt,21pt);
    \draw[](0,11pt) -- (8pt,21pt);
  }\,
  -
  \,\tikz[baseline=(basepoint)]{ 
    \path (0,3pt) coordinate (basepoint) (-8pt,1pt) node[dot] {} (8pt,11pt) node[dot] {};
    \draw[](-8pt,-9pt) -- (-8pt,1pt);
    \draw[](8pt,-9pt) -- (8pt,11pt);
    \draw[](-8pt,1pt) -- (8pt,11pt);
    \draw[](-8pt,1pt) -- (-8pt,21pt);
    \draw[](8pt,11pt) -- (8pt,21pt);
  }\,
  ,\;
  \,\tikz[baseline=(basepoint)]{ 
    \path (0,3pt) coordinate (basepoint) (0,1pt) node[dot] {} (0,11pt) node[dot] {};
    \draw[](-8pt,-9pt) -- (0,1pt);
    \draw[](8pt,-9pt) -- (0,1pt);
    \draw[](0,1pt) -- (0,11pt);
    \draw[](0,11pt) -- (-8pt,21pt);
    \draw[](0,11pt) -- (8pt,21pt);
  }\,
  -
  \,\tikz[baseline=(basepoint)]{ 
    \path (0,3pt) coordinate (basepoint) (0,1pt) node[dot] {} (0,11pt) node[dot] {};
    \draw[](0,-9pt) -- (0,1pt);
    \draw[](0,1pt) -- (0,11pt);
    \draw[](0,11pt) -- (0,21pt);
    \draw[](-16pt,-9pt) -- (0,11pt);
    \draw[](0pt,1pt) -- (-16pt,21pt);
  }\,
  ,\;
  \,\tikz[baseline=(basepoint)]{ 
    \path (0,3pt) coordinate (basepoint) (0,1pt) node[dot] {} (0,11pt) node[dot] {};
    \draw[](-8pt,-9pt) -- (0,1pt);
    \draw[](8pt,-9pt) -- (0,1pt);
    \draw[](0,1pt) -- (0,11pt);
    \draw[](0,11pt) -- (-8pt,21pt);
    \draw[](0,11pt) -- (8pt,21pt);
  }\,
  -
  \,\tikz[baseline=(basepoint)]{ 
    \path (0,3pt) coordinate (basepoint) (0,1pt) node[dot] {} (0,11pt) node[dot] {};
    \draw[](0,-9pt) -- (0,1pt);
    \draw[](0,1pt) -- (0,11pt);
    \draw[](0,11pt) -- (0,21pt);
    \draw[](16pt,-9pt) -- (0,11pt);
    \draw[](0pt,1pt) -- (16pt,21pt);
  }\,
  . \end{gather*}
\end{proposition}

\begin{proof}
It is clear that this presentation defines the dioperad $\Frob_{1}$ and hence its universal enveloping properad.  But that universal enveloping properad is no larger, because the composition $\graphI$ vanishes, as it must transform both trivially and by the sign representation under the $\bS_{2}$-action interchanging the two interior edges.

We now check Koszulity.
  The suboperad of $\Frob_{1}$ generated by just the multiplication $\mult$ is nothing but the nonunital commutative operad $\mathrm{Com}$.  The sub(di/pr)operad generated by $\comult$ is a \define{shear} of a (di/pr)operad for cocommutative coalgebras, meaning its representations on $V$ are representations of $\mathrm{Cocom}$ on $V[1]$.  Both $\mathrm{Com}$ and $\mathrm{Cocom}$, as well as their shears, are  known to be Koszul~\cite[Theorem~8.5.7]{MR2954392}.  The second line of relations, along with the relation $
  \graphI=0$, are together a ``replacement rule'' in the sense of~\cite[8.1]{MR2320654}, and hence $\Frob_{1}$ is Koszul by~\cite[Proposition 8.4]{MR2320654}.
\end{proof}

\begin{corollary}\label{cor.shFrob}
  $\Frob_{1}$ has a cofibrant replacement $\sh\Frob_{1}$, which is quasifree with generating $\bS$-bimodule $(\Frob_{1})^{\shriek}[-1]$.  The generating $\bS$-bimodule $T$ of $\Frob_{1}$ has a bigrading by $\bigl(\#\mult,\#\comult\bigr)$, and the relations $R$ are homogeneous for this bigrading, hence this bigrading extends to $(\Frob_{1})^{\shriek}[-1]$, and the differential on $\sh\Frob_{1}$ preserves this bigrading.
  
  If we are working with dioperads, the piece $(\Frob_{1})^{\shriek}[-1](m,n)$ with $m$ inputs and $n$ outputs is homogeneous for this bigrading, with $\#\mult = m-1$ and $\#\comult = n-1$, and is entirely in cohomological degree $1-\#\mult  = 2-m$. 
  
  If we are working with properads, the piece $(\Frob_{1})^{\shriek}[-1](m,n)$ with $m$ inputs and $n$ outputs has a grading by \define{genus} $\beta$ satisfying
  $ \#\mult = \beta + m-1,$ $ \#\comult = \beta + n-1,$ and $ \text{cohomological degree} = 1- \#\mult  = 2-(\beta + m).$ 
\end{corollary}
\begin{proof}
  The formulas follow from definition-unpacking and elementary combinatorics.
\end{proof}

We conclude this section by describing the generators $(\Frob_{1})^{\shriek}[-1]$ of $\sh\Frob_{1}$ for small $\#\mult + \#\comult$.  We will record the representations of the symmetric group using the usual Young diagrams, placed under a generator of that representation. 
  The $\#\mult + \#\comult = 1$ piece of $(\Frob_{1})^{\shriek}[-1]$ is a copy of the generators $T$, and decomposes as:
$$
  \underset
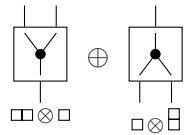
{ \tikz \draw (0,0) rectangle (4pt,4pt) (4pt,0pt) rectangle (8pt,4pt);  \,\otimes\, \tikz \draw (0,0) rectangle (4pt,4pt);}{
  \,\tikz[baseline=(main.base)]{
    \node[draw,rectangle,inner sep=2pt] (main) {\mult};
    \draw (main.south) -- ++(0,-8pt);
    \draw (main.north) ++(-6pt,0) -- ++(0,8pt);
    \draw (main.north) ++(6pt,0) -- ++(0,8pt);
  }\,
  }
  \;\oplus\;
\underset{\tikz \draw (0,0) rectangle (4pt,4pt); \,\otimes\, \tikz \draw (0,0) rectangle (4pt,4pt) (0pt,4pt) rectangle (4pt,8pt); }{
  \,\tikz[baseline=(main.base)]{
    \node[draw,rectangle,inner sep=2pt] (main) {\comult};
    \draw (main.north) -- ++(0,8pt);
    \draw (main.south) ++(-6pt,0) -- ++(0,-8pt);
    \draw (main.south) ++(6pt,0) -- ++(0,-8pt);
  }\,
  }
  $$
  
The $\#\mult + \#\comult = 2$ piece is a copy of $R[2][-1]$, decomposing as:
$$
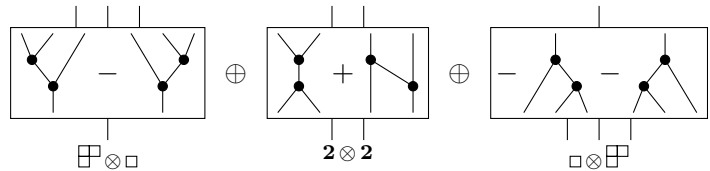
\underset{\tikz \draw (0,0) rectangle (4pt,4pt) (4pt,0pt) rectangle (8pt,4pt) (0pt,-4pt)  rectangle  (4pt,0pt) ;\,\otimes\,   \tikz \draw (0,0) rectangle (4pt,4pt);  }{
   \,\tikz[baseline=(main.base)]{
    \node[draw,rectangle,inner sep=2pt] (main) {$
       \,\tikz[baseline=(basepoint)]{ 
    \path (0,3pt) coordinate (basepoint) (0,1pt) node[dot] {} (-8pt,11pt) node[dot] {};
    \draw[](0,-9pt) -- (0,1pt);
    \draw[](0,1pt) -- (-8pt,11pt);
    \draw[](-8pt,11pt) -- (-12pt,21pt);
    \draw[](-8pt,11pt) -- (0pt,21pt);
    \draw[](0,1pt) -- (12pt,21pt);
  }\,
  -
  \,\tikz[baseline=(basepoint)]{ 
    \path (0,3pt) coordinate (basepoint) (0,1pt) node[dot] {} (8pt,11pt) node[dot] {};
    \draw[](0,-9pt) -- (0,1pt);
    \draw[](0,1pt) -- (8pt,11pt);
    \draw[](0,1pt) -- (-12pt,21pt);
    \draw[](8pt,11pt) -- (0pt,21pt);
    \draw[](8pt,11pt) -- (12pt,21pt);
  }\,
    $};
    \draw (main.south) -- ++(0,-8pt);
    \draw (main.north) ++(-12pt,0) -- ++(0,8pt);
    \draw (main.north) ++(0pt,0) -- ++(0,8pt);
    \draw (main.north) ++(12pt,0) -- ++(0,8pt);
  }\,
  }
  \;\oplus\;
\underset{\mathbf 2 \,\otimes\,  \mathbf 2  }{
   \,\tikz[baseline=(main.base)]{
    \node[draw,rectangle,inner sep=2pt] (main) {$
         \,\tikz[baseline=(basepoint)]{ 
    \path (0,3pt) coordinate (basepoint) (0,1pt) node[dot] {} (0,11pt) node[dot] {};
    \draw[](-8pt,-9pt) -- (0,1pt);
    \draw[](8pt,-9pt) -- (0,1pt);
    \draw[](0,1pt) -- (0,11pt);
    \draw[](0,11pt) -- (-8pt,21pt);
    \draw[](0,11pt) -- (8pt,21pt);
  }\,
  +
  \,\tikz[baseline=(basepoint)]{ 
    \path (0,3pt) coordinate (basepoint) (8pt,1pt) node[dot] {} (-8pt,11pt) node[dot] {};
    \draw[](-8pt,-9pt) -- (-8pt,11pt);
    \draw[](8pt,-9pt) -- (8pt,1pt);
    \draw[](8pt,1pt) -- (-8pt,11pt);
    \draw[](-8pt,11pt) -- (-8pt,21pt);
    \draw[](8pt,1pt) -- (8pt,21pt);
  }\,
    $};
    \draw (main.south) ++(-6pt,0) -- ++(0pt,-8pt);
    \draw (main.south) ++(6pt,0) -- ++(0pt,-8pt);
    \draw (main.north) ++(-6pt,0) -- ++(0,8pt);
    \draw (main.north) ++(6pt,0) -- ++(0,8pt);
  }\,
  }
  \;\oplus\;
  \underset{\tikz \draw (0,0) rectangle (4pt,4pt);  \,\otimes\, \tikz \draw (0,0) rectangle (4pt,4pt) (4pt,0pt) rectangle (8pt,4pt) (0pt,-4pt)  rectangle  (4pt,0pt) ;  }{
  \,\tikz[baseline=(main.base)]{
    \node[draw,rectangle,inner sep=2pt] (main) {$
   -   \,\tikz[baseline=(basepoint)]{ 
    \path (0,3pt) coordinate (basepoint) (8pt,1pt) node[dot] {} (0,11pt) node[dot] {};
    \draw[](-12pt,-9pt) -- (0,11pt);
    \draw[](0pt,-9pt) -- (8pt,1pt);
    \draw[](12pt,-9pt) -- (8pt,1pt);
    \draw[](8pt,1pt) -- (0,11pt);
    \draw[](0,11pt) -- (0,21pt);
  }\,
  -
  \,\tikz[baseline=(basepoint)]{ 
    \path (0,3pt) coordinate (basepoint) (-8pt,1pt) node[dot] {} (0,11pt) node[dot] {};
    \draw[](12pt,-9pt) -- (0,11pt);
    \draw[](-12pt,-9pt) -- (-8pt,1pt);
    \draw[](0pt,-9pt) -- (-8pt,1pt);
    \draw[](-8pt,1pt) -- (0,11pt);
    \draw[](0,11pt) -- (0,21pt);
  }\,
    $};
    \draw (main.north) -- ++(0,8pt);
    \draw (main.south) ++(-12pt,0) -- ++(0,-8pt);
    \draw (main.south) ++(0pt,0) -- ++(0,-8pt);
    \draw (main.south) ++(12pt,0) -- ++(0,-8pt);
  }\,
  }
$$
In the middle summand, $\mathbf 2 = \tikz \draw (0,0) rectangle (4pt,4pt) (4pt,0pt) rectangle (8pt,4pt); \oplus \tikz \draw (0,0) rectangle (4pt,4pt) (0pt,4pt) rectangle (4pt,8pt);$ denotes the two-dimensional permutation representation of $\bS_{2}$.  
The signs stem from the conventions discussed after Definition~\ref{defn.quaddual}.  They are essentially notational.  

We will refer to these three generators as ``the homotopies controlling'' associativity, the Frobenius relation, and coassociativity, respectively.  Tracking signs, their derivatives are:
\begin{gather} \label{eqn.associator}
  \partial\left(
     \,\tikz[baseline=(main.base)]{
    \node[draw,rectangle,inner sep=2pt] (main) {$
       \,\tikz[baseline=(basepoint)]{ 
    \path (0,3pt) coordinate (basepoint) (0,1pt) node[dot] {} (-8pt,11pt) node[dot] {};
    \draw[](0,-9pt) -- (0,1pt);
    \draw[](0,1pt) -- (-8pt,11pt);
    \draw[](-8pt,11pt) -- (-12pt,21pt);
    \draw[](-8pt,11pt) -- (0pt,21pt);
    \draw[](0,1pt) -- (12pt,21pt);
  }\,
  -
  \,\tikz[baseline=(basepoint)]{ 
    \path (0,3pt) coordinate (basepoint) (0,1pt) node[dot] {} (8pt,11pt) node[dot] {};
    \draw[](0,-9pt) -- (0,1pt);
    \draw[](0,1pt) -- (8pt,11pt);
    \draw[](0,1pt) -- (-12pt,21pt);
    \draw[](8pt,11pt) -- (0pt,21pt);
    \draw[](8pt,11pt) -- (12pt,21pt);
  }\,
    $};
    \draw (main.south) -- ++(0,-8pt);
    \draw (main.north) ++(-12pt,0) -- ++(0,8pt);
    \draw (main.north) ++(0pt,0) -- ++(0,8pt);
    \draw (main.north) ++(12pt,0) -- ++(0,8pt);
  }\,
  \right) = 
  \;-\;
  \,\tikz[baseline=(somewhere)]{ 
    \path (0,10pt) coordinate (somewhere);
    \path (0pt,0pt) node[draw,inner sep=2pt,fill=white] (a) {\mult} (-10pt,25pt) node[draw,inner sep=2pt,fill=white] (b) {\mult};
    \draw[](b) -- (a);
    \draw[](a) -- (0,-20pt);
    \draw[](a) -- (20pt,45pt);
    \draw[](b) -- (-20pt,45pt);
    \draw[](b) -- (0pt,45pt);
  }\,
  \;+\;
  \,\tikz[baseline=(somewhere)]{ 
    \path (0,10pt) coordinate (somewhere);
    \path (0pt,0pt) node[draw,inner sep=2pt,fill=white] (a) {\mult} (10pt,25pt) node[draw,inner sep=2pt,fill=white] (b) {\mult};
    \draw[](b) -- (a);
    \draw[](a) -- (0,-20pt);
    \draw[](a) -- (-20pt,45pt);
    \draw[](b) -- (20pt,45pt);
    \draw[](b) -- (0pt,45pt);
  }\,
  \\ \label{eqn.frobeniator}
  \partial \left(
    \,\tikz[baseline=(main.base)]{
    \node[draw,rectangle,inner sep=2pt] (main) {$
         \,\tikz[baseline=(basepoint)]{ 
    \path (0,3pt) coordinate (basepoint) (0,1pt) node[dot] {} (0,11pt) node[dot] {};
    \draw[](-8pt,-9pt) -- (0,1pt);
    \draw[](8pt,-9pt) -- (0,1pt);
    \draw[](0,1pt) -- (0,11pt);
    \draw[](0,11pt) -- (-8pt,21pt);
    \draw[](0,11pt) -- (8pt,21pt);
  }\,
  +
  \,\tikz[baseline=(basepoint)]{ 
    \path (0,3pt) coordinate (basepoint) (8pt,1pt) node[dot] {} (-8pt,11pt) node[dot] {};
    \draw[](-8pt,-9pt) -- (-8pt,11pt);
    \draw[](8pt,-9pt) -- (8pt,1pt);
    \draw[](8pt,1pt) -- (-8pt,11pt);
    \draw[](-8pt,11pt) -- (-8pt,21pt);
    \draw[](8pt,1pt) -- (8pt,21pt);
  }\,
    $};
    \draw (main.south) ++(-6pt,0) -- ++(0pt,-8pt);
    \draw (main.south) ++(6pt,0) -- ++(0pt,-8pt);
    \draw (main.north) ++(-6pt,0) -- ++(0,8pt);
    \draw (main.north) ++(6pt,0) -- ++(0,8pt);
  }\,
  \right) = 
  \,\tikz[baseline=(somewhere)]{ 
    \path (0,10pt) coordinate (somewhere);
    \path (0pt,0pt) node[draw,inner sep=2pt,fill=white] (a) {\comult} (0pt,25pt) node[draw,inner sep=2pt,fill=white] (b) {\mult};
    \draw[](b) -- (a);
    \draw[](b) -- (-10pt,45pt);
    \draw[](b) -- (10pt,45pt);
    \draw[](a) -- (10pt,-20pt);
    \draw[](a) -- (-10pt,-20pt);
  }\,
  \;-\; 
  \,\tikz[baseline=(somewhere)]{ 
    \path (0,10pt) coordinate (somewhere);
    \path (20pt,0pt) node[draw,inner sep=2pt,fill=white] (a) {\mult} (0pt,25pt) node[draw,inner sep=2pt,fill=white] (b) {\comult};
    \draw[](b) -- (a);
    \draw[](b) -- (0pt,45pt);
    \draw[](a) -- (20pt,45pt);
    \draw[](a) -- (20pt,-20pt);
    \draw[](b) -- (0pt,-20pt);
  }\,
  \\ \label{eqn.coassociator}
  \partial \left(
    \,\tikz[baseline=(main.base)]{
    \node[draw,rectangle,inner sep=2pt] (main) {$
   -   \,\tikz[baseline=(basepoint)]{ 
    \path (0,3pt) coordinate (basepoint) (8pt,1pt) node[dot] {} (0,11pt) node[dot] {};
    \draw[](-12pt,-9pt) -- (0,11pt);
    \draw[](0pt,-9pt) -- (8pt,1pt);
    \draw[](12pt,-9pt) -- (8pt,1pt);
    \draw[](8pt,1pt) -- (0,11pt);
    \draw[](0,11pt) -- (0,21pt);
  }\,
  -
  \,\tikz[baseline=(basepoint)]{ 
    \path (0,3pt) coordinate (basepoint) (-8pt,1pt) node[dot] {} (0,11pt) node[dot] {};
    \draw[](12pt,-9pt) -- (0,11pt);
    \draw[](-12pt,-9pt) -- (-8pt,1pt);
    \draw[](0pt,-9pt) -- (-8pt,1pt);
    \draw[](-8pt,1pt) -- (0,11pt);
    \draw[](0,11pt) -- (0,21pt);
  }\,
    $};
    \draw (main.north) -- ++(0,8pt);
    \draw (main.south) ++(-12pt,0) -- ++(0,-8pt);
    \draw (main.south) ++(0pt,0) -- ++(0,-8pt);
    \draw (main.south) ++(12pt,0) -- ++(0,-8pt);
  }\,
  \right) = 
  \;-\;
  \,\tikz[baseline=(somewhere)]{ 
    \path (0,10pt) coordinate (somewhere);
    \path (10pt,0pt) node[draw,inner sep=2pt,fill=white] (a) {\comult} (0pt,25pt) node[draw,inner sep=2pt,fill=white] (b) {\comult};
    \draw[](b) -- (a);
    \draw[](b) -- (0,45pt);
    \draw[](a) -- (20pt,-20pt);
    \draw[](b) -- (-20pt,-20pt);
    \draw[](a) -- (0pt,-20pt);
  }\,
  \;-\;
  \,\tikz[baseline=(somewhere)]{ 
    \path (0,10pt) coordinate (somewhere);
    \path (-10pt,0pt) node[draw,inner sep=2pt,fill=white] (a) {\comult} (0pt,25pt) node[draw,inner sep=2pt,fill=white] (b) {\comult};
    \draw[](b) -- (a);
    \draw[](b) -- (0,45pt);
    \draw[](b) -- (20pt,-20pt);
    \draw[](a) -- (-20pt,-20pt);
    \draw[](a) -- (0pt,-20pt);
  }\,
\end{gather}
The signs are because, in $(\Frob_{1})^{\shriek}$, we have $\deg\bigl(\mult\bigr) = -1$ and $\deg\bigl(\comult\bigr) = 0$, and we always collect a sign for the unboxed degree of the bottom vertex.

When $\#\mult + \#\comult \geq 3$, the dioperadic and properadic versions of $(\Frob_{1})^{\shriek}[-1]$ diverge.  The dioperadic version appears as the genus $\beta = 0$ direct summand of the properadic version.  
 Of the  $\#\mult + \#\comult = 3$ part of $(\Frob_{1})^{\shriek}[-1]$, we will compute just the summand with genus $\beta = 1$.  Any composition of $\mult$ and $\comult$ with genus $\beta \geq 1$ has $\#\mult \geq 1$ and $\#\comult \geq 1$.  Since $
  \graphI=0$, the only graphs with $\beta = 1$ and $\bigl(\#\mult,\#\comult\bigr) = (1,2)$ are 
{\graphD}
and
\reflectbox{\graphD}.  Each of these is in $(\Frob_{1})^{\shriek}[-1]$, so that the $\#\mult + \#\comult = 3$ part of the properadic $(\Frob_{1})^{\shriek}[-1]$ is:
$$ \bigl(\beta = 0 \text{ part}\bigr) \;\oplus\;
\underset{ \mathbf 2 \,\otimes\, \tikz \draw (0,0) rectangle (4pt,4pt); }{
  \graphDbox
}
\;\oplus\;
\underset{\tikz \draw (0,0) rectangle (4pt,4pt);\,\otimes\,  \mathbf 2 }{
  \graphAbox  }
$$
 
To verify this claim, one must compute 
  $\partial\left(
  \graphDbox
  \right)$ and $\partial\left(
  \graphAbox
  \right)$ in $\B(\Ff(T[1]))$, and see that in fact each derivative lands in the subproperad $\sh\Frob_{1}= \B\bigl((\Frob_{1})^{\shriek}\bigr)$.  As we will also need these derivatives later, we will work out the derivatives carefully, keeping track of signs.  
  
  We first claim:
  \begin{equation}\label{Dderivative}
  \partial\left(
  \graphDbox
  \right) = 
  -
  \;\begin{tikzpicture}[baseline=(outerbasepoint)]
  \node[draw,inner sep=2pt] (box) {
    $
    \,\tikz[baseline=(basepoint)]{ 
      \path (0,3pt) coordinate (basepoint) (0,1pt) node[dot] {} (-8pt,11pt) node[dot] {};
      \draw[](0,-9pt) -- (0,1pt);
      \draw[](0,1pt) -- (-8pt,11pt);
      \draw[](-8pt,11pt) -- (-12pt,21pt);
      \draw[](-8pt,11pt) -- (0pt,21pt);
      \draw[](0,1pt) -- (12pt,21pt);
    }\,
    -
    \,\tikz[baseline=(basepoint)]{ 
      \path (0,3pt) coordinate (basepoint) (0,1pt) node[dot] {} (8pt,11pt) node[dot] {};
      \draw[](0,-9pt) -- (0,1pt);
      \draw[](0,1pt) -- (8pt,11pt);
      \draw[](0,1pt) -- (-12pt,21pt);
      \draw[](8pt,11pt) -- (0pt,21pt);
      \draw[](8pt,11pt) -- (12pt,21pt);
    }\,
    $};
  \draw (box.south) -- ++(0,-10pt) coordinate (bottom);
  \draw (box.north) -- ++(0pt,10pt);
  \draw (box.north) ++(-12pt,0) -- ++(0pt,10pt);
  \draw (box.north) ++(-6pt,18pt) node[draw,inner sep=2pt,fill=white] (upperbox) {\comult};
  \draw (upperbox.north) -- ++(0,10pt) ++(18pt,0) coordinate (top);
  \draw (box.north) ++(12pt,0) -- (top);
  \path (bottom) ++(0,40pt) coordinate (outerbasepoint);  
  \end{tikzpicture}\;
  \;-\;
  \;\begin{tikzpicture}[baseline=(outerbasepoint)]
  \node[draw,inner sep=2pt] (box) {
    $
    \,\tikz[baseline=(basepoint)]{ 
      \path (0,3pt) coordinate (basepoint) (0,1pt) node[dot] {} (0,11pt) node[dot] {};
      \draw[](-8pt,-9pt) -- (0,1pt);
      \draw[](8pt,-9pt) -- (0,1pt);
      \draw[](0,1pt) -- (0,11pt);
      \draw[](0,11pt) -- (-8pt,21pt);
      \draw[](0,11pt) -- (8pt,21pt);
    }\,
    +
    \,\tikz[baseline=(basepoint)]{ 
      \path (0,3pt) coordinate (basepoint) (8pt,1pt) node[dot] {} (-8pt,11pt) node[dot] {};
      \draw[](-8pt,-9pt) -- (-8pt,11pt);
      \draw[](8pt,-9pt) -- (8pt,1pt);
      \draw[](8pt,1pt) -- (-8pt,11pt);
      \draw[](-8pt,11pt) -- (-8pt,21pt);
      \draw[](8pt,1pt) -- (8pt,21pt);
    }\,
    $};
  \draw (box.north) ++(6pt,0) -- ++(0,10pt) coordinate(top);
  \draw (box.north) ++(-6pt,0) -- ++(0,10pt);
  \draw (box.south) ++(6pt,0) -- ++(0,-10pt);
  \draw (box.south) ++(-6pt,0) -- ++(0,-10pt);
  \draw (box.south) ++(0,-18pt) node[draw,inner sep=2pt,fill=white] (lowerbox) {\mult};
  \draw (lowerbox.south) -- ++(0,-10pt) coordinate (bottom);
  \path (bottom) ++(0,40pt) coordinate (outerbasepoint);  
  \end{tikzpicture}\;.
\end{equation}
In both summands, we use the fact that $\graphI = 0$.
The sign on the first summand does not come from Remark~\ref{remark.signs2}, since the inside of the bottom box has cohomological degree $-2$ in $(\Frob_{1})^{\shriek}$; the sign there outside the diagram cancels the sign inside.  The sign on the second summand does stem from Remark~\ref{remark.signs2}: in $(\Frob_{1})^{\shriek}$, $\deg\bigl(\mult\bigr) = -1$.

Second, we claim:
\begin{equation}
  \label{Aderivative}
  \partial\left(
  \graphAbox
  \right) = 
  -
  \;\begin{tikzpicture}[baseline=(outerbasepoint)]
  \node[draw,inner sep=2pt] (box) {
    $
    \,\tikz[baseline=(basepoint)]{ 
      \path (0,3pt) coordinate (basepoint) (0,1pt) node[dot] {} (0,11pt) node[dot] {};
      \draw[](-8pt,-9pt) -- (0,1pt);
      \draw[](8pt,-9pt) -- (0,1pt);
      \draw[](0,1pt) -- (0,11pt);
      \draw[](0,11pt) -- (-8pt,21pt);
      \draw[](0,11pt) -- (8pt,21pt);
    }\,
    +
    \,\tikz[baseline=(basepoint)]{ 
      \path (0,3pt) coordinate (basepoint) (8pt,1pt) node[dot] {} (-8pt,11pt) node[dot] {};
      \draw[](-8pt,-9pt) -- (-8pt,11pt);
      \draw[](8pt,-9pt) -- (8pt,1pt);
      \draw[](8pt,1pt) -- (-8pt,11pt);
      \draw[](-8pt,11pt) -- (-8pt,21pt);
      \draw[](8pt,1pt) -- (8pt,21pt);
    }\,
    $
    };
  \draw (box.south) ++(6pt,0) -- ++(0,-10pt) coordinate (bottom);
  \draw (box.south) ++(-6pt,0) -- ++(0,-10pt) coordinate (bottom);
  \draw (box.north) ++(6pt,0) -- ++(0pt,10pt);
  \draw (box.north) ++(-6pt,0) -- ++(0pt,10pt);
  \draw (box.north) ++(-0pt,18pt) node[draw,inner sep=2pt,fill=white] (upperbox) {\comult};
  \draw (upperbox.north) -- ++(0,10pt) ++(18pt,0) coordinate (top);
  \path (bottom) ++(0,40pt) coordinate (outerbasepoint);  
  \end{tikzpicture}\;
  \;+\;
  \;\begin{tikzpicture}[baseline=(outerbasepoint)]
  \node[draw,inner sep=2pt] (box) {
    $
    -  \,\tikz[baseline=(basepoint)]{ 
    \path (0,3pt) coordinate (basepoint) (8pt,1pt) node[dot] {} (0,11pt) node[dot] {};
    \draw[](-12pt,-9pt) -- (0,11pt);
    \draw[](0pt,-9pt) -- (8pt,1pt);
    \draw[](12pt,-9pt) -- (8pt,1pt);
    \draw[](8pt,1pt) -- (0,11pt);
    \draw[](0,11pt) -- (0,21pt);
  }\,
  -
  \,\tikz[baseline=(basepoint)]{ 
    \path (0,3pt) coordinate (basepoint) (-8pt,1pt) node[dot] {} (0,11pt) node[dot] {};
    \draw[](12pt,-9pt) -- (0,11pt);
    \draw[](-12pt,-9pt) -- (-8pt,1pt);
    \draw[](0pt,-9pt) -- (-8pt,1pt);
    \draw[](-8pt,1pt) -- (0,11pt);
    \draw[](0,11pt) -- (0,21pt);
  }\,
    $
    };
  \draw (box.north) ++(0pt,0) -- ++(0,10pt) coordinate(top);
  \draw (box.south) ++(0pt,0) -- ++(0,-10pt);
  \draw (box.south) ++(12pt,0) -- ++(0,-10pt);
  \draw (box.south) ++(+6pt,-18pt) node[draw,inner sep=2pt,fill=white] (lowerbox) {\mult};
  \draw (lowerbox.south) -- ++(0,-10pt) ++(-18pt,0pt) coordinate (bottom);
  \draw (box.south) ++(-12pt,0) -- (bottom);
  \path (bottom) ++(0,40pt) coordinate (outerbasepoint);  
  \end{tikzpicture}\;.
 \end{equation}
 This time both summands collect a sign from Remark~\ref{remark.signs2}, since the insides of both bottom boxes have degree $-1$ in $(\Frob_1)^\shriek$.
As a check-sum, the reader may want to verify that in both equations~(\ref{Dderivative}) and~(\ref{Aderivative}) we have $\partial^{2} = 0$.

Finally, we will need later to know the piece of $(\Frob_{1})^{\shriek}[-1]$ with $\#\mult + \#\comult = 4$ and genus $\beta = 2$.  Since $\#\mult = \beta + m-1 \geq \beta$ and $\#\comult = \beta + n-1 \geq \beta$, the only way to have a composition of $\mult$ and $\comult$ with $\bigl(\#\mult + \#\comult,\beta\bigr) = (4,2)$ is if $\bigl(\#\mult,\#\comult\bigr) = (2,2)$ and $(m,n) = (1,1)$.  Recalling that $
  \graphI=0$, the only graph with  $(m,n,\beta) = (1,1,2)$ is 
  \graphB.
    In particular, the (anti)symmetry of the $\mult$ and $\comult$ implies that this graph is equal up to sign to any of its permutations. 
    (Given these signs, it is worth checking that $\graphB \neq 0$ in $\Ff(T[1])$.  The $(\bS_{1}^{\op}\times\bS_{3})$-module $\Ff^{(2)}(\comult)$ splits as $\tikz \draw (0,0) rectangle (4pt,4pt);\otimes\bigl(\tikz \draw (0,0) rectangle (4pt,4pt) (4pt,0pt) rectangle (8pt,4pt) (0pt,-4pt)  rectangle  (4pt,0pt) ; \oplus \tikz \draw (0,0) rectangle (4pt,4pt) (4pt,0pt) rectangle (8pt,4pt) (8pt,0pt)  rectangle  (12pt,4pt) ;\bigr)$, whereas $\Ff^{(2)}(\mult) \cong \bigl(\tikz \draw (0,0) rectangle (4pt,4pt) (4pt,0pt) rectangle (8pt,4pt) (0pt,-4pt)  rectangle  (4pt,0pt) ; \oplus \tikz \draw (0,0) rectangle (4pt,4pt) (0pt,4pt) rectangle (4pt,8pt) (0pt,-4pt)  rectangle  (4pt,0pt) ;\bigr)\otimes \tikz \draw (0,0) rectangle (4pt,4pt);$.  The $(\bS_{1}^{\op}\times \bS_{1})$-module $\Ff^{(2)}(\mult) \otimes_{\bS_{3}} \Ff^{2}(\comult)$ is thus $1$-dimensional, and $\graphB$ is its generator.)
    
   Since $\graphB$ spans a homogeneous piece of $\Ff(T[1])$, it either is or is not in $(\Frob_{1})^{\shriek}$, depending on whether its derivative in $\B(\Ff(T[1]))$ is in $\sh\Frob_{1} = \B((\Frob_{1})^{\shriek})$.  In fact:
  \begin{equation} \label{Bderivative}
  \partial\left(
  \graphBbox  
  \right) = 
  \,\tikz[baseline=(outerbasepoint)]{
    \node[draw,rectangle,inner sep=2pt] (main) {\graphD};
    \draw (main.south) -- ++(0,-8pt) coordinate(bottom);
    \draw (main.north) ++(-6pt,0) -- ++(0,10pt);
    \draw (main.north) ++(6pt,0) -- ++(0,10pt);
    \draw (main.north) ++(0,18pt) node[draw,inner sep=2pt,fill=white] (upperbox) {\comult};
    \draw (upperbox.north) -- ++(0,10pt) coordinate (top);
    \path (bottom) -- coordinate (outerbasepoint) (top);
  }\,
  \;-\;
  \,\tikz[baseline=(outerbasepoint)]{
    \node[draw,rectangle,inner sep=2pt] (main) {\graphA};
    \draw (main.north) -- ++(0,8pt) coordinate (top);
    \draw (main.south) ++(-6pt,0) -- ++(0,-10pt);
    \draw (main.south) ++(6pt,0) -- ++(0,-10pt);
    \draw (main.south) ++(0,-18pt) node[draw,inner sep=2pt,fill=white] (lowerbox) {\mult};
    \draw (lowerbox.south) -- ++(0,-10pt) coordinate (bottom);
    \path (bottom) -- coordinate (outerbasepoint) (top);
  }\,
  \;+\; \frac13\;\,
  \begin{tikzpicture}[baseline=(outerbasepoint)]
    \path coordinate (outerbasepoint) ++(0,0)
      +(0,-25pt) node[draw,inner sep = 2pt] (lowerbox) {
        $
        \,\tikz[baseline=(basepoint)]{ 
          \path (0,3pt) coordinate (basepoint) (0,1pt) node[dot] {} (-8pt,11pt) node[dot] {};
          \draw[](0,-9pt) -- (0,1pt);
          \draw[](0,1pt) -- (-8pt,11pt);
          \draw[](-8pt,11pt) -- (-12pt,21pt);
          \draw[](-8pt,11pt) -- (0pt,21pt);
          \draw[](0,1pt) -- (12pt,21pt);
        }\,
        -
        \,\tikz[baseline=(basepoint)]{ 
          \path (0,3pt) coordinate (basepoint) (0,1pt) node[dot] {} (8pt,11pt) node[dot] {};
          \draw[](0,-9pt) -- (0,1pt);
          \draw[](0,1pt) -- (8pt,11pt);
          \draw[](0,1pt) -- (-12pt,21pt);
          \draw[](8pt,11pt) -- (0pt,21pt);
          \draw[](8pt,11pt) -- (12pt,21pt);
        }\,
        $}
      +(0,25pt) node[draw,inner sep = 2pt] (upperbox) {
        $
    -  \,\tikz[baseline=(basepoint)]{ 
    \path (0,3pt) coordinate (basepoint) (8pt,1pt) node[dot] {} (0,11pt) node[dot] {};
    \draw[](-12pt,-9pt) -- (0,11pt);
    \draw[](0pt,-9pt) -- (8pt,1pt);
    \draw[](12pt,-9pt) -- (8pt,1pt);
    \draw[](8pt,1pt) -- (0,11pt);
    \draw[](0,11pt) -- (0,21pt);
  }\,
  -
  \,\tikz[baseline=(basepoint)]{ 
    \path (0,3pt) coordinate (basepoint) (-8pt,1pt) node[dot] {} (0,11pt) node[dot] {};
    \draw[](12pt,-9pt) -- (0,11pt);
    \draw[](-12pt,-9pt) -- (-8pt,1pt);
    \draw[](0pt,-9pt) -- (-8pt,1pt);
    \draw[](-8pt,1pt) -- (0,11pt);
    \draw[](0,11pt) -- (0,21pt);
  }\,
    $
      };
    \draw (lowerbox.south) -- ++(0,-10pt);
    \draw (upperbox.north) -- ++(0,10pt);
    \path (lowerbox.north) +(12pt,0) coordinate (la) +(-12pt,0) coordinate (lb) (upperbox.south)  +(12pt,0) coordinate (ua) +(-12pt,0) coordinate (ub);
    \draw (lowerbox.north) -- (ub);
    \draw (la) -- (upperbox.south); \draw (lb) -- (ua);
  \end{tikzpicture}
  \;\in\; \sh\Frob_{1}.
  \end{equation}
  The sign on the second summand occurs because $\mult \in (\Frob_{1})^{\shriek}$ has cohomological degree $-1$.
Lemma~\ref{lemma.cobarinductive} implies therefore that in weight $\#\mult+ \#\comult = 4$, $(\Frob_{1})^{\shriek}[-1]$ looks like:
$$ (\beta = 0 \text{ part}) \;\oplus\; (\beta = 1 \text{ part}) \;\oplus \;
\underset{\tikz \draw (0,0) rectangle (4pt,4pt);\,\otimes\,\tikz \draw (0,0) rectangle (4pt,4pt);}{
  \graphBbox  
}$$

\section{Some obstruction theory and a proof of Theorem \texorpdfstring{\ref{mainthm.dioperads}}{1}} \label{section.mainthm1}

The goal of this  section is to prove Theorem~\ref{mainthm.dioperads}.  In the proof, and also in the proof of Theorem~\ref{mainthm.properads}, 
will use the following reasonably well known \define{basic facts of obstruction theory}.  Let $P$ be a quasifree (di/pr)operad with a well-ordered set of generators $f$ (each of which is  homogeneous  of homological degree $\deg(f)$), and such that for each generator $f$ of $P$, $\partial(f)$ is a composition of generators of $P$ that are strictly earlier than $f$ in the well-ordering.  For each generator $f$, let $P_{< f}$ denote the sub(di/pr)operad of $P$ generated by the generators that are strictly earlier than $f$, and $P_{\leq f}$ the sub(di/pr)operad generated $P_{<f}$ and $f$; the condition then guarantees that $P_{<f}$ and $P_{\leq f}$ are dg (di/pr)operads and that $\partial(f) \in P_{<f}$.  The basic facts that we will use provide ways to study, for any (di/pr)operad $Q$, the space of homomorphisms $\eta : P \to Q$.

\begin{enumerate}
  \item Homomorphisms $\eta: P \to Q$ may be built and studied inductively.  Suppose we have defined a homomorphism $\eta_{<f} : P_{<f} \to Q$, which we want to extend to $P_{\leq f}$.  Then $\eta(\partial f)\in Q$ is closed of cohomological degree $\deg(f)+1$.  The \define{obstruction to defining $f$} is the class of $\eta(\partial f)$ in $\H^{{\bullet}}(Q)$.  The first basic fact of obstruction theory is that $\eta(f)$ can be defined, and therefore the induction can be continued, if and only if the obstruction $[\eta(\partial f)]\in \H^{\bullet}(Q)$ vanishes.  In particular, maps $P \to Q$ are easy to construct whenever the cohomology groups of $Q$ vanish in degrees $\deg(f)+1$ for generators $f$ of $P$.
  \item There are, of course, many choices for $\eta(f)$ --- as many as there are closed elements of $Q$ (with the appropriate number of inputs and outputs and transforming appropriately under the $\bS$-actions) of degree $\deg (f)$.  Different choices might lead to later steps of the induction succeeding or failing.  The second basic fact says that whether a later step succeeds is not affected by changing $\eta(f)$ by something exact, so that the ``true'' space of choices for $\eta(f)$ (provided the obstruction is exact) is a torsor for $\H^{{\deg(f)}}(Q)$.
  
  To prove this, note first that two different extensions of $\eta_{<f} : P_{<f}\to Q$ whose values on $f$ differ by something exact can be connected by a ``linear'' path $P_{\geq f} \to Q\otimes \bR[\Delta^{1}] = Q\otimes\bR[t]$ (and conversely one may check that the endpoints of any path $P_{\leq f} \to Q\otimes \bR[\Delta^{1}]$ which is constant on $P_{<f}$ assigns to $f$ values that differ by something exact).  Suppose then by induction that we have a ``curve'' $\eta_{<f'} : P_{< f'} \to Q\otimes \bR[\Delta^{1}]$, and that we choose an extension of its left endpoint to $\eta_{\leq f'}|_{t=0} : P_{\leq f'} \to Q$.  It suffices to define $\eta_{\leq f'} : P_{\leq f'} \to Q\otimes \bR[\Delta^{1}]$ extending $\eta_{<f'}$ with the given boundary condition.  
  To find such an extension $\eta_{\leq f'}$, one may  use~\cite[Proposition~37]{MR2572248}, which implies that the inclusion $P_{<f'}\mono P_{\leq f'}$ is a cofibration; the projection $Q\otimes \bR[\Delta^{1}] \to Q$ that sets $t = 0$ is a surjective quasiisomorphism, i.e.\ an acyclic fibration; the extension then exists by the left lifting property.
  
  \item More generally, a similar analysis shows that $\pi_{j}\{$space of choices for $\eta_{\leq f}$ extending $\eta_{<f}\} = \H^{{\deg(f)-j}}(Q)$, provided $\pi_{i} = 0$ for $0<i<j$ and $\{$space of choices$\}$ is nonempty and connected.  
   In particular, we may immediately conclude that the space of choices for $\eta(f)$ is contractible whenever $\H^{{\deg(f) - j}}(Q) = 0$ for all $j \geq -1$.
\end{enumerate}

We will henceforth denote by $\sh^{\di}\Frob_{1}$  the cofibrant replacement \emph{as a dioperad} of $\Frob_{1}$ coming from Corollary~\ref{cor.shFrob}, and  $\sh^{\pr}\Frob_{1}$ will denote the cofibrant replacement \define{as a properad}.  There is a canonical inclusion $\sh^{\di}\Frob_{1} \mono \sh^{\pr}\Frob_{1}$, given by including the generators of $\sh^{\di}\Frob_{1}$ as the genus $\beta = 0$ generators of $\sh^{\pr}\Frob_{1}$.  
We will often write the wedge product of de Rham forms; we will omit the ``$\wedge$'' sign, writing ``$\alpha \beta$'' for $\alpha\wedge \beta$.

\begin{proof}[Proof of Theorem~\ref{mainthm.dioperads}]
  In the language of Section~\ref{section.diproperads}, Theorem~\ref{mainthm.dioperads} asserts that the space of homomorphisms $\sh^{\di}\Frob_{1} \to \qloc$ inducing the standard $1$-shifted Frobenius algebra structure $\Frob_{1} \to \End(\H^{\bullet}(S^{1}))$ is contractible.  According to the basic facts of obstruction theory discussed at the beginning of this section, to study the space of such homomorphisms it is enough to look at generators of $\sh^{\di}\Frob_{1}$ with certain cohomological degrees depending on the cohomology of $\qloc$.  In Proposition~\ref{prop.Hqloc} we computed that $\H^{\bullet}\qloc(m,n)$ has homology only in degrees $n-1$ and $n$.  On the other hand, in Corollary~\ref{cor.shFrob} we saw that a generator of $\sh^{\di}\Frob_{1}$ with $m$ inputs has cohomological degree $2-m$.  
  
  Thus the only generators for which there might be an obstruction are those for which $2-m +1 = n-1$ or $n$, i.e.\ $m+n = 3$ or $4$.  The only way there could be inequivalent choices are when $m+n = 2$ or $3$, and contributions to the $j$th homotopy group occur when $m+n = 2-j$ or $3-j$.  But $m+n\geq 3$, so we already conclude that the space of homomorphisms has no higher homotopy groups.
  
  We have therefore reduced the problem to analyzing those generators of $\sh^{\di}\Frob_{1}$ with $m+n = 3$ or $4$.  The only generators with $m+n = 3$ are the multiplication and comultiplication
  $$\,
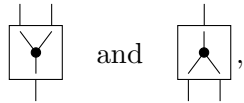
\tikz[baseline=(main.base)]{
    \node[draw,rectangle,inner sep=2pt] (main) {\mult};
    \draw (main.south) -- ++(0,-8pt);
    \draw (main.north) ++(-6pt,0) -- ++(0,8pt);
    \draw (main.north) ++(6pt,0) -- ++(0,8pt);
  }\, \quad \text{and} \quad 
  \,\tikz[baseline=(main.base)]{
    \node[draw,rectangle,inner sep=2pt] (main) {\comult};
    \draw (main.north) -- ++(0,8pt);
    \draw (main.south) ++(-6pt,0) -- ++(0,-8pt);
    \draw (main.south) ++(6pt,0) -- ++(0,-8pt);
  }\,,
  $$
  and the only generators with $m+n = 4$ are the homotopies imposing associativity, coassociativity, and the Frobenius axiom
  $$
  \,
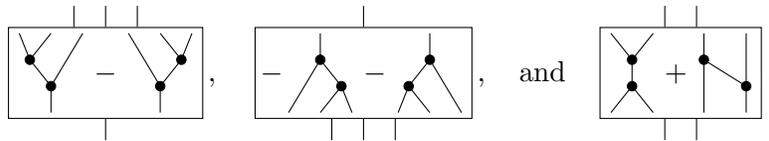
\tikz[baseline=(main.base)]{
    \node[draw,rectangle,inner sep=2pt] (main) {$
        \,\tikz[baseline=(basepoint)]{ 
    \path (0,3pt) coordinate (basepoint) (0,1pt) node[dot] {} (-8pt,11pt) node[dot] {};
    \draw[](0,-9pt) -- (0,1pt);
    \draw[](0,1pt) -- (-8pt,11pt);
    \draw[](-8pt,11pt) -- (-12pt,21pt);
    \draw[](-8pt,11pt) -- (0pt,21pt);
    \draw[](0,1pt) -- (12pt,21pt);
  }\,
  -
  \,\tikz[baseline=(basepoint)]{ 
    \path (0,3pt) coordinate (basepoint) (0,1pt) node[dot] {} (8pt,11pt) node[dot] {};
    \draw[](0,-9pt) -- (0,1pt);
    \draw[](0,1pt) -- (8pt,11pt);
    \draw[](0,1pt) -- (-12pt,21pt);
    \draw[](8pt,11pt) -- (0pt,21pt);
    \draw[](8pt,11pt) -- (12pt,21pt);
  }\,
    $};
    \draw (main.south) -- ++(0,-8pt);
    \draw (main.north) ++(-12pt,0) -- ++(0,8pt);
    \draw (main.north) ++(0pt,0) -- ++(0,8pt);
    \draw (main.north) ++(12pt,0) -- ++(0,8pt);
  }\,
  ,\quad
  \,\tikz[baseline=(main.base)]{
    \node[draw,rectangle,inner sep=2pt] (main) {$
    -  \,\tikz[baseline=(basepoint)]{ 
    \path (0,3pt) coordinate (basepoint) (8pt,1pt) node[dot] {} (0,11pt) node[dot] {};
    \draw[](-12pt,-9pt) -- (0,11pt);
    \draw[](0pt,-9pt) -- (8pt,1pt);
    \draw[](12pt,-9pt) -- (8pt,1pt);
    \draw[](8pt,1pt) -- (0,11pt);
    \draw[](0,11pt) -- (0,21pt);
  }\,
  -
  \,\tikz[baseline=(basepoint)]{ 
    \path (0,3pt) coordinate (basepoint) (-8pt,1pt) node[dot] {} (0,11pt) node[dot] {};
    \draw[](12pt,-9pt) -- (0,11pt);
    \draw[](-12pt,-9pt) -- (-8pt,1pt);
    \draw[](0pt,-9pt) -- (-8pt,1pt);
    \draw[](-8pt,1pt) -- (0,11pt);
    \draw[](0,11pt) -- (0,21pt);
  }\,
    $};
    \draw (main.north) -- ++(0,8pt);
    \draw (main.south) ++(-12pt,0) -- ++(0,-8pt);
    \draw (main.south) ++(0pt,0) -- ++(0,-8pt);
    \draw (main.south) ++(12pt,0) -- ++(0,-8pt);
  }\,
  ,\quad
  \text{and}\quad
  \,\tikz[baseline=(main.base)]{
    \node[draw,rectangle,inner sep=2pt] (main) {$
          \,\tikz[baseline=(basepoint)]{ 
    \path (0,3pt) coordinate (basepoint) (0,1pt) node[dot] {} (0,11pt) node[dot] {};
    \draw[](-8pt,-9pt) -- (0,1pt);
    \draw[](8pt,-9pt) -- (0,1pt);
    \draw[](0,1pt) -- (0,11pt);
    \draw[](0,11pt) -- (-8pt,21pt);
    \draw[](0,11pt) -- (8pt,21pt);
  }\,
  +
  \,\tikz[baseline=(basepoint)]{ 
    \path (0,3pt) coordinate (basepoint) (8pt,1pt) node[dot] {} (-8pt,11pt) node[dot] {};
    \draw[](-8pt,-9pt) -- (-8pt,11pt);
    \draw[](8pt,-9pt) -- (8pt,1pt);
    \draw[](8pt,1pt) -- (-8pt,11pt);
    \draw[](-8pt,11pt) -- (-8pt,21pt);
    \draw[](8pt,1pt) -- (8pt,21pt);
  }\,
    $};
    \draw (main.south) ++(-6pt,0) -- ++(0pt,-8pt);
    \draw (main.south) ++(6pt,0) -- ++(0pt,-8pt);
    \draw (main.north) ++(-6pt,0) -- ++(0,8pt);
    \draw (main.north) ++(6pt,0) -- ++(0,8pt);
  }\,
  .
  $$

  Since $\mult$ and $\comult$ act on $\H^{\bullet}(S^{1})$ with the given relations, these five generators can certainly be lifted to $\End(\Omega^{\bullet}(S^{1}))$.  The  question is to see that they can be lifted quasilocally.
  The multiplication $\mult$, of course, can be lifted to wedge multiplication, which is a local operation in $\qloc(2,1)^{0}$ and for which
  $$ \partial \left(   \,\tikz[baseline=(main.base)]{
    \node[draw,rectangle,inner sep=2pt] (main) {$
        \,\tikz[baseline=(basepoint)]{ 
    \path (0,3pt) coordinate (basepoint) (0,1pt) node[dot] {} (-8pt,11pt) node[dot] {};
    \draw[](0,-9pt) -- (0,1pt);
    \draw[](0,1pt) -- (-8pt,11pt);
    \draw[](-8pt,11pt) -- (-12pt,21pt);
    \draw[](-8pt,11pt) -- (0pt,21pt);
    \draw[](0,1pt) -- (12pt,21pt);
  }\,
  -
  \,\tikz[baseline=(basepoint)]{ 
    \path (0,3pt) coordinate (basepoint) (0,1pt) node[dot] {} (8pt,11pt) node[dot] {};
    \draw[](0,-9pt) -- (0,1pt);
    \draw[](0,1pt) -- (8pt,11pt);
    \draw[](0,1pt) -- (-12pt,21pt);
    \draw[](8pt,11pt) -- (0pt,21pt);
    \draw[](8pt,11pt) -- (12pt,21pt);
  }\,
    $};
    \draw (main.south) -- ++(0,-8pt);
    \draw (main.north) ++(-12pt,0) -- ++(0,8pt);
    \draw (main.north) ++(0pt,0) -- ++(0,8pt);
    \draw (main.north) ++(12pt,0) -- ++(0,8pt);
  }\,
 \right) = 0 \text{ in } \qloc(3,1), $$
  since wedge multiplication of smooth forms is strictly associative.
  
  By Proposition~\ref{prop.Hqloc}, $\H^{0}(\qloc(2,1))$ is one-dimensional; since wedge multiplication represents a non-zero class in $\End(\H^{\bullet}(S^{1}))$, the map $\H^{\bullet}\qloc \to \End(\H^{\bullet}(S^{1}))$ is non-zero in cohomological degree~$0$, and hence an inclusion.  In particular, any two lifts of $\mult$ to $\qloc$ represent the same class in $\H^{0}(\qloc)$.  Said another way, the space of choices for $\,\tikz[baseline=(main.base)]{
    \node[draw,rectangle,inner sep=2pt] (main) {\mult};
    \draw (main.south) -- ++(0,-8pt);
    \draw (main.north) ++(-6pt,0) -- ++(0,8pt);
    \draw (main.north) ++(6pt,0) -- ++(0,8pt);
  }\,$ is connected, and hence contractible since we saw already that there is no higher homotopy.  The same argument shows that the space of choices for $\,\tikz[baseline=(main.base)]{
    \node[draw,rectangle,inner sep=2pt] (main) {\comult};
    \draw (main.north) -- ++(0,8pt);
    \draw (main.south) ++(-6pt,0) -- ++(0,-8pt);
    \draw (main.south) ++(6pt,0) -- ++(0,-8pt);
  }\,$ is contractible if nonempty.
  
  Wedge multiplication is local but has distributional integral kernel.  We could have made a different choice for $\,\tikz[baseline=(main.base)]{
    \node[draw,rectangle,inner sep=2pt] (main) {\mult};
    \draw (main.south) -- ++(0,-8pt);
    \draw (main.north) ++(-6pt,0) -- ++(0,8pt);
    \draw (main.north) ++(6pt,0) -- ++(0,8pt);
  }\,$: we could have chosen it to have integral kernel that is smooth.  By the second basic fact of obstruction theory, making such a choice won't spoil the exactness 
  \begin{equation} \label{eqn.assoc}
   \partial \left(   \,\tikz[baseline=(main.base)]{
    \node[draw,rectangle,inner sep=2pt] (main) {$
        \,\tikz[baseline=(basepoint)]{ 
    \path (0,3pt) coordinate (basepoint) (0,1pt) node[dot] {} (-8pt,11pt) node[dot] {};
    \draw[](0,-9pt) -- (0,1pt);
    \draw[](0,1pt) -- (-8pt,11pt);
    \draw[](-8pt,11pt) -- (-12pt,21pt);
    \draw[](-8pt,11pt) -- (0pt,21pt);
    \draw[](0,1pt) -- (12pt,21pt);
  }\,
  -
  \,\tikz[baseline=(basepoint)]{ 
    \path (0,3pt) coordinate (basepoint) (0,1pt) node[dot] {} (8pt,11pt) node[dot] {};
    \draw[](0,-9pt) -- (0,1pt);
    \draw[](0,1pt) -- (8pt,11pt);
    \draw[](0,1pt) -- (-12pt,21pt);
    \draw[](8pt,11pt) -- (0pt,21pt);
    \draw[](8pt,11pt) -- (12pt,21pt);
  }\,
    $};
    \draw (main.south) -- ++(0,-8pt);
    \draw (main.north) ++(-12pt,0) -- ++(0,8pt);
    \draw (main.north) ++(0pt,0) -- ++(0,8pt);
    \draw (main.north) ++(12pt,0) -- ++(0,8pt);
  }\,
 \right) 
 \want
   \;-\;
  \,\tikz[baseline=(somewhere)]{ 
    \path (0,10pt) coordinate (somewhere);
    \path (0pt,0pt) node[draw,inner sep=2pt,fill=white] (a) {\mult} (-10pt,25pt) node[draw,inner sep=2pt,fill=white] (b) {\mult};
    \draw[](b) -- (a);
    \draw[](a) -- (0,-20pt);
    \draw[](a) -- (20pt,45pt);
    \draw[](b) -- (-20pt,45pt);
    \draw[](b) -- (0pt,45pt);
  }\,
  \;+\;
  \,\tikz[baseline=(somewhere)]{ 
    \path (0,10pt) coordinate (somewhere);
    \path (0pt,0pt) node[draw,inner sep=2pt,fill=white] (a) {\mult} (10pt,25pt) node[draw,inner sep=2pt,fill=white] (b) {\mult};
    \draw[](b) -- (a);
    \draw[](a) -- (0,-20pt);
    \draw[](a) -- (-20pt,45pt);
    \draw[](b) -- (20pt,45pt);
    \draw[](b) -- (0pt,45pt);
  }\,
 = 0 \text{ in } \H^{\bullet}\qloc(3,1), \end{equation}
 although it may no longer vanish identically at the chain level.
 
 Let us choose therefore a smooth kernel $\psi_{\epsilon}(x,y,z) \in \Omega^{2}\bigl(\bR_{>0}\times (S^{1})^{\times 3}\bigr)$, where $$\,\tikz[baseline=(main.base)]{
    \node[draw,rectangle,inner sep=2pt] (main) {\mult};
    \draw (main.south) -- ++(0,-8pt);
    \draw (main.north) ++(-6pt,0) -- ++(0,8pt);
    \draw (main.north) ++(6pt,0) -- ++(0,8pt);
  }\,: \alpha \mapsto \int_{x,y\in S^{1}} \psi_{\epsilon}(x,y,-)\, \alpha(x,y), \quad\quad \alpha \in \Omega^{\bullet}(S^{1})\otimes \Omega^{\bullet}(S^{1}) = \Omega^{\bullet}(S^{1}\times S^{1})$$ is a closed quasilocal operation with the same homology class as wedge multiplication.  By averaging, we can assume $\psi_{\epsilon}$ is invariant under permutations of all three $S^{1}$-variables.  We can choose to lift the comultiplication to
  $$ \,\tikz[baseline=(main.base)]{
    \node[draw,rectangle,inner sep=2pt] (main) {\comult};
    \draw (main.north) -- ++(0,8pt);
    \draw (main.south) ++(-6pt,0) -- ++(0,-8pt);
    \draw (main.south) ++(6pt,0) -- ++(0,-8pt);
  }\, : \alpha \mapsto \int_{z\in S^{1}} \psi_{\epsilon}(-,-,z)\, \alpha(z). $$
  With this choice, the vanishing of the obstructions corresponding to coassociativity and Frobenius  axioms
  $$ \partial \left(  \,\tikz[baseline=(main.base)]{
    \node[draw,rectangle,inner sep=2pt] (main) {$
     - \,\tikz[baseline=(basepoint)]{ 
    \path (0,3pt) coordinate (basepoint) (8pt,1pt) node[dot] {} (0,11pt) node[dot] {};
    \draw[](-12pt,-9pt) -- (0,11pt);
    \draw[](0pt,-9pt) -- (8pt,1pt);
    \draw[](12pt,-9pt) -- (8pt,1pt);
    \draw[](8pt,1pt) -- (0,11pt);
    \draw[](0,11pt) -- (0,21pt);
  }\,
  -
  \,\tikz[baseline=(basepoint)]{ 
    \path (0,3pt) coordinate (basepoint) (-8pt,1pt) node[dot] {} (0,11pt) node[dot] {};
    \draw[](12pt,-9pt) -- (0,11pt);
    \draw[](-12pt,-9pt) -- (-8pt,1pt);
    \draw[](0pt,-9pt) -- (-8pt,1pt);
    \draw[](-8pt,1pt) -- (0,11pt);
    \draw[](0,11pt) -- (0,21pt);
  }\,
    $};
    \draw (main.north) -- ++(0,8pt);
    \draw (main.south) ++(-12pt,0) -- ++(0,-8pt);
    \draw (main.south) ++(0pt,0) -- ++(0,-8pt);
    \draw (main.south) ++(12pt,0) -- ++(0,-8pt);
  }\,
  \right) = 0 \text{ in }\H^{\bullet}\qloc(1,3)
  \quad
  \text{and}\quad
  \partial\left(
  \,\tikz[baseline=(main.base)]{
    \node[draw,rectangle,inner sep=2pt] (main) {$
          \,\tikz[baseline=(basepoint)]{ 
    \path (0,3pt) coordinate (basepoint) (0,1pt) node[dot] {} (0,11pt) node[dot] {};
    \draw[](-8pt,-9pt) -- (0,1pt);
    \draw[](8pt,-9pt) -- (0,1pt);
    \draw[](0,1pt) -- (0,11pt);
    \draw[](0,11pt) -- (-8pt,21pt);
    \draw[](0,11pt) -- (8pt,21pt);
  }\,
  +
  \,\tikz[baseline=(basepoint)]{ 
    \path (0,3pt) coordinate (basepoint) (8pt,1pt) node[dot] {} (-8pt,11pt) node[dot] {};
    \draw[](-8pt,-9pt) -- (-8pt,11pt);
    \draw[](8pt,-9pt) -- (8pt,1pt);
    \draw[](8pt,1pt) -- (-8pt,11pt);
    \draw[](-8pt,11pt) -- (-8pt,21pt);
    \draw[](8pt,1pt) -- (8pt,21pt);
  }\,
    $};
    \draw (main.south) ++(-6pt,0) -- ++(0pt,-8pt);
    \draw (main.south) ++(6pt,0) -- ++(0pt,-8pt);
    \draw (main.north) ++(-6pt,0) -- ++(0,8pt);
    \draw (main.north) ++(6pt,0) -- ++(0,8pt);
  }\,
  \right) = 0 \text{ in }\H^{\bullet}\qloc(2,2)
  $$
  follow from equation~(\ref{eqn.assoc}), by interpreting all three as asserting the quasilocal exactness of forms on $(S^1)^{\times 4}$ which are related by the $\bS_4$ action.
\end{proof}

The proof of Theorem~\ref{mainthm.properads} (in particular, equations \ref{eqn.answer for frobeniator}, \ref{diagrammaticcoassocdefn}, and \ref{half of the coassociator}) will contain more explicit formulas also verifying that these obstructions vanish.

\section{Some calculus and a proof of Theorem \texorpdfstring{\ref{mainthm.properads}}{2}} \label{section.mainthm2}

We turn now to Theorem~\ref{mainthm.properads}, which asserts that the space of homomorphisms $\sh^{\pr}\Frob_{1}\to \qloc$ inducing the appropriate homomorphism to $\End(\H^{\bullet}(S^{1}))$ is empty.
Our argument involves a fair amount of explicit calculation, and so we  record first a few conventions.  As in Section~\ref{section.mainthm1}, we write the wedge product of forms simply by concatenation.  The total de Rham differential is~$\d$.  We will write essentially all operations in terms of their integral kernels, and identify $\Omega^{\bullet}\bigl(S^{1}\bigr)^{\otimes m}$ with $\Omega^{\bullet}\bigl((S^{1})^{\times m}\bigr)$.  Thus a typical operation $\Psi \in \Omega^{\bullet}\bigl((S^{1})^{\times 2}\bigr) \to \Omega^{\bullet}\bigl((S^{1})^{\times 2}\bigr)$ with kernel $\psi$ might be
$$ \Psi : \alpha \mapsto \left( (x,y) \mapsto \int_{v\in S^{1}}\int_{w\in S^{1}} \psi(x,y,v,w)\,\alpha(v,w)\right). $$
It will be convenient not to write all the ``$\mapsto$''s.  We can safely call by ``$\alpha$'' the input of any operation, but we will need to track the coordinates at which the output is evaluated.  
  In diagrams, we will write the above formula as
$$ \tikz[baseline=(Gamma.base)] \draw node[draw,circle,inner sep=1pt] (Gamma) {$\psi$} (Gamma) -- ++(-.25,.5)  (Gamma) -- ++(.25,.5) (Gamma) -- ++(-.25,-.5) node[anchor=north] {$\scriptstyle x$}  (Gamma) -- ++(.25,-.5) node[anchor=north] {$\scriptstyle y$}; = \int_{x\in S^{1}}\int_{y\in S^{1}} \psi(x,y,v,w)\,\alpha(v,w). $$

The one operation that we will use without writing its integral kernel is wedge multiplication.  As a map $\Omega^{\bullet}\bigl((S^{1})^{\times 2}\bigr) \to \Omega^{\bullet}(S^{1})$, it just pulls back forms along the diagonal:
$$ \tikz[baseline=(Gamma.base)] \draw node[draw,circle,inner sep=1pt] (Gamma) {$\wedge$} (Gamma) -- ++(-.25,.5)  (Gamma) -- ++(.25,.5) (Gamma) -- ++(0,-.5) node[anchor=north] {$\scriptstyle x$} ; = \alpha(x,x) = \tikz[baseline=(Gamma.base)] \draw node[text=white] (Gamma) {$\wedge$} (Gamma)  ++(-.25,.5) -- ++(0,-.5) node[anchor=north] {$\scriptstyle x$}  (Gamma)  ++(.25,.5) -- ++(0,-.5) node[anchor=north] {$\scriptstyle x$}  ;. $$

If $\psi \in \Omega^{\bullet}\bigl((S^{1})^{\times 4}\bigr)$ has degree $d$, then the above operation $\tikz[baseline=(Gamma.base)] \draw node[draw,circle,inner sep=.5pt] (Gamma) {$\psi$} (Gamma) -- ++(-.25,.25)  (Gamma) -- ++(.25,.25) (Gamma) -- ++(-.25,-.25) (Gamma) -- ++(.25,-.25);$ has degree $d-2$.  We understand this degree shift by assigning the symbol ``$\int_{x}$'' degree $-1$.  This affects (indeed, effects) signs in formulas after enforcing Remark~\ref{remark.signs}:
$$ \int_{x}\int_{y} = -\int_{y}\int_{x}; \quad\quad \d \int = - \int \d; \quad\quad \int_{x} \varphi(y,\dots) = (-1)^{\deg(\varphi)}\varphi(y,\dots)\int_{x} $$
if $\varphi$ is homogeneous  and independent of $x$.

Fix once and for all an identification $S^{1} = \bR / \bZ$, so that all coordinates $x,y,\dots$ on $S^{1}$ are periodic with period $1$.  Given a compactly-supported form $\varphi \in \Omega^{\bullet}_{\mathrm{cpt}}\bigl((-\frac12,\frac12)\bigr)$, we will abuse notation and write $\varphi(x),\,x\in S^{1}$ for the pushforward of $\varphi$ along the inclusion $(-\frac12,\frac12) \to S^{1}$.  We always denote by $\epsilon \in \bR_{>0}$ the variable controlling quasilocality, and generally write a form $\varphi \in \Omega^{\bullet}\bigl(\bR_{>0}\times S^{1}\bigr)$ as~$\varphi_{\epsilon}(x)$.
We also denote by $\varphi_{\epsilon}(x), \, \epsilon \in \bR_{>0},\, x\in S^{1}$ the pushforward of a form  $\varphi \in \Omega^{\bullet}\bigl(\bR_{>0} \times (-\frac12,\frac12)\bigr)$ with compact support in the $(-\frac12,\frac12)$-direction.

We always let $\delta$ refer to the $\delta$-distribution \emph{thought of as a one-form} (so our $\delta(x)$ is what many would write as $\delta(x)\,\d x$).  With this convention, $\delta(ax) = \delta(x)$ for $a \in \bR_{>0}$, and $\delta(-x) = -\delta(x)$.  Signs are normalized by:
$$ \int_{x}\delta(x)\,\alpha(x,y,\dots) = \alpha(0,y,\dots). $$
Note that since $\delta$ is a one-form, if $\varphi$ is homogeneous then $\delta(x)\,\varphi(x,\dots) = (-1)^{\deg \varphi}\varphi(x,\dots)\,\delta(x)$.

The antiderivative of $\delta$ is the non-compactly-supported discontinuous function
$$ \Theta(x) = \begin{cases} 1, & x > 0, \\
0, & x < 0, \end{cases} $$
which when convolved against a smooth form gives:
$$ \int_{x}\Theta(x) \, \alpha(x,y,\dots) = \int_{x=0}^{\infty} \alpha(x,y,\dots). $$
Although the exterior product $\Theta(x)\,\delta(y)$ makes sense as a distributional form in two variables, the product $\Theta(x)\,\delta(x)$ is undefined.  As such, we will need to be quite careful when working with these distributional forms.

Note finally that
\begin{multline*}\left[ \d, \alpha \mapsto \int \psi \alpha \right] = \d \int \psi \,\alpha - (-1)^{\deg(\int\psi)} \int \psi \,\d \alpha  \\ = (-1)^{\deg \int}\int \d (\psi \, \alpha) -  (-1)^{\deg(\int\psi)} \int \psi \, \d \alpha = (-1)^{\deg \int}\int \d \psi\, \alpha, \end{multline*}
and so we can always ignore the $\alpha$ input when calculating derivatives.

\begin{proof}[Proof of Theorem~\ref{mainthm.properads}]
  Our strategy will be to find a non-exact obstruction.  We begin by analyzing the possible choices, by comparing Proposition~\ref{prop.Hqloc} with Corollary~\ref{cor.shFrob}: $\qloc(m,n)$ has cohomology only in degrees $n-1$ and $n$; a generator of $\sh^{\pr}\Frob_{1}$  with $m$ inputs and genus $\beta$ is in cohomological degree $2-(\beta+m)$.  Thus there can be obstructions only when $m+n+\beta = 3$ or $4$, and inequivalent choices only when $m+n+\beta = 2$ or $3$.  But $\beta \geq 0$ and $m+n \geq 2$.  At the end of Section~\ref{section.Koszul} we listed all generators of $\sh^{\pr}\Frob_{1}$ with $m+n+\beta \leq 4$.  There are none with $(m+n,\beta) = (2,0)$ or $(2,1)$ (the latter since $\graphI = 0$).  
  The only generator with $(m+n,\beta) = (2,2)$ is $\graphBbox$.\vspace{-12pt}  
  The only generators with $(m+n,\beta) = (3,0)$ are $\,\tikz[baseline=(main.base)]{
    \node[draw,rectangle,inner sep=2pt] (main) {\mult};
    \draw (main.south) -- ++(0,-8pt);
    \draw (main.north) ++(-6pt,0) -- ++(0,8pt);
    \draw (main.north) ++(6pt,0) -- ++(0,8pt);
  }\,$ and $\,\tikz[baseline=(main.base)]{
    \node[draw,rectangle,inner sep=2pt] (main) {\comult};
    \draw (main.north) -- ++(0,8pt);
    \draw (main.south) ++(-6pt,0) -- ++(0,-8pt);
    \draw (main.south) ++(6pt,0) -- ++(0,-8pt);
  }\,$.\vspace{-6pt}
  There are two generators with $(m+n,\beta) = (3,1)$, namely $ \graphDbox \text{ and }  \graphAbox$.  Finally, there are three generators with $(m+n,\beta) = (4,0)$:
    $$
  \,
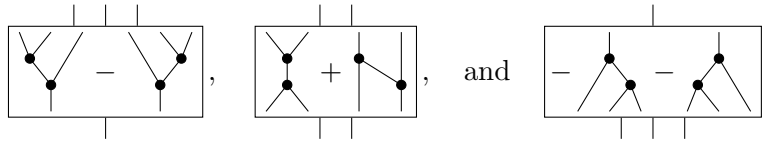
\tikz[baseline=(main.base)]{
    \node[draw,rectangle,inner sep=2pt] (main) {$
        \,\tikz[baseline=(basepoint)]{ 
    \path (0,3pt) coordinate (basepoint) (0,1pt) node[dot] {} (-8pt,11pt) node[dot] {};
    \draw[](0,-9pt) -- (0,1pt);
    \draw[](0,1pt) -- (-8pt,11pt);
    \draw[](-8pt,11pt) -- (-12pt,21pt);
    \draw[](-8pt,11pt) -- (0pt,21pt);
    \draw[](0,1pt) -- (12pt,21pt);
  }\,
  -
  \,\tikz[baseline=(basepoint)]{ 
    \path (0,3pt) coordinate (basepoint) (0,1pt) node[dot] {} (8pt,11pt) node[dot] {};
    \draw[](0,-9pt) -- (0,1pt);
    \draw[](0,1pt) -- (8pt,11pt);
    \draw[](0,1pt) -- (-12pt,21pt);
    \draw[](8pt,11pt) -- (0pt,21pt);
    \draw[](8pt,11pt) -- (12pt,21pt);
  }\,
    $};
    \draw (main.south) -- ++(0,-8pt);
    \draw (main.north) ++(-12pt,0) -- ++(0,8pt);
    \draw (main.north) ++(0pt,0) -- ++(0,8pt);
    \draw (main.north) ++(12pt,0) -- ++(0,8pt);
  }\,
  ,\quad
  \,\tikz[baseline=(main.base)]{
    \node[draw,rectangle,inner sep=2pt] (main) {$
          \,\tikz[baseline=(basepoint)]{ 
    \path (0,3pt) coordinate (basepoint) (0,1pt) node[dot] {} (0,11pt) node[dot] {};
    \draw[](-8pt,-9pt) -- (0,1pt);
    \draw[](8pt,-9pt) -- (0,1pt);
    \draw[](0,1pt) -- (0,11pt);
    \draw[](0,11pt) -- (-8pt,21pt);
    \draw[](0,11pt) -- (8pt,21pt);
  }\,
  +
  \,\tikz[baseline=(basepoint)]{ 
    \path (0,3pt) coordinate (basepoint) (8pt,1pt) node[dot] {} (-8pt,11pt) node[dot] {};
    \draw[](-8pt,-9pt) -- (-8pt,11pt);
    \draw[](8pt,-9pt) -- (8pt,1pt);
    \draw[](8pt,1pt) -- (-8pt,11pt);
    \draw[](-8pt,11pt) -- (-8pt,21pt);
    \draw[](8pt,1pt) -- (8pt,21pt);
  }\,
    $};
    \draw (main.south) ++(-6pt,0) -- ++(0pt,-8pt);
    \draw (main.south) ++(6pt,0) -- ++(0pt,-8pt);
    \draw (main.north) ++(-6pt,0) -- ++(0,8pt);
    \draw (main.north) ++(6pt,0) -- ++(0,8pt);
  }\,
  ,\quad
  \text{and}\quad
  \,\tikz[baseline=(main.base)]{
    \node[draw,rectangle,inner sep=2pt] (main) {$
   -   \,\tikz[baseline=(basepoint)]{ 
    \path (0,3pt) coordinate (basepoint) (8pt,1pt) node[dot] {} (0,11pt) node[dot] {};
    \draw[](-12pt,-9pt) -- (0,11pt);
    \draw[](0pt,-9pt) -- (8pt,1pt);
    \draw[](12pt,-9pt) -- (8pt,1pt);
    \draw[](8pt,1pt) -- (0,11pt);
    \draw[](0,11pt) -- (0,21pt);
  }\,
  -
  \,\tikz[baseline=(basepoint)]{ 
    \path (0,3pt) coordinate (basepoint) (-8pt,1pt) node[dot] {} (0,11pt) node[dot] {};
    \draw[](12pt,-9pt) -- (0,11pt);
    \draw[](-12pt,-9pt) -- (-8pt,1pt);
    \draw[](0pt,-9pt) -- (-8pt,1pt);
    \draw[](-8pt,1pt) -- (0,11pt);
    \draw[](0,11pt) -- (0,21pt);
  }\,
    $};
    \draw (main.north) -- ++(0,8pt);
    \draw (main.south) ++(-12pt,0) -- ++(0,-8pt);
    \draw (main.south) ++(0pt,0) -- ++(0,-8pt);
    \draw (main.south) ++(12pt,0) -- ++(0,-8pt);
  }\,
    .
  $$
    For all other generators, $m+n+\beta \geq 5$.  
    The basic facts of obstruction theory (along with the requirement that the homomorphism lift the $\Frob_{1}$ structure on $\H^{\bullet}(S^{1})$, as in Theorem~\ref{mainthm.dioperads}) assure that there are never homotopically-inequivalent choices available for any generator.  In particular, if an obstruction fails to vanish in cohomology for some set of choices, it will fail to vanish for any other choices as well. \vspace{-6pt}
  
  We might as well choose $\,\tikz[baseline=(main.base)]{
    \node[draw,rectangle,inner sep=2pt] (main) {\mult};
    \draw (main.south) -- ++(0,-8pt) ;
    \draw (main.north) ++(-6pt,0) -- ++(0,8pt);
    \draw (main.north) ++(6pt,0) -- ++(0,8pt);
  }\,$ to be the wedge multiplication of de Rham forms:
  $$\,\tikz[baseline=(main.base)]{
    \node[draw,rectangle,inner sep=2pt] (main) {\mult};
    \draw (main.south) -- ++(0,-8pt)  node[anchor=north] {$\scriptstyle x$};
    \draw (main.north) ++(-6pt,0) -- ++(0,8pt);
    \draw (main.north) ++(6pt,0) -- ++(0,8pt);
  }\, = \,\tikz[baseline=(main.base)]{
    \node[white,text=white,draw,rectangle,inner sep=2pt] (main) {\mult};
    \node[draw,circle,inner sep=1pt] (label) {$\wedge$};
    \path (main.south)  ++(0,-8pt) coordinate (a) (main.north) ++(-6pt,0)  ++(0,8pt)  coordinate (b) (main.north) ++(6pt,0) -- ++(0,8pt)  coordinate (c);
    \draw (label) -- (a)  node[anchor=north] {$\scriptstyle x$} (label) -- (b) (label) -- (c);
  }\, = \, \alpha(x,x).$$
    Since this is strictly associative, we can set
  \begin{equation} \label{eqn.assoc=0}
    \,\tikz[baseline=(main.base)]{
    \node[draw,rectangle,inner sep=2pt] (main) {$
        \,\tikz[baseline=(basepoint)]{ 
    \path (0,3pt) coordinate (basepoint) (0,1pt) node[dot] {} (-8pt,11pt) node[dot] {};
    \draw[](0,-9pt) -- (0,1pt);
    \draw[](0,1pt) -- (-8pt,11pt);
    \draw[](-8pt,11pt) -- (-12pt,21pt);
    \draw[](-8pt,11pt) -- (0pt,21pt);
    \draw[](0,1pt) -- (12pt,21pt);
  }\,
  -
  \,\tikz[baseline=(basepoint)]{ 
    \path (0,3pt) coordinate (basepoint) (0,1pt) node[dot] {} (8pt,11pt) node[dot] {};
    \draw[](0,-9pt) -- (0,1pt);
    \draw[](0,1pt) -- (8pt,11pt);
    \draw[](0,1pt) -- (-12pt,21pt);
    \draw[](8pt,11pt) -- (0pt,21pt);
    \draw[](8pt,11pt) -- (12pt,21pt);
  }\,
    $};
    \draw (main.south) -- ++(0,-8pt);
    \draw (main.north) ++(-12pt,0) -- ++(0,8pt);
    \draw (main.north) ++(0pt,0) -- ++(0,8pt);
    \draw (main.north) ++(12pt,0) -- ++(0,8pt);
  }\,
  = 0.\end{equation}

  We cannot make such a convenient choice for comultiplication $\,\tikz[baseline=(main.base)]{
    \node[draw,rectangle,inner sep=2pt] (main) {\comult};
    \draw (main.north) -- ++(0,8pt);
    \draw (main.south) ++(-6pt,0) -- ++(0,-8pt);
    \draw (main.south) ++(6pt,0) -- ++(0,-8pt);
  }\,$.  
  Let us instead choose a smooth 1-form $\phi \in \Omega^{1}_{\mathrm{cpt}}\bigl( (-\frac\pi4,\frac\pi4)\bigr)$ with total integral $\int \phi = 1$.  Then we can define:
  \begin{equation*}\label{eqn.Phi}
   \Phi_{\epsilon}(x) = \phi\bigl( x/\arctan(\epsilon)\bigr).
   \end{equation*}
  This is a smooth closed 1-form in $\Omega^{1}\bigl(\bR_{>0} \times (-\frac12,\frac12)\bigr) $ with compact support in the  $(-\frac12,\frac12)$-direction.  Indeed, it is supported within $x \in (-\epsilon,\epsilon)$; as $\epsilon \to 0$, $\Phi_{\epsilon}$ converges to the $\delta$-distribution.  We will later use the following primitive of $\Phi_\epsilon$:
  \begin{equation} \label{eqn.F}
    F_\epsilon(x) = \int_{y = -\infty}^x \Phi_\epsilon(y) = \int_y \Theta(x-y)\,\Phi_\epsilon(y).
  \end{equation}
  Then $\d F_\epsilon = \Phi_\epsilon$.  The function $F_\epsilon$ is not compactly-supported on $\bR$, and so does not descend to a function on $S^1$.  But the difference $F_\epsilon - \Theta$ is compactly supported --- indeed, it is supported in the interval $(-\epsilon,\epsilon)$ --- and it is a homotopy between $\Phi_\epsilon$ and $\delta$.
  
    With these choices, we hereby set:
  \begin{equation*}
    \,\tikz[baseline=(main.base)]{
    \node[draw,rectangle,inner sep=2pt] (main) {\comult};
    \draw (main.north) -- ++(0,8pt);
    \draw (main.south) ++(-6pt,0) -- ++(0,-8pt) node[anchor=north] {$\scriptstyle x$};
    \draw (main.south) ++(6pt,0) -- ++(0,-8pt) node[anchor=north] {$\scriptstyle y$};
  }\, = \int_{v} \Phi_{\epsilon}(x-v)\,\Phi_{\epsilon}(y-v)\,\alpha(v).
  \end{equation*}
  The corresponding integral kernel $(x,y,v) \mapsto \Phi_{\epsilon}(x-v)\,\Phi_{\epsilon}(y-v)$ is then non-zero only when $|x-v| < \epsilon$ and $|y-v| < \epsilon$, and is therefore quasilocal.  Since $\Phi_{\epsilon}(x-v)$ is a smooth approximation of  $\delta(x-v)$, the above choice does indeed represent comultiplication.
  
    Of the Frobenius and coassociative homotopies, the former is easier, and so we turn to it first.  According to equation~(\ref{eqn.frobeniator}), in order to represent the Frobenius homotopy, we need to find a quasilocal primitive for the following operation:
  \begin{align*}
    \left[\d,
    \,\tikz[baseline=(main.base)]{
    \node[draw,rectangle,inner sep=2pt] (main) {$
         \,\tikz[baseline=(basepoint)]{ 
    \path (0,3pt) coordinate (basepoint) (0,1pt) node[dot] {} (0,11pt) node[dot] {};
    \draw[](-8pt,-9pt) -- (0,1pt);
    \draw[](8pt,-9pt) -- (0,1pt);
    \draw[](0,1pt) -- (0,11pt);
    \draw[](0,11pt) -- (-8pt,21pt);
    \draw[](0,11pt) -- (8pt,21pt);
  }\,
  +
  \,\tikz[baseline=(basepoint)]{ 
    \path (0,3pt) coordinate (basepoint) (8pt,1pt) node[dot] {} (-8pt,11pt) node[dot] {};
    \draw[](-8pt,-9pt) -- (-8pt,11pt);
    \draw[](8pt,-9pt) -- (8pt,1pt);
    \draw[](8pt,1pt) -- (-8pt,11pt);
    \draw[](-8pt,11pt) -- (-8pt,21pt);
    \draw[](8pt,1pt) -- (8pt,21pt);
  }\,
    $};
    \draw (main.south) ++(-6pt,0) -- ++(0pt,-8pt) node[anchor=north] {$\scriptstyle x$};
    \draw (main.south) ++(6pt,0) -- ++(0pt,-8pt) node[anchor=north] {$\scriptstyle y$};
    \draw (main.north) ++(-6pt,0) -- ++(0,8pt);
    \draw (main.north) ++(6pt,0) -- ++(0,8pt);
  }\,
  \right] & \want  \int_{v} \Phi_{\epsilon}(x-v)\,\Phi_{\epsilon}(y-v)\, \alpha(v,v)  
  -\int_{v} \Phi_{\epsilon}(x-v)\,\Phi_{\epsilon}(y-v)\, \alpha(v,y) \\
   & \wantspacing \int_{v} \Phi_{\epsilon}(x-v)\,\Phi_{\epsilon}(y-v)\,
   \int_w \bigl(\delta(w-v) - \delta(w-y)\bigr)\, \alpha(v,w).
  \end{align*}
  
  There are multiple solutions to this equation.  One of them is
  \begin{equation}\label{eqn.answer for frobeniator}
    \,\tikz[baseline=(main.base)]{
    \node[draw,rectangle,inner sep=2pt] (main) {$
         \,\tikz[baseline=(basepoint)]{ 
    \path (0,3pt) coordinate (basepoint) (0,1pt) node[dot] {} (0,11pt) node[dot] {};
    \draw[](-8pt,-9pt) -- (0,1pt);
    \draw[](8pt,-9pt) -- (0,1pt);
    \draw[](0,1pt) -- (0,11pt);
    \draw[](0,11pt) -- (-8pt,21pt);
    \draw[](0,11pt) -- (8pt,21pt);
  }\,
  +
  \,\tikz[baseline=(basepoint)]{ 
    \path (0,3pt) coordinate (basepoint) (8pt,1pt) node[dot] {} (-8pt,11pt) node[dot] {};
    \draw[](-8pt,-9pt) -- (-8pt,11pt);
    \draw[](8pt,-9pt) -- (8pt,1pt);
    \draw[](8pt,1pt) -- (-8pt,11pt);
    \draw[](-8pt,11pt) -- (-8pt,21pt);
    \draw[](8pt,1pt) -- (8pt,21pt);
  }\,
    $};
    \draw (main.south) ++(-6pt,0) -- ++(0pt,-8pt) node[anchor=north] {$\scriptstyle x$};
    \draw (main.south) ++(6pt,0) -- ++(0pt,-8pt) node[anchor=north] {$\scriptstyle y$};
    \draw (main.north) ++(-6pt,0) -- ++(0,8pt);
    \draw (main.north) ++(6pt,0) -- ++(0,8pt);
  }\,
  =  
  \int_{v} \Phi_{\epsilon}(x-v)\,\Phi_{\epsilon}(y-v)\,
   \int_w \bigl(\Theta(w-v) - \Theta(w-y)\bigr)\, \alpha(v,w).
  \end{equation}
  That this has the correct derivative is clear, but we must check quasilocality.  The integral kernel is $\Phi_{\epsilon}(x-v)\,\Phi_{\epsilon}(y-v)\,\bigl(\Theta(w-v) - \Theta(w-y)\bigr)$, which  certainly vanishes except when $|x-v|<\epsilon$ and $|y-v|<\epsilon$.  The final term $\bigl(\Theta(w-v) - \Theta(w-y)\bigr)$ is non-zero only when $w$ is between~$y$ and~$v$ (regardless of which is greater), which are already forced near each other.  It follows that all four of $x,y,v,w$ must be  within $2\epsilon$ of each other in order for the kernel to be non-zero, thus verifying quasilocality.  
  
    We now turn to the homotopy imposing coassociativity.  According to equation~(\ref{eqn.coassociator}), that homotopy should be a primitive of
  \begin{align*} \left[\d,
    \,\tikz[baseline=(main.base)]{
    \node[draw,rectangle,inner sep=2pt] (main) {$
   -   \,\tikz[baseline=(basepoint)]{ 
    \path (0,3pt) coordinate (basepoint) (8pt,1pt) node[dot] {} (0,11pt) node[dot] {};
    \draw[](-12pt,-9pt) -- (0,11pt);
    \draw[](0pt,-9pt) -- (8pt,1pt);
    \draw[](12pt,-9pt) -- (8pt,1pt);
    \draw[](8pt,1pt) -- (0,11pt);
    \draw[](0,11pt) -- (0,21pt);
  }\,
  -
  \,\tikz[baseline=(basepoint)]{ 
    \path (0,3pt) coordinate (basepoint) (-8pt,1pt) node[dot] {} (0,11pt) node[dot] {};
    \draw[](12pt,-9pt) -- (0,11pt);
    \draw[](-12pt,-9pt) -- (-8pt,1pt);
    \draw[](0pt,-9pt) -- (-8pt,1pt);
    \draw[](-8pt,1pt) -- (0,11pt);
    \draw[](0,11pt) -- (0,21pt);
  }\,
    $};
    \draw (main.north) -- ++(0,8pt);
    \draw (main.south) ++(-12pt,0) -- ++(0,-8pt) node[anchor=north] {$\scriptstyle x$};
    \draw (main.south) ++(0pt,0) -- ++(0,-8pt) node[anchor=north] {$\scriptstyle y$};
    \draw (main.south) ++(12pt,0) -- ++(0,-8pt) node[anchor=north] {$\scriptstyle z$};
  }\,
  \right]
  & \want 
  -
  \int_{w}
  \Phi_{\epsilon}(y-w)\, \Phi_{\epsilon}(z-w)\,
  \int_{v}
  \Phi_{\epsilon}(x-v)\, \Phi_{\epsilon}(w-v)\,
  \alpha(v)
  \\[-24pt] & \hspace{1in}  - 
  \int_{w}
  \Phi_{\epsilon}(x-w)\, \Phi_{\epsilon}(y-w)\,
  \int_{v}
  \Phi_{\epsilon}(w-v)\, \Phi_{\epsilon}(z-v)\,
  \alpha(v) \\
  & \wantspacing 
    -
  \int_{w}
  \Phi_{\epsilon}(y-w)\, \Phi_{\epsilon}(z-w)\,
  \int_{v}
  \Phi_{\epsilon}(x-v)\, \Phi_{\epsilon}(w-v)\,
  \alpha(v)
  \\ & \hspace{1in}  +
  \int_{w}
  \Phi_{\epsilon}(x-w)\, \Phi_{\epsilon}(y-w)\,
  \int_{v}
  \Phi_{\epsilon}(z-v)\, \Phi_{\epsilon}(w-v)\,
  \alpha(v). \\
  \end{align*}
  In addition to being quasilocal, the primitive must live within a copy of the 2-dimensional representation \tikz \draw (0,0) rectangle (4pt,4pt) (4pt,0pt) rectangle (8pt,4pt) (0pt,-4pt)  rectangle  (4pt,0pt) ; under the $\bS_{3}$-action permuting the variables $x,y,z$.

Let us suppose that we can instead get partway, finding a quasilocal solution to
  \begin{align} \label{eqn.formula for coassociator} \left[\d,
      \,\tikz[baseline=(main.base)]{
    \node[draw,rectangle,inner sep=2pt] (main) {$
  \,\tikz[baseline=(basepoint)]{ 
    \path (0,3pt) coordinate (basepoint)  (0,11pt) node[dot] {};
    \draw[](12pt,-9pt) -- (0,11pt);
    \draw[](-12pt,-9pt) -- (0,11pt);
    \draw[](0pt,-9pt) -- (0,11pt);
    \draw[](0,11pt) -- (0,21pt);
  }\,  
   -   \,\tikz[baseline=(basepoint)]{ 
    \path (0,3pt) coordinate (basepoint) (8pt,1pt) node[dot] {} (0,11pt) node[dot] {};
    \draw[](-12pt,-9pt) -- (0,11pt);
    \draw[](0pt,-9pt) -- (8pt,1pt);
    \draw[](12pt,-9pt) -- (8pt,1pt);
    \draw[](8pt,1pt) -- (0,11pt);
    \draw[](0,11pt) -- (0,21pt);
  }\,
    $};
    \draw (main.north) -- ++(0,8pt);
    \draw (main.south) ++(-12pt,0) -- ++(0,-8pt) node[anchor=north] {$\scriptstyle x$};
    \draw (main.south) ++(0pt,0) -- ++(0,-8pt) node[anchor=north] {$\scriptstyle y$};
    \draw (main.south) ++(12pt,0) -- ++(0,-8pt) node[anchor=north] {$\scriptstyle z$};
  }\,
  \right]
  & \want 
  \int_{v}
  \Phi_{\epsilon}(x-v)\, 
  \int_{w}
  \bigl(\delta(w-v) - \Phi_{\epsilon}(w-v)\bigr)\,
  \Phi_{\epsilon}(y-w)\, \Phi_{\epsilon}(z-w)\,
  \alpha(v)
    \\[-18pt] \notag &
  \wantspacing
  \int_{v}
  \Phi_{\epsilon}(x-v)\,\Phi_{\epsilon}(y-v)\, \Phi_{\epsilon}(z-v)\,
  \alpha(v)
    \\ \notag
  & \hspace{1in} - 
  \int_{w}
  \Phi_{\epsilon}(y-w)\, \Phi_{\epsilon}(z-w)\,
  \int_{v}
  \Phi_{\epsilon}(x-v)\, \Phi_{\epsilon}(w-v)\,
  \alpha(v)
  \end{align}
 which (like its derivative) is antisymmetric under swapping $y$ and $z$.  Then we can set
  $$ 
      \,\tikz[baseline=(main.base)]{
    \node[draw,rectangle,inner sep=2pt] (main) {$
   -   \,\tikz[baseline=(basepoint)]{ 
    \path (0,3pt) coordinate (basepoint) (8pt,1pt) node[dot] {} (0,11pt) node[dot] {};
    \draw[](-12pt,-9pt) -- (0,11pt);
    \draw[](0pt,-9pt) -- (8pt,1pt);
    \draw[](12pt,-9pt) -- (8pt,1pt);
    \draw[](8pt,1pt) -- (0,11pt);
    \draw[](0,11pt) -- (0,21pt);
  }\,
  -
  \,\tikz[baseline=(basepoint)]{ 
    \path (0,3pt) coordinate (basepoint) (-8pt,1pt) node[dot] {} (0,11pt) node[dot] {};
    \draw[](12pt,-9pt) -- (0,11pt);
    \draw[](-12pt,-9pt) -- (-8pt,1pt);
    \draw[](0pt,-9pt) -- (-8pt,1pt);
    \draw[](-8pt,1pt) -- (0,11pt);
    \draw[](0,11pt) -- (0,21pt);
  }\,
    $};
    \draw (main.north) -- ++(0,8pt);
    \draw (main.south) ++(-12pt,0) -- ++(0,-8pt) node[anchor=north] {$\scriptstyle x$};
    \draw (main.south) ++(0pt,0) -- ++(0,-8pt) node[anchor=north] {$\scriptstyle y$};
    \draw (main.south) ++(12pt,0) -- ++(0,-8pt) node[anchor=north] {$\scriptstyle z$};
  }\,
  =
    \,\tikz[baseline=(main.base)]{
    \node[draw,rectangle,inner sep=2pt] (main) {$
  \,\tikz[baseline=(basepoint)]{ 
    \path (0,3pt) coordinate (basepoint)  (0,11pt) node[dot] {};
    \draw[](12pt,-9pt) -- (0,11pt);
    \draw[](-12pt,-9pt) -- (0,11pt);
    \draw[](0pt,-9pt) -- (0,11pt);
    \draw[](0,11pt) -- (0,21pt);
  }\,  
   -   \,\tikz[baseline=(basepoint)]{ 
    \path (0,3pt) coordinate (basepoint) (8pt,1pt) node[dot] {} (0,11pt) node[dot] {};
    \draw[](-12pt,-9pt) -- (0,11pt);
    \draw[](0pt,-9pt) -- (8pt,1pt);
    \draw[](12pt,-9pt) -- (8pt,1pt);
    \draw[](8pt,1pt) -- (0,11pt);
    \draw[](0,11pt) -- (0,21pt);
  }\,
    $};
    \draw (main.north) -- ++(0,8pt);
    \draw (main.south) ++(-12pt,0) -- ++(0,-8pt) node[anchor=north] {$\scriptstyle x$};
    \draw (main.south) ++(0pt,0) -- ++(0,-8pt) node[anchor=north] {$\scriptstyle y$};
    \draw (main.south) ++(12pt,0) -- ++(0,-8pt) node[anchor=north] {$\scriptstyle z$};
  }\,
  -  
      \,\tikz[baseline=(main.base)]{
    \node[draw,rectangle,inner sep=2pt] (main) {$
  \,\tikz[baseline=(basepoint)]{ 
    \path (0,3pt) coordinate (basepoint)  (0,11pt) node[dot] {};
    \draw[](12pt,-9pt) -- (0,11pt);
    \draw[](-12pt,-9pt) -- (0,11pt);
    \draw[](0pt,-9pt) -- (0,11pt);
    \draw[](0,11pt) -- (0,21pt);
  }\,  
   -   \,\tikz[baseline=(basepoint)]{ 
    \path (0,3pt) coordinate (basepoint) (8pt,1pt) node[dot] {} (0,11pt) node[dot] {};
    \draw[](-12pt,-9pt) -- (0,11pt);
    \draw[](0pt,-9pt) -- (8pt,1pt);
    \draw[](12pt,-9pt) -- (8pt,1pt);
    \draw[](8pt,1pt) -- (0,11pt);
    \draw[](0,11pt) -- (0,21pt);
  }\,
    $};
    \draw (main.north) -- ++(0,8pt);
    \draw (main.south) ++(-12pt,0) -- ++(0,-8pt) node[anchor=north] {$\scriptstyle z$};
    \draw (main.south) ++(0pt,0) -- ++(0,-8pt) node[anchor=north] {$\scriptstyle x$};
    \draw (main.south) ++(12pt,0) -- ++(0,-8pt) node[anchor=north] {$\scriptstyle y$};
  }\,,
  $$
  i.e.
  \begin{align} \label{diagrammaticcoassocdefn}
      \,\tikz[baseline=(main.base)]{
    \node[draw,rectangle,inner sep=2pt] (main) {$
   -   \,\tikz[baseline=(basepoint)]{ 
    \path (0,3pt) coordinate (basepoint) (8pt,1pt) node[dot] {} (0,11pt) node[dot] {};
    \draw[](-12pt,-9pt) -- (0,11pt);
    \draw[](0pt,-9pt) -- (8pt,1pt);
    \draw[](12pt,-9pt) -- (8pt,1pt);
    \draw[](8pt,1pt) -- (0,11pt);
    \draw[](0,11pt) -- (0,21pt);
  }\,
  -
  \,\tikz[baseline=(basepoint)]{ 
    \path (0,3pt) coordinate (basepoint) (-8pt,1pt) node[dot] {} (0,11pt) node[dot] {};
    \draw[](12pt,-9pt) -- (0,11pt);
    \draw[](-12pt,-9pt) -- (-8pt,1pt);
    \draw[](0pt,-9pt) -- (-8pt,1pt);
    \draw[](-8pt,1pt) -- (0,11pt);
    \draw[](0,11pt) -- (0,21pt);
  }\,
    $};
    \draw (main.north) -- ++(0,8pt);
    \draw (main.south) ++(-12pt,0) -- ++(0,-16pt) ;
    \draw (main.south) ++(0pt,0) -- ++(0,-16pt);
    \draw (main.south) ++(12pt,0) -- ++(0,-16pt);
  }\,
  & =
    \,\tikz[baseline=(main.base)]{
    \node[draw,rectangle,inner sep=2pt] (main) {$
  \,\tikz[baseline=(basepoint)]{ 
    \path (0,3pt) coordinate (basepoint)  (0,11pt) node[dot] {};
    \draw[](12pt,-9pt) -- (0,11pt);
    \draw[](-12pt,-9pt) -- (0,11pt);
    \draw[](0pt,-9pt) -- (0,11pt);
    \draw[](0,11pt) -- (0,21pt);
  }\,  
   -   \,\tikz[baseline=(basepoint)]{ 
    \path (0,3pt) coordinate (basepoint) (8pt,1pt) node[dot] {} (0,11pt) node[dot] {};
    \draw[](-12pt,-9pt) -- (0,11pt);
    \draw[](0pt,-9pt) -- (8pt,1pt);
    \draw[](12pt,-9pt) -- (8pt,1pt);
    \draw[](8pt,1pt) -- (0,11pt);
    \draw[](0,11pt) -- (0,21pt);
  }\,
    $};
    \draw (main.north) -- ++(0,8pt);
    \draw (main.south) ++(-12pt,0) -- ++(0,-16pt) ;
    \draw (main.south) ++(0pt,0) -- ++(0,-16pt);
    \draw (main.south) ++(12pt,0) -- ++(0,-16pt);
  }\,
  -  
      \,\tikz[baseline=(main.base)]{
    \node[draw,rectangle,inner sep=2pt] (main) {$
  \,\tikz[baseline=(basepoint)]{ 
    \path (0,3pt) coordinate (basepoint)  (0,11pt) node[dot] {};
    \draw[](12pt,-9pt) -- (0,11pt);
    \draw[](-12pt,-9pt) -- (0,11pt);
    \draw[](0pt,-9pt) -- (0,11pt);
    \draw[](0,11pt) -- (0,21pt);
  }\,  
   -   \,\tikz[baseline=(basepoint)]{ 
    \path (0,3pt) coordinate (basepoint) (8pt,1pt) node[dot] {} (0,11pt) node[dot] {};
    \draw[](-12pt,-9pt) -- (0,11pt);
    \draw[](0pt,-9pt) -- (8pt,1pt);
    \draw[](12pt,-9pt) -- (8pt,1pt);
    \draw[](8pt,1pt) -- (0,11pt);
    \draw[](0,11pt) -- (0,21pt);
  }\,
    $};
    \draw (main.north) -- ++(0,8pt);
    \draw (main.south) ++(-12pt,0) -- ++(24pt,-16pt) ;
    \draw (main.south) ++(0pt,0) -- ++(-12pt,-16pt) ;
    \draw (main.south) ++(12pt,0) -- ++(-12pt,-16pt) ;
  }\,
  \\ \notag & =
  \left( 
\,\tikz[baseline=(base)] \draw coordinate (base) ++(0,10pt) -- +(0pt,-16pt) ++(12pt,0) -- +(0pt,-16pt) ++(12pt,0) -- ++(0pt,-16pt);\,  -
  \,\tikz[baseline=(base)] \draw coordinate (base) ++(0,10pt) -- +(24pt,-16pt) ++(12pt,0) -- +(-12pt,-16pt) ++(12pt,0) -- ++(-12pt,-16pt);\,
  \right)
  \circ \,\tikz[baseline=(main.base)]{
    \node[draw,rectangle,inner sep=2pt] (main) {$
  \,\tikz[baseline=(basepoint)]{ 
    \path (0,3pt) coordinate (basepoint)  (0,11pt) node[dot] {};
    \draw[](12pt,-9pt) -- (0,11pt);
    \draw[](-12pt,-9pt) -- (0,11pt);
    \draw[](0pt,-9pt) -- (0,11pt);
    \draw[](0,11pt) -- (0,21pt);
  }\,  
   -   \,\tikz[baseline=(basepoint)]{ 
    \path (0,3pt) coordinate (basepoint) (8pt,1pt) node[dot] {} (0,11pt) node[dot] {};
    \draw[](-12pt,-9pt) -- (0,11pt);
    \draw[](0pt,-9pt) -- (8pt,1pt);
    \draw[](12pt,-9pt) -- (8pt,1pt);
    \draw[](8pt,1pt) -- (0,11pt);
    \draw[](0,11pt) -- (0,21pt);
  }\,
    $};
    \draw (main.north) -- ++(0,8pt);
    \draw (main.south) ++(-12pt,0) -- ++(0,-16pt) ;
    \draw (main.south) ++(0pt,0) -- ++(0,-16pt);
    \draw (main.south) ++(12pt,0) -- ++(0,-16pt);
  }\,.
  \end{align}
  To see that this works, note that the first summand in the right-hand side of equation~(\ref{eqn.formula for coassociator}) is fixed by the permutation \,\tikz[baseline=(base)] \draw coordinate (base) ++(0,9pt) -- +(16pt,-12pt) ++(8pt,0) -- +(-8pt,-12pt) ++(8pt,0) -- ++(-8pt,-12pt);\,, and so cancels upon composing with $
  \,\tikz[baseline=(base)] \draw coordinate (base) ++(0,9pt) -- +(0pt,-12pt) ++(8pt,0) -- +(0pt,-12pt) ++(8pt,0) -- ++(0pt,-12pt);\,
  -
  \,\tikz[baseline=(base)] \draw coordinate (base) ++(0,9pt) -- +(16pt,-12pt) ++(8pt,0) -- +(-8pt,-12pt) ++(8pt,0) -- ++(-8pt,-12pt);\,$, and the other two terms give exactly the desired derivative.  Moreover, the two-dimensional $\bS_{3}$-representation \tikz \draw (0,0) rectangle (4pt,4pt) (4pt,0pt) rectangle (8pt,4pt) (0pt,-4pt)  rectangle  (4pt,0pt) ; is generated, for example, by a non-zero element $\xi \in \tikz \draw (0,0) rectangle (4pt,4pt) (4pt,0pt) rectangle (8pt,4pt) (0pt,-4pt)  rectangle  (4pt,0pt) ;$ satisfying 
  $\bigl( \,\tikz[baseline=(base)] \draw coordinate (base) ++(0,9pt) -- +(0pt,-12pt) ++(8pt,0) -- +(0pt,-12pt) ++(8pt,0) -- ++(0pt,-12pt);\,
  +
  \,\tikz[baseline=(base)] \draw coordinate (base) ++(0,9pt) -- +(16pt,-12pt) ++(8pt,0) -- +(-8pt,-12pt) ++(8pt,0) -- ++(-8pt,-12pt);\,
    +
  \,\tikz[baseline=(base)] \draw coordinate (base) ++(0,9pt) -- +(8pt,-12pt) ++(8pt,0) -- +(8pt,-12pt) ++(8pt,0) -- ++(-16pt,-12pt);\, 
  \bigr) \,\xi = \bigl( 
  \,\tikz[baseline=(base)] \draw coordinate (base) ++(0,9pt) -- +(0pt,-12pt) ++(8pt,0) -- +(0pt,-12pt) ++(8pt,0) -- ++(0pt,-12pt);\,
  -
  \,\tikz[baseline=(base)] \draw coordinate (base) ++(0,9pt) -- +(16pt,-12pt) ++(8pt,0) -- +(0pt,-12pt) ++(8pt,0) -- ++(-16pt,-12pt);\,\bigr)\,\xi = 0$.  But
  $$ \bigl( 
  \,\tikz[baseline=(base)] \draw coordinate (base) ++(0,9pt) -- +(0pt,-12pt) ++(8pt,0) -- +(0pt,-12pt) ++(8pt,0) -- ++(0pt,-12pt);\,
  +
  \,\tikz[baseline=(base)] \draw coordinate (base) ++(0,9pt) -- +(16pt,-12pt) ++(8pt,0) -- +(-8pt,-12pt) ++(8pt,0) -- ++(-8pt,-12pt);\,
    +
  \,\tikz[baseline=(base)] \draw coordinate (base) ++(0,9pt) -- +(8pt,-12pt) ++(8pt,0) -- +(8pt,-12pt) ++(8pt,0) -- ++(-16pt,-12pt);\, 
  \bigr)
  \circ
  \bigl( 
  \,\tikz[baseline=(base)] \draw coordinate (base) ++(0,9pt) -- +(0pt,-12pt) ++(8pt,0) -- +(0pt,-12pt) ++(8pt,0) -- ++(0pt,-12pt);\,
  -
  \,\tikz[baseline=(base)] \draw coordinate (base) ++(0,9pt) -- +(16pt,-12pt) ++(8pt,0) -- +(-8pt,-12pt) ++(8pt,0) -- ++(-8pt,-12pt);\,
  \bigr)
  = 0$$
  automatically, and
  $$
  \bigl( 
  \,\tikz[baseline=(base)] \draw coordinate (base) ++(0,9pt) -- +(0pt,-12pt) ++(8pt,0) -- +(0pt,-12pt) ++(8pt,0) -- ++(0pt,-12pt);\,
  -
  \,\tikz[baseline=(base)] \draw coordinate (base) ++(0,9pt) -- +(16pt,-12pt) ++(8pt,0) -- +(0pt,-12pt) ++(8pt,0) -- ++(-16pt,-12pt);\,\bigr)
    \circ
  \bigl( 
  \,\tikz[baseline=(base)] \draw coordinate (base) ++(0,9pt) -- +(0pt,-12pt) ++(8pt,0) -- +(0pt,-12pt) ++(8pt,0) -- ++(0pt,-12pt);\,
  -
  \,\tikz[baseline=(base)] \draw coordinate (base) ++(0,9pt) -- +(16pt,-12pt) ++(8pt,0) -- +(-8pt,-12pt) ++(8pt,0) -- ++(-8pt,-12pt);\,
  \bigr)
  =
  \bigl( 
  \,\tikz[baseline=(base)] \draw coordinate (base) ++(0,9pt) -- +(0pt,-12pt) ++(8pt,0) -- +(0pt,-12pt) ++(8pt,0) -- ++(0pt,-12pt);\,
  -
  \,\tikz[baseline=(base)] \draw coordinate (base) ++(0,9pt) -- +(16pt,-12pt) ++(8pt,0) -- +(-8pt,-12pt) ++(8pt,0) -- ++(-8pt,-12pt);\,
  \bigr)
  \circ
    \bigl( 
  \,\tikz[baseline=(base)] \draw coordinate (base) ++(0,9pt) -- +(0pt,-12pt) ++(8pt,0) -- +(0pt,-12pt) ++(8pt,0) -- ++(0pt,-12pt);\,
  +
  \,\tikz[baseline=(base)] \draw coordinate (base) ++(0,9pt) -- +(0pt,-12pt) ++(8pt,0) -- +(8pt,-12pt) ++(8pt,0) -- ++(-8pt,-12pt);\,\bigr),
  $$
  and we have requested that
  $$\left( 
\,\tikz[baseline=(base)] \draw coordinate (base) ++(0,10pt) -- +(0pt,-16pt) ++(12pt,0) -- +(0pt,-16pt) ++(12pt,0) -- ++(0pt,-16pt);\,  +
  \,\tikz[baseline=(base)] \draw coordinate (base) ++(0,10pt) -- +(0pt,-16pt) ++(12pt,0) -- +(12pt,-16pt) ++(12pt,0) -- ++(-12pt,-16pt);\,
  \right)
  \circ \,\tikz[baseline=(main.base)]{
    \node[draw,rectangle,inner sep=2pt] (main) {$
  \,\tikz[baseline=(basepoint)]{ 
    \path (0,3pt) coordinate (basepoint)  (0,11pt) node[dot] {};
    \draw[](12pt,-9pt) -- (0,11pt);
    \draw[](-12pt,-9pt) -- (0,11pt);
    \draw[](0pt,-9pt) -- (0,11pt);
    \draw[](0,11pt) -- (0,21pt);
  }\,  
   -   \,\tikz[baseline=(basepoint)]{ 
    \path (0,3pt) coordinate (basepoint) (8pt,1pt) node[dot] {} (0,11pt) node[dot] {};
    \draw[](-12pt,-9pt) -- (0,11pt);
    \draw[](0pt,-9pt) -- (8pt,1pt);
    \draw[](12pt,-9pt) -- (8pt,1pt);
    \draw[](8pt,1pt) -- (0,11pt);
    \draw[](0,11pt) -- (0,21pt);
  }\,
    $};
    \draw (main.north) -- ++(0,8pt);
    \draw (main.south) ++(-12pt,0) -- ++(0,-16pt) ;
    \draw (main.south) ++(0pt,0) -- ++(0,-16pt);
    \draw (main.south) ++(12pt,0) -- ++(0,-16pt);
  }\,
  = 0.
  $$
  
    It therefore suffices to find a solution to equation~(\ref{eqn.formula for coassociator}).  But the following works:
  \begin{equation}\label{half of the coassociator}
    \,\tikz[baseline=(main.base)]{
    \node[draw,rectangle,inner sep=2pt] (main) {$
  \,\tikz[baseline=(basepoint)]{ 
    \path (0,3pt) coordinate (basepoint)  (0,11pt) node[dot] {};
    \draw[](12pt,-9pt) -- (0,11pt);
    \draw[](-12pt,-9pt) -- (0,11pt);
    \draw[](0pt,-9pt) -- (0,11pt);
    \draw[](0,11pt) -- (0,21pt);
  }\,  
   -   \,\tikz[baseline=(basepoint)]{ 
    \path (0,3pt) coordinate (basepoint) (8pt,1pt) node[dot] {} (0,11pt) node[dot] {};
    \draw[](-12pt,-9pt) -- (0,11pt);
    \draw[](0pt,-9pt) -- (8pt,1pt);
    \draw[](12pt,-9pt) -- (8pt,1pt);
    \draw[](8pt,1pt) -- (0,11pt);
    \draw[](0,11pt) -- (0,21pt);
  }\,
    $};
    \draw (main.north) -- ++(0,8pt);
    \draw (main.south) ++(-12pt,0) -- ++(0,-8pt) node[anchor=north] {$\scriptstyle x$};
    \draw (main.south) ++(0pt,0) -- ++(0,-8pt) node[anchor=north] {$\scriptstyle y$};
    \draw (main.south) ++(12pt,0) -- ++(0,-8pt) node[anchor=north] {$\scriptstyle z$};
  }\,
   = 
  \int_{v}
  \Phi_{\epsilon}(x-v)\, 
  \int_{w}
  \bigl(F_{\epsilon}(w-v) - \Theta(w-v) \bigr)
  \Phi_{\epsilon}(y-w)\, \Phi_{\epsilon}(z-w)\,
  \alpha(v)\end{equation}
    Indeed, when differentiating, the $\d$ hits only the term $\bigl(F_{\epsilon}(w-v) - \Theta(w-v) \bigr)$, picking up a sign along the way (since it has to pass three degree-odd terms), to produce  $\delta(w-v) - \Phi_{\epsilon}(w-v)$, giving the desired derivative.  It is clearly antisymmetric under swapping $y$ and $z$.  Quasilocality is also clear: the terms $\Phi_\epsilon(x-v)$, $\Phi_{\epsilon}(y-w)$, and $\Phi_{\epsilon}(z-w)$ are non-zero only when $|x-v|<\epsilon$, $|y-w|<\epsilon$, and $|z-w|<\epsilon$, respectively, and $\bigl(F_{\epsilon}(w-v) - \Theta(w-v) \bigr)$ is non-zero only when $|w-v|<\epsilon$.

  Having now addressed all genus-zero generators for which the obstruction could have failed to vanish, we turn to the generators of genus-one.
  Because in equation~(\ref{eqn.assoc=0}) we were able to choose the homotopy controlling associativity to vanish identically, equations~(\ref{Dderivative}) and~(\ref{eqn.answer for frobeniator}) give:
  \begin{equation}
  \label{obstruction for Dgraphbox}
    \left[\d,
  \,\tikz[baseline=(main.base)]{
    \node[draw,rectangle,inner sep=2pt] (main) {\graphD};
    \draw (main.south) -- ++(0,-8pt) node[anchor=north] {$\scriptstyle x$};
    \draw (main.north) ++(-6pt,0) -- ++(0,8pt);
    \draw (main.north) ++(6pt,0) -- ++(0,8pt);
  }\,
  \right] \want 
  -
  \;\begin{tikzpicture}[baseline=(outerbasepoint)]
  \node[draw,inner sep=2pt] (box) {
    $
    \,\tikz[baseline=(basepoint)]{ 
      \path (0,3pt) coordinate (basepoint) (0,1pt) node[dot] {} (0,11pt) node[dot] {};
      \draw[](-8pt,-9pt) -- (0,1pt);
      \draw[](8pt,-9pt) -- (0,1pt);
      \draw[](0,1pt) -- (0,11pt);
      \draw[](0,11pt) -- (-8pt,21pt);
      \draw[](0,11pt) -- (8pt,21pt);
    }\,
    +
    \,\tikz[baseline=(basepoint)]{ 
      \path (0,3pt) coordinate (basepoint) (8pt,1pt) node[dot] {} (-8pt,11pt) node[dot] {};
      \draw[](-8pt,-9pt) -- (-8pt,11pt);
      \draw[](8pt,-9pt) -- (8pt,1pt);
      \draw[](8pt,1pt) -- (-8pt,11pt);
      \draw[](-8pt,11pt) -- (-8pt,21pt);
      \draw[](8pt,1pt) -- (8pt,21pt);
    }\,
    $};
  \draw (box.north) ++(6pt,0) -- ++(0,10pt) coordinate(top);
  \draw (box.north) ++(-6pt,0) -- ++(0,10pt);
  \draw (box.south) ++(6pt,0) -- ++(0,-10pt);
  \draw (box.south) ++(-6pt,0) -- ++(0,-10pt);
  \draw (box.south) ++(0,-18pt) node[draw,inner sep=2pt,fill=white] (lowerbox) {\mult};
  \draw (lowerbox.south) -- ++(0,-10pt) coordinate (bottom) node[anchor=north] {$\scriptstyle x$};
  \path (bottom) ++(0,40pt) coordinate (outerbasepoint);  
  \end{tikzpicture}\;
  =
  -  \int_{v} \Phi_{\epsilon}(x-v)\,\Phi_{\epsilon}(x-v)\,
   \int_w \bigl(\Theta(w-v) - \Theta(w-y)\bigr)\, \alpha(v,w).
  \end{equation}
  But this vanishes since $\Phi_{\epsilon}(x-v)$ is a one-form, so $\Phi_{\epsilon}(x-v)\,\Phi_{\epsilon}(x-v) = 0$.  We can therefore choose:
    \begin{equation} \label{expression for Dgraphbox}
  \,\tikz[baseline=(main.base)]{
    \node[draw,rectangle,inner sep=2pt] (main) {\graphD};
    \draw (main.south) -- ++(0,-8pt);
    \draw (main.north) ++(-6pt,0) -- ++(0,8pt);
    \draw (main.north) ++(6pt,0) -- ++(0,8pt);
  }\,
  = 0. \end{equation}

  We do not get so lucky with the other genus-one generator.  Inspecting equation~(\ref{Aderivative}) and substituting in equations~(\ref{eqn.answer for frobeniator}), (\ref{diagrammaticcoassocdefn}), and~(\ref{half of the coassociator}) gives:
  \begin{align}\label{obstruction for Agraph}
    \left[ \d,\,\tikz[baseline=(main.base)]{
    \node[draw,rectangle,inner sep=2pt] (main) {\graphA};
    \draw (main.north) -- ++(0,8pt);
    \draw (main.south) ++(-6pt,0) -- ++(0,-8pt) node[anchor=north] {$\scriptstyle x$};
    \draw (main.south) ++(6pt,0) -- ++(0,-8pt) node[anchor=north] {$\scriptstyle y$};
  }\,\right] 
  & \want
    -
  \;\begin{tikzpicture}[baseline=(outerbasepoint)]
  \node[draw,inner sep=2pt] (box) {
    $
    \,\tikz[baseline=(basepoint)]{ 
      \path (0,3pt) coordinate (basepoint) (0,1pt) node[dot] {} (0,11pt) node[dot] {};
      \draw[](-8pt,-9pt) -- (0,1pt);
      \draw[](8pt,-9pt) -- (0,1pt);
      \draw[](0,1pt) -- (0,11pt);
      \draw[](0,11pt) -- (-8pt,21pt);
      \draw[](0,11pt) -- (8pt,21pt);
    }\,
    +
    \,\tikz[baseline=(basepoint)]{ 
      \path (0,3pt) coordinate (basepoint) (8pt,1pt) node[dot] {} (-8pt,11pt) node[dot] {};
      \draw[](-8pt,-9pt) -- (-8pt,11pt);
      \draw[](8pt,-9pt) -- (8pt,1pt);
      \draw[](8pt,1pt) -- (-8pt,11pt);
      \draw[](-8pt,11pt) -- (-8pt,21pt);
      \draw[](8pt,1pt) -- (8pt,21pt);
    }\,
    $
    };
  \draw (box.south) ++(6pt,0) -- ++(0,-10pt) coordinate (bottom) node[anchor=north] {$\scriptstyle y$};
  \draw (box.south) ++(-6pt,0) -- ++(0,-10pt) coordinate (bottom) node[anchor=north] {$\scriptstyle x$};
  \draw (box.north) ++(6pt,0) -- ++(0pt,10pt);
  \draw (box.north) ++(-6pt,0) -- ++(0pt,10pt);
  \draw (box.north) ++(-0pt,18pt) node[draw,inner sep=2pt,fill=white] (upperbox) {\comult};
  \draw (upperbox.north) -- ++(0,10pt) ++(18pt,0) coordinate (top);
  \path (bottom) ++(0,40pt) coordinate (outerbasepoint);  
  \end{tikzpicture}\;
  \;+\;
  \;\begin{tikzpicture}[baseline=(outerbasepoint)]
  \node[draw,inner sep=2pt] (box) {
    $
  \,\tikz[baseline=(basepoint)]{ 
    \path (0,3pt) coordinate (basepoint)  (0,11pt) node[dot] {};
    \draw[](12pt,-9pt) -- (0,11pt);
    \draw[](-12pt,-9pt) -- (0,11pt);
    \draw[](0pt,-9pt) -- (0,11pt);
    \draw[](0,11pt) -- (0,21pt);
  }\,  
   -   \,\tikz[baseline=(basepoint)]{ 
    \path (0,3pt) coordinate (basepoint) (8pt,1pt) node[dot] {} (0,11pt) node[dot] {};
    \draw[](-12pt,-9pt) -- (0,11pt);
    \draw[](0pt,-9pt) -- (8pt,1pt);
    \draw[](12pt,-9pt) -- (8pt,1pt);
    \draw[](8pt,1pt) -- (0,11pt);
    \draw[](0,11pt) -- (0,21pt);
  }\,
    $
    };
  \draw (box.north) ++(0pt,0) -- ++(0,10pt) coordinate(top);
  \draw (box.south) ++(0pt,0) -- ++(0,-10pt);
  \draw (box.south) ++(12pt,0) -- ++(0,-10pt);
  \draw (box.south) ++(+6pt,-18pt) node[draw,inner sep=2pt,fill=white] (lowerbox) {\mult};
  \draw (lowerbox.south) -- ++(0,-10pt) node[anchor=north] {$\scriptstyle y$} ++(-18pt,0pt) coordinate (bottom) node[anchor=north] {$\scriptstyle x$};
  \draw (box.south) ++(-12pt,0) -- (bottom);
  \path (bottom) ++(0,40pt) coordinate (outerbasepoint);  
  \end{tikzpicture}\;
  \;-\;
  \;\begin{tikzpicture}[baseline=(outerbasepoint)]
  \node[draw,inner sep=2pt] (box) {
    $
  \,\tikz[baseline=(basepoint)]{ 
    \path (0,3pt) coordinate (basepoint)  (0,11pt) node[dot] {};
    \draw[](12pt,-9pt) -- (0,11pt);
    \draw[](-12pt,-9pt) -- (0,11pt);
    \draw[](0pt,-9pt) -- (0,11pt);
    \draw[](0,11pt) -- (0,21pt);
  }\,  
   -   \,\tikz[baseline=(basepoint)]{ 
    \path (0,3pt) coordinate (basepoint) (8pt,1pt) node[dot] {} (0,11pt) node[dot] {};
    \draw[](-12pt,-9pt) -- (0,11pt);
    \draw[](0pt,-9pt) -- (8pt,1pt);
    \draw[](12pt,-9pt) -- (8pt,1pt);
    \draw[](8pt,1pt) -- (0,11pt);
    \draw[](0,11pt) -- (0,21pt);
  }\,
    $
    };
  \draw (box.north) ++(0pt,0) -- ++(0,10pt) coordinate(top);
  \draw (box.south) ++(-12pt,0) -- ++(24pt,-10pt);
  \draw (box.south) ++(12pt,0) -- ++(-12pt,-10pt);
  \draw (box.south) ++(+6pt,-20pt) node[draw,inner sep=2pt,fill=white] (lowerbox) {\mult};
  \draw (lowerbox.south) -- ++(0,-8pt) node[anchor=north] {$\scriptstyle y$} ++(-18pt,0pt) coordinate (bottom) node[anchor=north] {$\scriptstyle x$};
  \draw (box.south) ++(0pt,0) .. controls +(-12pt,-10pt) and +(0,20pt) ..  (bottom);
  \path (bottom) ++(0,40pt) coordinate (outerbasepoint);  
  \end{tikzpicture}\;
  \\ \notag
  & \wantspacing 
   - \int_{u} \Phi_{\epsilon}(x-u)\,\Phi_{\epsilon}(y-u)\,
   \int_{w=u}^y \int_v \Phi_\epsilon(u-v)\,\Phi_\epsilon(w-v)\,\alpha(v) 
   \\ \notag & \hspace{.5in} + 0 
   \\ \notag & \hspace{.5in} -   \int_{v}
  \Phi_{\epsilon}(y-v)
  \int_{w}
  \bigl(F_{\epsilon}(w-v) - \Theta(w-v) \bigr)\,
  \Phi_{\epsilon}(x-w)\, \Phi_{\epsilon}(y-w) \, \alpha(v)
  \end{align}
  
  We could try to solve equation~(\ref{obstruction for Agraph}) directly, as we did for the earlier generators.  But instead, let us note a few generalities.  The first is that the right-hand side is of the form
  $$     \left[ \d,\,\tikz[baseline=(main.base)]{
    \node[draw,rectangle,inner sep=2pt] (main) {\graphA};
    \draw (main.north) -- ++(0,8pt);
    \draw (main.south) ++(-6pt,0) -- ++(0,-8pt) node[anchor=north] {$\scriptstyle x$};
    \draw (main.south) ++(6pt,0) -- ++(0,-8pt) node[anchor=north] {$\scriptstyle y$};
  }\,\right] 
   \want
   \int_v \mu_\epsilon(x-v,y-v)\,\alpha(v)
   $$
 where
 \begin{multline}\label{eqn.mu}
  \mu_\epsilon(x',y') = 
    - \int_{u'} \Phi_{\epsilon}(x'-u')\,\Phi_{\epsilon}(y'-u')\,
      \int_{w'=u'}^{y'}\Phi_\epsilon(u')\,\Phi_\epsilon(w')
   \\ -   \Phi_{\epsilon}(y')
  \int_{w'}
  \bigl(F_{\epsilon}(w') - \Theta(w') \bigr)\,
  \Phi_{\epsilon}(x'-w')\, \Phi_{\epsilon}(y'-w')
 \end{multline}
  is a compactly-supported two-form on $(-3\epsilon,3\epsilon)^2 \subseteq \bR^2$.  (Here for each variable $s \in \{x,y,w,u\}$ we declare $s' = s-v$.)  The second generality is that a compactly-supported two-form on $\bR^2$ is exact among compactly-supported forms if and only if its integral over $\bR^2$ vanishes.  We may therefore abandon $S^1$ in favor of $\bR$ and ask for $\int_{x'}\int_{y'} \mu(x',y')$.  We calculate in pieces.  
  
  We begin by simplifying the first summand in equation~(\ref{eqn.mu}):
  \begin{align*} 
  \hspace{1in} & \hspace{-1in}  
   - \int_{u'} \Phi_{\epsilon}(x'-u')\,\Phi_{\epsilon}(y'-u')\,
      \int_{w'=u'}^{y'} \, \Phi_\epsilon(u')\,\Phi_\epsilon(w') 
      \\ \notag & = 
       \int_{u'}\Phi_{\epsilon}(x'-u')\,\Phi_{\epsilon}(y'-u')\, \Phi_\epsilon(u')\,
      \int_{w'=u'}^{y'} \, \Phi_\epsilon(w')
      \\ \notag & =
       \int_{u'}\Phi_{\epsilon}(x'-u')\,\Phi_{\epsilon}(y'-u')\, \Phi_\epsilon(u')\, \bigl( F_\epsilon(y') - F_\epsilon(u') \bigr).
  \end{align*}
  We can then substitute $w' = u'$ in the second summand in equation~(\ref{eqn.mu}) and rearrange to conclude:
  \begin{multline*} 
    -   \Phi_{\epsilon}(y')   \int_{w'}  \bigl(F_{\epsilon}(w') - \Theta(w') \bigr)\,  \Phi_{\epsilon}(x'-w')\, \Phi_{\epsilon}(y'-w')
    \\ =  \int_{u'} \Phi_{\epsilon}(x'-u')\, \Phi_{\epsilon}(y'-u')\, \Phi_{\epsilon}(y')\, \bigl(F_{\epsilon}(u') - \Theta(u') \bigr).
  \end{multline*}
  Summing gives
  \begin{align} \label{simplified mu}
    \mu_\epsilon(x',y') & =  \int_{u'} \Phi_{\epsilon}(x'-u')\,\Phi_{\epsilon}(y'-u')\, \Bigl( \Phi_\epsilon(u')\bigl(F_\epsilon(y') - F_\epsilon(u')\bigr) +\Phi_{\epsilon}(y') \bigl(F_{\epsilon}(u') - \Theta(u')\bigr) \Bigr)
    \\ \notag & =  \int_{u'} \Phi_{\epsilon}(x'-u')\, \Phi_{\epsilon}(y'-u')\,\Bigl( \d\bigl(F_\epsilon(u')\,F_\epsilon(y')\bigr) - \frac12 \d\bigl( F_\epsilon(u')^2\bigr) - \Phi_{\epsilon}(y')\,\Theta(u') \Bigr)
  \end{align}
  
  Only the first term in equation~(\ref{simplified mu}) contains an $x'$.  By substituting $x'' = x'-u'$ and then $y'' = y' - u'$, we see that
  \begin{align}  \label{xintegral}
  \hspace{.5in} & \hspace{-.5in}  
    \int_{x'}\int_{y'} \mu_\epsilon(x',y')
     \\ \notag & = 
    \int_{x'}\int_{y'}\int_{u'} \Phi_{\epsilon}(x'-u')\, \Phi_{\epsilon}(y'-u')\,\Bigl( \d\bigl(F_\epsilon(u')\,F_\epsilon(y')\bigr) - \frac12 \d\bigl( F_\epsilon(u')^2\bigr) - \Phi_{\epsilon}(y')\,\Theta(u') \Bigr)
     \\ \notag & = 
    \int_{y'}\int_{u'} \int_{x''} \Phi_\epsilon(x'')\, \Phi_{\epsilon}(y'-u')\,\Bigl( \d\bigl(F_\epsilon(u')\,F_\epsilon(y')\bigr) - \frac12 \d\bigl( F_\epsilon(u')^2\bigr) - \Phi_{\epsilon}(y')\,\Theta(u') \Bigr)
     \\ \notag & = 
    \int_{y'}\int_{u'} \Phi_{\epsilon}(y'-u')\,\Bigl( \d\bigl(F_\epsilon(u')\,F_\epsilon(y')\bigr) - \frac12 \d\bigl( F_\epsilon(u')^2\bigr) - \Phi_{\epsilon}(y')\,\Theta(u') \Bigr)
     \\ \notag & = 
    - \int_{u'} \int_{y'}\Phi_{\epsilon}(y'-u')\,\Bigl( \d\bigl(F_\epsilon(u')\,F_\epsilon(y')\bigr) - \frac12 \d\bigl( F_\epsilon(u')^2\bigr) - \Phi_{\epsilon}(y')\,\Theta(u') \Bigr)
     \\ \notag & = 
    - \int_{u'} \int_{y''}\Phi_{\epsilon}(y'')\,\Bigl( \d\bigl(F_\epsilon(u')\,F_\epsilon(y''+u')\bigr) - \frac12 \d\bigl( F_\epsilon(u')^2\bigr) - \Phi_{\epsilon}(y''+u')\,\Theta(u') \Bigr)
     \\ \notag & = 
    - \int_{y''}\Phi_{\epsilon}(y'')\int_{u'} \Bigl( \d\bigl(F_\epsilon(u')\,F_\epsilon(y''+u')\bigr) - \frac12 \d\bigl( F_\epsilon(u')^2\bigr) - \Phi_{\epsilon}(y''+u')\,\Theta(u') \Bigr).
  \end{align}
  We may now evaluate the three integrals in $u'$.  We have:
  \begin{gather} \label{uintegral1}
     \int_{u'}\d\bigl(F_\epsilon(u')\,F_\epsilon(y''+u')\bigr) = 1. \\ \label{uintegral2}
     -\frac12 \int_{u'} \d\bigl( F_\epsilon(u')^2\bigr) = -\frac12. \\ 
     - \int_{u'} \Phi_{\epsilon}(y''+u')\,\Theta(u') = -\int_{u'=0}^\infty \d F_\epsilon(y''+u') = - \bigl( 1 - F_\epsilon(y'')\bigr). \label{uintegral3}
  \end{gather}
  All together, we find:
  \begin{align*}
    \int_{x'}\int_{y'} \mu_\epsilon(x',y') & = - \int_{y''}\Phi_{\epsilon}(y'') \left( 1 - \frac12 - 1 + F_\epsilon(y'')\right)
    \\ & = \frac12 - \int_{y''} \Phi_{\epsilon}(y'') \,F_\epsilon(y'')
    \\ & = \frac12 - \int_{y''} \frac12 \d\bigl( F_\epsilon(y'')^2\bigr) = 0.
  \end{align*}

  Thus $\mu$ is exact, and so we can imagine choosing  a compactly-supported one-form $\lambda_\epsilon(x',y')$ on $(-3\epsilon,3\epsilon)^2$ with $-\d\lambda = \mu$.  Given a choice of such $\lambda$, we can then set
  $$  \,\tikz[baseline=(main.base)]{
    \node[draw,rectangle,inner sep=2pt] (main) {\graphA};
    \draw (main.north) -- ++(0,8pt);
    \draw (main.south) ++(-6pt,0) -- ++(0,-8pt) node[anchor=north] {$\scriptstyle x$};
    \draw (main.south) ++(6pt,0) -- ++(0,-8pt) node[anchor=north] {$\scriptstyle y$};
  }\,
   =
   \int_v \lambda_\epsilon(x-v,y-v)\,\alpha(v),
   $$
   thereby solving equation~(\ref{obstruction for Agraph}).
   
   There is one last generator of $\sh^{\pr}\Frob_1$ whose obstruction must be exact in order to construct a homomorphism to $\qloc$, namely $\graphBbox$.  To complete the proof of Theorem~\ref{mainthm.properads}, we must show that in fact this last obstruction is \emph{not} exact.  Recall that in equations~(\ref{eqn.assoc=0}) and~(\ref{expression for Dgraphbox}) we were able to choose certain generators of $\sh^{\pr}\Frob_1$ to vanish identically.  Then equation~(\ref{Bderivative}) says that the last obstruction is
   $$ \left[ \d, \,\tikz[baseline=(main.base)]{
    \node[draw,rectangle,inner sep=2pt] (main) {
      \graphB
    };
    \draw (main.south) -- ++(0,-8pt) node[anchor=north] {$\scriptstyle x$};
    \draw (main.north) -- ++(0,8pt);
  }\, \right] \want - \,\tikz[baseline=(outerbasepoint)]{
    \node[draw,rectangle,inner sep=2pt] (main) {\graphA};
    \draw (main.north) -- ++(0,8pt) coordinate (top);
    \draw (main.south) ++(-6pt,0) -- ++(0,-10pt);
    \draw (main.south) ++(6pt,0) -- ++(0,-10pt);
    \draw (main.south) ++(0,-18pt) node[draw,inner sep=2pt,fill=white] (lowerbox) {\mult};
    \draw (lowerbox.south) -- ++(0,-10pt) coordinate (bottom) node[anchor=north] {$\scriptstyle x$};
    \path (bottom) -- coordinate (outerbasepoint) (top);
  }\, = -\int_v \lambda_\epsilon(x-v,x-v)\,\alpha(v). \vspace{-24pt}$$
  As for $\graphAbox$, to decide whether this obstruction is exact it suffices to compute $\int_{x'} \lambda_\epsilon(x',x')$.  But Stokes' theorem allows us to turn this into a question about $\mu = \d \lambda$:
  $$ -\int_{x'} \lambda_\epsilon(x',x') = \int_{x'} \int_{y'} \Theta(y'-x') \,\d\lambda_\epsilon(x',y') = \int_{x'} \int_{y'} \Theta(y'-x') \,\mu_\epsilon(x',y'). $$
  Note that $\mu$ is compactly-supported, and so the only boundary contribution to the integral $\int_{x'} \int_{y'} \Theta(y'-x') = \iint_{y' \geq x'}$ comes from the line $y' = x'$.
  
  From equation~(\ref{simplified mu}), we find that we want to compute:
  \begin{align*} 
  \hspace{.5in} & \hspace{-.5in}  
   \int_{x'} \int_{y'} \Theta(y'-x') \,\mu_\epsilon(x',y')
   \\ & =
   \int_{x'} \int_{y'} \Theta(y'-x') \int_{u'} \Phi_{\epsilon}(x'-u') \Phi_{\epsilon}(y'-u')\Bigl( \d\bigl(F_\epsilon(u')F_\epsilon(y')\bigr) - \frac12 \d\bigl( F_\epsilon(u')^2\bigr) - \Phi_{\epsilon}(y')\Theta(u') \Bigr).
\intertext{We will proceed as in equation~(\ref{xintegral}), first rearranging and substituting $x'' = x'-u'$ and $y'' = y'-u'$, and then computing the integral in $x''$:}
    & = 
   -\int_{u'}  \int_{y'} \int_{x'}\Theta(y'-x')  \Phi_{\epsilon}(x'-u') \Phi_{\epsilon}(y'-u')\Bigl( \d\bigl(F_\epsilon(u')F_\epsilon(y')\bigr) - \frac12 \d\bigl( F_\epsilon(u')^2\bigr) - \Phi_{\epsilon}(y')\Theta(u') \Bigr)
   \\ & = 
   -\int_{u'}  \int_{y''} \int_{x''}\Theta(y''-x'')  \Phi_{\epsilon}(x'') \Phi_{\epsilon}(y'')\Bigl( \d\bigl(F_\epsilon(u')F_\epsilon(y''+u')\bigr) - \frac12 \d\bigl( F_\epsilon(u')^2\bigr) - \Phi_{\epsilon}(y''+u')\Theta(u') \Bigr)
   \\ & = 
   -\int_{u'}  \int_{y''} F_\epsilon(y'')\, \Phi_{\epsilon}(y'')\Bigl( \d\bigl(F_\epsilon(u')F_\epsilon(y''+u')\bigr) - \frac12 \d\bigl( F_\epsilon(u')^2\bigr) - \Phi_{\epsilon}(y''+u')\Theta(u') \Bigr)
   \\ & = 
   - \int_{y''} F_\epsilon(y'')\, \Phi_{\epsilon}(y'')\int_{u'} \Bigl( \d\bigl(F_\epsilon(u')F_\epsilon(y''+u')\bigr) - \frac12 \d\bigl( F_\epsilon(u')^2\bigr) - \Phi_{\epsilon}(y''+u')\Theta(u') \Bigr).
\intertext{From equations~(\ref{uintegral1}), (\ref{uintegral2}), and~(\ref{uintegral3}), we find:}
    & = 
   - \int_{y''} F_\epsilon(y'')\, \Phi_{\epsilon}(y'') \left( 1 - \frac12 - 1 + F_\epsilon(y'')\right)
   \\ & = 
   \frac12 \int_{y''} F_\epsilon(y'')\, \Phi_{\epsilon}(y'') - \int_{y''} F_\epsilon(y'')^2\, \Phi_{\epsilon}(y'')
   \\ & =
   \frac14 \int_{y''} \d \bigl( F_\epsilon(y'')^2\bigr) - \frac13 \int_{y''} \d \bigl( F_\epsilon(y'')^3\bigr) 
   \\ & =
   \frac14 - \frac13 = -\frac1{12}.
  \end{align*}
  Since this is not zero,  $\lambda$ is not exact, and the proof is complete.  
\end{proof}

The above computations show that the action of $\partial \left( \graphBbox\right)$ on $\H^\bullet(S^1)$ is precisely multiplication by $-\frac1{12}$.  This fails to be exact whether it is considered as a quasilocal operation or not.
It is worth pointing out, then,  the following  corollary:
\begin{corollary}
  Any properadic lift to the cochain level of the 1-shifted Frobenius algebra structure on $\H^\bullet(S^1)$ must necessarily include a non-quasilocal action of one of the generators 
  
  \mbox{}\hfill $ \,
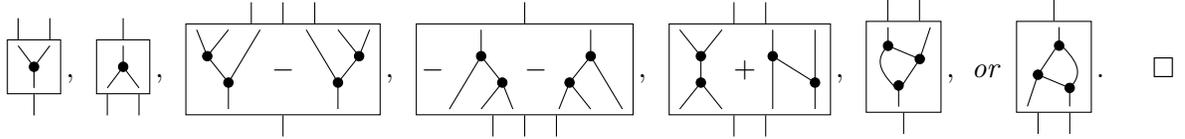
\tikz[baseline=(main.base)]{
    \node[draw,rectangle,inner sep=2pt] (main) {\mult};
    \draw (main.south) -- ++(0,-8pt);
    \draw (main.north) ++(-6pt,0) -- ++(0,8pt);
    \draw (main.north) ++(6pt,0) -- ++(0,8pt);
  }\,, \hspace{1ex} 
  \,\tikz[baseline=(main.base)]{
    \node[draw,rectangle,inner sep=2pt] (main) {\comult};
    \draw (main.north) -- ++(0,8pt);
    \draw (main.south) ++(-6pt,0) -- ++(0,-8pt);
    \draw (main.south) ++(6pt,0) -- ++(0,-8pt);
  }\,, \hspace{1ex}
  \,\tikz[baseline=(main.base)]{
    \node[draw,rectangle,inner sep=2pt] (main) {$
        \,\tikz[baseline=(basepoint)]{ 
    \path (0,3pt) coordinate (basepoint) (0,1pt) node[dot] {} (-8pt,11pt) node[dot] {};
    \draw[](0,-9pt) -- (0,1pt);
    \draw[](0,1pt) -- (-8pt,11pt);
    \draw[](-8pt,11pt) -- (-12pt,21pt);
    \draw[](-8pt,11pt) -- (0pt,21pt);
    \draw[](0,1pt) -- (12pt,21pt);
  }\,
  -
  \,\tikz[baseline=(basepoint)]{ 
    \path (0,3pt) coordinate (basepoint) (0,1pt) node[dot] {} (8pt,11pt) node[dot] {};
    \draw[](0,-9pt) -- (0,1pt);
    \draw[](0,1pt) -- (8pt,11pt);
    \draw[](0,1pt) -- (-12pt,21pt);
    \draw[](8pt,11pt) -- (0pt,21pt);
    \draw[](8pt,11pt) -- (12pt,21pt);
  }\,
    $};
    \draw (main.south) -- ++(0,-8pt);
    \draw (main.north) ++(-12pt,0) -- ++(0,8pt);
    \draw (main.north) ++(0pt,0) -- ++(0,8pt);
    \draw (main.north) ++(12pt,0) -- ++(0,8pt);
  }\,
  ,\hspace{1ex}
  \,\tikz[baseline=(main.base)]{
    \node[draw,rectangle,inner sep=2pt] (main) {$
    -  \,\tikz[baseline=(basepoint)]{ 
    \path (0,3pt) coordinate (basepoint) (8pt,1pt) node[dot] {} (0,11pt) node[dot] {};
    \draw[](-12pt,-9pt) -- (0,11pt);
    \draw[](0pt,-9pt) -- (8pt,1pt);
    \draw[](12pt,-9pt) -- (8pt,1pt);
    \draw[](8pt,1pt) -- (0,11pt);
    \draw[](0,11pt) -- (0,21pt);
  }\,
  -
  \,\tikz[baseline=(basepoint)]{ 
    \path (0,3pt) coordinate (basepoint) (-8pt,1pt) node[dot] {} (0,11pt) node[dot] {};
    \draw[](12pt,-9pt) -- (0,11pt);
    \draw[](-12pt,-9pt) -- (-8pt,1pt);
    \draw[](0pt,-9pt) -- (-8pt,1pt);
    \draw[](-8pt,1pt) -- (0,11pt);
    \draw[](0,11pt) -- (0,21pt);
  }\,
    $};
    \draw (main.north) -- ++(0,8pt);
    \draw (main.south) ++(-12pt,0) -- ++(0,-8pt);
    \draw (main.south) ++(0pt,0) -- ++(0,-8pt);
    \draw (main.south) ++(12pt,0) -- ++(0,-8pt);
  }\,
  ,\hspace{1ex}
  \,\tikz[baseline=(main.base)]{
    \node[draw,rectangle,inner sep=2pt] (main) {$
          \,\tikz[baseline=(basepoint)]{ 
    \path (0,3pt) coordinate (basepoint) (0,1pt) node[dot] {} (0,11pt) node[dot] {};
    \draw[](-8pt,-9pt) -- (0,1pt);
    \draw[](8pt,-9pt) -- (0,1pt);
    \draw[](0,1pt) -- (0,11pt);
    \draw[](0,11pt) -- (-8pt,21pt);
    \draw[](0,11pt) -- (8pt,21pt);
  }\,
  +
  \,\tikz[baseline=(basepoint)]{ 
    \path (0,3pt) coordinate (basepoint) (8pt,1pt) node[dot] {} (-8pt,11pt) node[dot] {};
    \draw[](-8pt,-9pt) -- (-8pt,11pt);
    \draw[](8pt,-9pt) -- (8pt,1pt);
    \draw[](8pt,1pt) -- (-8pt,11pt);
    \draw[](-8pt,11pt) -- (-8pt,21pt);
    \draw[](8pt,1pt) -- (8pt,21pt);
  }\,
    $};
    \draw (main.south) ++(-6pt,0) -- ++(0pt,-8pt);
    \draw (main.south) ++(6pt,0) -- ++(0pt,-8pt);
    \draw (main.north) ++(-6pt,0) -- ++(0,8pt);
    \draw (main.north) ++(6pt,0) -- ++(0,8pt);
  }\,
  ,\hspace{1ex}
  \graphDbox
  ,\hspace{1ex}
  \text{or}\hspace{1ex}
  \graphAbox
  .
  $ \qedhere
  \label{a corollary}
\end{corollary}

\section{Other cochain models}  \label{section.conclusion}

There are undoubtedly some readers who prefer combinatorics to calculus and who therefore might wonder if Theorems~\ref{mainthm.dioperads} and~\ref{mainthm.properads} are special to the choice of de Rham forms as the model for cochains.  Do the same results hold if $\Omega^\bullet(S^1)$ is replaced by, say, cellular cochains for a fine subdivision of $S^1$?  The answer in a technical sense is ``no,'' because it is difficult to define a properad of quasilocal operations in  such a model.  But this is really the only problem.  The goal of this section is to illustrate what happens for cellular cochains, while being a bit \emph{ad hoc} about ``quasilocality.''

Consider then subdividing $S^{1}$ into some large number $N \gg 0$ of cells.  We can coordinatize $S^{1}$ by identifying the vertices with $\bZ/N\bZ$; the intervals connecting them then correspond to the set $\frac12 + \bZ/N\bZ$.  Let us write $\C^{\bullet}(S^{1})$ for the cellular cochains for this cell decomposition.  It has a basis consisting of the functions $f_{x}, x\in \bZ/N\bZ$ and the 1-cochains $g_{x+\frac12}, x\in \bZ/N\bZ$, where $f_{x}$ is non-zero only at vertex $x$, and $g_{x+\frac12}$ is non-zero only at the edge connecting vertex $x$ with vertex $x+1$.  The 1-cochains $g_{x+\frac12}$ are, of course, all closed.  The derivatives of the $0$-cochains are:
$$ \d f_{x} = g_{x-\frac12} - g_{x+\frac12}. $$
The data of an operation in $\End(\C^{\bullet}(S^{1}))(m,n)$ then consists of a matrix whose entries are indexed by the abelian group $(\frac12 \bZ/N\bZ)^{\times(m+n)}$.

The set of vertices $\bZ/N\bZ$ has a natural metric, in which vertex $x$ and vertex $x+1$ are at distance~$1$ from each other.  Let us extend this to a non-symmetric ``metric'' on the set of basic cochains $\{f_{x}\}_{x\in \bZ/N\bZ} \cup \{g_{y}\}_{y\in \frac12 + \bZ/N\bZ}$ be declaring that the ``distance'' from $f_{x}$ to $g_{x\pm \frac12}$ is $0$, but the distance from $g_{x\pm \frac12}$ to $f_{x}$ is $1$.  The idea of this ``metric'' is that the distance from a cochain $a$ to a cochain $b$ is $\max_{x\in a}\min_{y\in b}|y-x|$.  The symmetrization of this ``metric'' is the usual metric on $\frac12\bZ/N\bZ$.

\begin{definition} \label{defn.ell-quasilocal}
  For a ``length scale'' $\ell \in \bN$, say that an $m$-to-$n$ operation $P \in \End(\C^{\bullet}(S^{1}))(m,n)$ is \define{$\ell$-quasilocal} if it enjoys the following property.  Suppose that for $\vec x \in (\frac12 \bZ / N\bZ)^{\times m} $ and $\vec y \in (\frac12 \bZ / N\bZ)^{\times n} $, the $(\vec x,\vec y)$th matrix entry of $P$ is non-zero.  Then for every $i \in\{1,\dots,m\}$ and every $j \in \{1,\dots,n\}$, the ``distance'' for the above non-symmetric metric from $x_{i}$ to $y_{j}$ is at most $\ell$.
\end{definition}

For example, every operation in $\End(\C^{\bullet}(S^{1}))$ is $N$-quasilocal, and any properadic composition of an $\ell$-quasilocal operation with an $\ell'$-quasilocal operation is $(\ell + \ell'+1)$-quasilocal.  ({Operadic} compositions, in which every operation has only one output, do not require the ``$+1$,'' but in general properadic compositions do.  PROPic compositions, in which graphs need not be connected, can completely destroy quasilocality as we have defined it.)  The point of using the non-symmetric version of ``distance'' in Definition~\ref{defn.ell-quasilocal} is that differentiation $\d$ is $0$-quasilocal, and that the set $\qloc_{\ell}(m,n)$ of $\ell$-quasilocal $m$-to-$n$ operations is a cochain complex, and in fact a sub-$\bS$-bimodule of $\End(\C^{\bullet}(S^{1}))(m,n)$.

An argument analogous to Proposition~\ref{prop.Hqloc} implies:
\begin{proposition}\label{prop.Hqloc-discrete}
  Suppose that $\ell \ll N$.  Then $\H^{\bullet}\qloc_{\ell}(m,n) \cong \H^{\bullet}(S^{1})[1-n]$, and the inclusion $\qloc_{\ell}(m,n) \hookrightarrow \qloc_{\ell+1}(m,n)$ is a quasiisomorphism. \qedhere
\end{proposition}
At some cutoff depending on $m,n,N$, eventually $\ell$ becomes too large and Proposition~\ref{prop.Hqloc-discrete} fails.  The cutoff grows linearly with $N$.

Since the length scale at which an operation is quasilocal grows under compositions, there is no reasonable properad of ``$(\ll N)$-quasilocal operations.''  But we can nevertheless ask for homomorphisms $\sh^{\di}\Frob_{1} \to \End(\C^{\bullet}(S^{1}))$ or $\sh^{\pr}\Frob_{1} \to \End(\C^{\bullet}(S^{1}))$ lifting the action on cohomology in which the ``first few'' generators are $\ell$-quasilocal for $\ell \ll N$.  An argument analogous to the proof of Theorem~\ref{mainthm.dioperads} shows that for $\sh^{\di}\Frob_{1}$, such a request can be satisfied, and that any finite number of generators desired to be among these ``first few'' is allowed provided  $N$ is large enough. 
However, corresponding to Theorem~\ref{mainthm.properads}, we have:

\begin{theorem}  \label{discrete thm.properad}
There does not exist an action of $\sh^{\pr}\Frob_{1}$ on $\C^{\bullet}(S^{1})$ in which the ``first few'' generators act $\ell$-quasilocally for $\ell \ll N$ if the ``first few'' includes $\graphBbox$\vspace{-21pt} (as well as any generators appearing in the derivative of any of the ``first few'').
   Indeed, the obstruction to defining $\graphBbox$\vspace{-15pt} will act as $-\frac1{12}\id$ on $\H^{\bullet}(S^{1})$.
\end{theorem}

Our proof of Theorem~\ref{discrete thm.properad} will follow the same overall logic as our proof of Theorem~\ref{mainthm.properads}, although the calculations will be different.

  Together with Corollary~\ref{a corollary}, we find the following general result about cochain-level Frobenius-algebra structures.  (We call it a ``conclusion,'' and provide a ``justification,'' because we will not try to make it completely precise.)
\begin{metacorollary}\label{meta}
  For any ``geometric'' cochain model on $S^{1}$, every homotopy Frobenius-algebra structure, even in the dioperadic sense, must include a nontrivial homotopy controlling at least one of associativity, coassociativity, or the Frobenius axiom.
\end{metacorollary}
\begin{proof}[Justification]
  By a ``geometric'' cochain model, we mean at minimum that it should admit some sort of notion of ``quasilocality'' satisfying a version of Propositions~\ref{prop.Hqloc} and~\ref{prop.Hqloc-discrete}.  If so, one can ask whether an $\sh^{\pr}\Frob_{1}$ structure exists in which the generators up to $\graphBbox$\vspace{-21pt} are quasilocal, and the answer will be ``no.''  In particular, any quasilocal choices for the generators before $\graphBbox$\vspace{-21pt} leads to the same non-zero obstruction for $\graphBbox$, by some zig-zag of cochain models 
  comparing the chosen one with  de Rham forms or cellular cochains.
  
  But suppose that one can choose all three homotopies controlling associativity, coassociativity, and the Frobenius axiom to vanish identically.  Then the obstructions for $\graphDbox$\vspace{-18pt} and $\graphAbox$ vanish identically, and so we are allowed to set $\graphDbox$\vspace{-21pt} and $\graphAbox$ both identically zero, in which case the obstruction for $\graphBbox$ vanishes.
\end{proof}

Similar results in arbitrary dimensions are suggested by~\cite[Conjecture 3.5]{PoissonAKSZ}, which predicts that,
although quasilocal $\sh^{\pr}\Frob_{d}$ structures exist at cochain level on $(d>1)$-dimensional manifolds,
there is no \emph{canonical} such structure.
But a canonical structure would exist if the first three homotopies could be made to vanish.

\begin{proof}[Proof of Theorem~\ref{discrete thm.properad}]
  The cohomological degree arguments from the proof of Theorem~\ref{mainthm.properads} carry through verbatim, and imply that as long as we stay in the $(\ll N)$-quasilocal regime, no choices will affect whether later obstructions vanish.  We need therefore only to choose  solutions to some inductive system of equations, and calculate some obstructions for that particular system of choices.

  Our convention will be to write only the non-zero matrix coefficients of operations.  We begin by choosing $0$-quasilocal lifts of the multiplication and comultiplication:
  \begin{align*}
    \,\tikz[baseline=(main.base)]{
    \node[draw,rectangle,inner sep=2pt] (main) {\mult};
    \draw (main.south) -- ++(0,-8pt);
    \draw (main.north) ++(-6pt,0) -- ++(0,8pt);
    \draw (main.north) ++(6pt,0) -- ++(0,8pt);
  }\,
  & : \quad f_{x}\otimes f_{x}\mapsto f_{x},\quad f_{x}\otimes g_{x\pm \frac12} \mapsto \frac12g_{x\pm \frac12}, \quad g_{x\pm \frac12} \otimes f_{x} \mapsto \frac12g_{x\pm \frac12},
  \\
  \,\tikz[baseline=(main.base)]{
    \node[draw,rectangle,inner sep=2pt] (main) {\comult};
    \draw (main.north) -- ++(0,8pt);
    \draw (main.south) ++(-6pt,0) -- ++(0,-8pt);
    \draw (main.south) ++(6pt,0) -- ++(0,-8pt);
  }\,
  & : \quad f_{x}\mapsto \frac12(g_{x-\frac12} + g_{x + \frac12})\otimes f_{x} - f_{x}\otimes \frac12(g_{x-\frac12} + g_{x + \frac12}), \quad g_{x+\frac12} \mapsto g_{x+\frac12} \otimes g_{x+\frac12}.
  \end{align*}
  
  The homotopy controlling associativity should then satisfy:
  \begin{equation*}
  \left[\d,
  \,\tikz[baseline=(main.base)]{
    \node[draw,rectangle,inner sep=2pt] (main) {$
        \,\tikz[baseline=(basepoint)]{ 
    \path (0,3pt) coordinate (basepoint) (0,1pt) node[dot] {} (-8pt,11pt) node[dot] {};
    \draw[](0,-9pt) -- (0,1pt);
    \draw[](0,1pt) -- (-8pt,11pt);
    \draw[](-8pt,11pt) -- (-12pt,21pt);
    \draw[](-8pt,11pt) -- (0pt,21pt);
    \draw[](0,1pt) -- (12pt,21pt);
  }\,
  -
  \,\tikz[baseline=(basepoint)]{ 
    \path (0,3pt) coordinate (basepoint) (0,1pt) node[dot] {} (8pt,11pt) node[dot] {};
    \draw[](0,-9pt) -- (0,1pt);
    \draw[](0,1pt) -- (8pt,11pt);
    \draw[](0,1pt) -- (-12pt,21pt);
    \draw[](8pt,11pt) -- (0pt,21pt);
    \draw[](8pt,11pt) -- (12pt,21pt);
  }\,
    $};
    \draw (main.south) -- ++(0,-8pt);
    \draw (main.north) ++(-12pt,0) -- ++(0,8pt);
    \draw (main.north) ++(0pt,0) -- ++(0,8pt);
    \draw (main.north) ++(12pt,0) -- ++(0,8pt);
  }\,
  \right]
  \dwant  \quad \begin{cases} \displaystyle
  g_{x\pm \frac12} \otimes f_{x} \otimes f_{x} \mapsto \frac14g_{x\pm\frac12}, & \displaystyle \quad f_{x} \otimes f_{x} \otimes g_{x\pm \frac12} \mapsto -\frac14g_{x\pm\frac12},
  \\[12pt] \displaystyle
  g_{x\pm \frac12} \otimes f_{x} \otimes f_{x\pm1} \mapsto -\frac14g_{x\pm\frac12}, & \displaystyle \quad f_{x}\otimes f_{x\pm 1}\otimes g_{x\pm \frac12} \mapsto \frac14g_{x\pm \frac12}. \end{cases}
  \end{equation*}
  Indeed, when all inputs are an ``$f$,'' $\,\tikz[baseline=(main.base)]{
    \node[draw,rectangle,inner sep=2pt] (main) {\mult};
    \draw (main.south) -- ++(0,-8pt);
    \draw (main.north) ++(-6pt,0) -- ++(0,8pt);
    \draw (main.north) ++(6pt,0) -- ++(0,8pt);
  }\,$ is strictly associative, and when at least two inputs are a ``$g$,'' both terms in equation~(\ref{eqn.associator}) vanish for degree reasons.  Considering the remaining cases gives the above as the only nontrivial matrix coefficients.  (Repeated $\pm$ signs in the same formula should all agree.)
  
  One choice for the primitive is:
  \begin{equation*}
  \,\tikz[baseline=(main.base)]{
    \node[draw,rectangle,inner sep=2pt] (main) {$
        \,\tikz[baseline=(basepoint)]{ 
    \path (0,3pt) coordinate (basepoint) (0,1pt) node[dot] {} (-8pt,11pt) node[dot] {};
    \draw[](0,-9pt) -- (0,1pt);
    \draw[](0,1pt) -- (-8pt,11pt);
    \draw[](-8pt,11pt) -- (-12pt,21pt);
    \draw[](-8pt,11pt) -- (0pt,21pt);
    \draw[](0,1pt) -- (12pt,21pt);
  }\,
  -
  \,\tikz[baseline=(basepoint)]{ 
    \path (0,3pt) coordinate (basepoint) (0,1pt) node[dot] {} (8pt,11pt) node[dot] {};
    \draw[](0,-9pt) -- (0,1pt);
    \draw[](0,1pt) -- (8pt,11pt);
    \draw[](0,1pt) -- (-12pt,21pt);
    \draw[](8pt,11pt) -- (0pt,21pt);
    \draw[](8pt,11pt) -- (12pt,21pt);
  }\,
    $};
    \draw (main.south) -- ++(0,-8pt);
    \draw (main.north) ++(-12pt,0) -- ++(0,8pt);
    \draw (main.north) ++(0pt,0) -- ++(0,8pt);
    \draw (main.north) ++(12pt,0) -- ++(0,8pt);
  }\,
  : \quad\begin{cases}  \displaystyle   g_{x\pm \frac12} \otimes g_{x\pm \frac12} \otimes f_{x} \mapsto \pm\frac1{12}g_{x\pm \frac12} , 
  \\[12pt] \displaystyle     g_{x\pm \frac12}  \otimes f_{x} \otimes g_{x\pm \frac12}    \mapsto \pm\frac1{6}g_{x\pm \frac12} , 
  \\[12pt] \displaystyle  f_{x} \otimes  g_{x\pm \frac12}  \otimes g_{x\pm \frac12}    \mapsto \pm\frac1{12}g_{x\pm \frac12}.
  \end{cases}
  \end{equation*}
  To check this, we can compute:
  \begin{multline*}
    \left[\d,
  \,\tikz[baseline=(main.base)]{
    \node[draw,rectangle,inner sep=2pt] (main) {$
        \,\tikz[baseline=(basepoint)]{ 
    \path (0,3pt) coordinate (basepoint) (0,1pt) node[dot] {} (-8pt,11pt) node[dot] {};
    \draw[](0,-9pt) -- (0,1pt);
    \draw[](0,1pt) -- (-8pt,11pt);
    \draw[](-8pt,11pt) -- (-12pt,21pt);
    \draw[](-8pt,11pt) -- (0pt,21pt);
    \draw[](0,1pt) -- (12pt,21pt);
  }\,
  -
  \,\tikz[baseline=(basepoint)]{ 
    \path (0,3pt) coordinate (basepoint) (0,1pt) node[dot] {} (8pt,11pt) node[dot] {};
    \draw[](0,-9pt) -- (0,1pt);
    \draw[](0,1pt) -- (8pt,11pt);
    \draw[](0,1pt) -- (-12pt,21pt);
    \draw[](8pt,11pt) -- (0pt,21pt);
    \draw[](8pt,11pt) -- (12pt,21pt);
  }\,
    $};
    \draw (main.south) -- ++(0,-8pt);
    \draw (main.north) ++(-12pt,0) -- ++(0,8pt);
    \draw (main.north) ++(0pt,0) -- ++(0,8pt);
    \draw (main.north) ++(12pt,0) -- ++(0,8pt);
  }\,
  \right]
  \left( g_{x\pm \frac12} \otimes f_{x} \otimes f_{x}\right) 
  = 
  \,\tikz[baseline=(main.base)]{
    \node[draw,rectangle,inner sep=2pt] (main) {$
        \,\tikz[baseline=(basepoint)]{ 
    \path (0,3pt) coordinate (basepoint) (0,1pt) node[dot] {} (-8pt,11pt) node[dot] {};
    \draw[](0,-9pt) -- (0,1pt);
    \draw[](0,1pt) -- (-8pt,11pt);
    \draw[](-8pt,11pt) -- (-12pt,21pt);
    \draw[](-8pt,11pt) -- (0pt,21pt);
    \draw[](0,1pt) -- (12pt,21pt);
  }\,
  -
  \,\tikz[baseline=(basepoint)]{ 
    \path (0,3pt) coordinate (basepoint) (0,1pt) node[dot] {} (8pt,11pt) node[dot] {};
    \draw[](0,-9pt) -- (0,1pt);
    \draw[](0,1pt) -- (8pt,11pt);
    \draw[](0,1pt) -- (-12pt,21pt);
    \draw[](8pt,11pt) -- (0pt,21pt);
    \draw[](8pt,11pt) -- (12pt,21pt);
  }\,
    $};
    \draw (main.south) -- ++(0,-8pt);
    \draw (main.north) ++(-12pt,0) -- ++(0,8pt);
    \draw (main.north) ++(0pt,0) -- ++(0,8pt);
    \draw (main.north) ++(12pt,0) -- ++(0,8pt);
  }\,
  \left( \d(g_{x\pm \frac12} \otimes f_{x} \otimes f_{x})\right)
  \\ = 
  \,\tikz[baseline=(main.base)]{
    \node[draw,rectangle,inner sep=2pt] (main) {$
        \,\tikz[baseline=(basepoint)]{ 
    \path (0,3pt) coordinate (basepoint) (0,1pt) node[dot] {} (-8pt,11pt) node[dot] {};
    \draw[](0,-9pt) -- (0,1pt);
    \draw[](0,1pt) -- (-8pt,11pt);
    \draw[](-8pt,11pt) -- (-12pt,21pt);
    \draw[](-8pt,11pt) -- (0pt,21pt);
    \draw[](0,1pt) -- (12pt,21pt);
  }\,
  -
  \,\tikz[baseline=(basepoint)]{ 
    \path (0,3pt) coordinate (basepoint) (0,1pt) node[dot] {} (8pt,11pt) node[dot] {};
    \draw[](0,-9pt) -- (0,1pt);
    \draw[](0,1pt) -- (8pt,11pt);
    \draw[](0,1pt) -- (-12pt,21pt);
    \draw[](8pt,11pt) -- (0pt,21pt);
    \draw[](8pt,11pt) -- (12pt,21pt);
  }\,
    $};
    \draw (main.south) -- ++(0,-8pt);
    \draw (main.north) ++(-12pt,0) -- ++(0,8pt);
    \draw (main.north) ++(0pt,0) -- ++(0,8pt);
    \draw (main.north) ++(12pt,0) -- ++(0,8pt);
  }\,
  \left( -g_{x\pm \frac12} \otimes (g_{x-\frac12} - g_{x + \frac12}) \otimes f_{x} -g_{x\pm \frac12} \otimes f_{x}  \otimes (g_{x-\frac12} - g_{x + \frac12})  \right)
  \\ = \pm \left( \pm \frac1{12}\right)g_{x\pm \frac12} \pm\left( \pm \frac16\right) g_{x\pm \frac12} = \frac14  g_{x\pm \frac12}.
  \end{multline*}
  \begin{multline*}
    \left[\d,
  \,\tikz[baseline=(main.base)]{
    \node[draw,rectangle,inner sep=2pt] (main) {$
        \,\tikz[baseline=(basepoint)]{ 
    \path (0,3pt) coordinate (basepoint) (0,1pt) node[dot] {} (-8pt,11pt) node[dot] {};
    \draw[](0,-9pt) -- (0,1pt);
    \draw[](0,1pt) -- (-8pt,11pt);
    \draw[](-8pt,11pt) -- (-12pt,21pt);
    \draw[](-8pt,11pt) -- (0pt,21pt);
    \draw[](0,1pt) -- (12pt,21pt);
  }\,
  -
  \,\tikz[baseline=(basepoint)]{ 
    \path (0,3pt) coordinate (basepoint) (0,1pt) node[dot] {} (8pt,11pt) node[dot] {};
    \draw[](0,-9pt) -- (0,1pt);
    \draw[](0,1pt) -- (8pt,11pt);
    \draw[](0,1pt) -- (-12pt,21pt);
    \draw[](8pt,11pt) -- (0pt,21pt);
    \draw[](8pt,11pt) -- (12pt,21pt);
  }\,
    $};
    \draw (main.south) -- ++(0,-8pt);
    \draw (main.north) ++(-12pt,0) -- ++(0,8pt);
    \draw (main.north) ++(0pt,0) -- ++(0,8pt);
    \draw (main.north) ++(12pt,0) -- ++(0,8pt);
  }\,
  \right]
  \left( g_{x\pm \frac12} \otimes f_{x} \otimes f_{x \pm 1}\right) 
  = 
  \,\tikz[baseline=(main.base)]{
    \node[draw,rectangle,inner sep=2pt] (main) {$
        \,\tikz[baseline=(basepoint)]{ 
    \path (0,3pt) coordinate (basepoint) (0,1pt) node[dot] {} (-8pt,11pt) node[dot] {};
    \draw[](0,-9pt) -- (0,1pt);
    \draw[](0,1pt) -- (-8pt,11pt);
    \draw[](-8pt,11pt) -- (-12pt,21pt);
    \draw[](-8pt,11pt) -- (0pt,21pt);
    \draw[](0,1pt) -- (12pt,21pt);
  }\,
  -
  \,\tikz[baseline=(basepoint)]{ 
    \path (0,3pt) coordinate (basepoint) (0,1pt) node[dot] {} (8pt,11pt) node[dot] {};
    \draw[](0,-9pt) -- (0,1pt);
    \draw[](0,1pt) -- (8pt,11pt);
    \draw[](0,1pt) -- (-12pt,21pt);
    \draw[](8pt,11pt) -- (0pt,21pt);
    \draw[](8pt,11pt) -- (12pt,21pt);
  }\,
    $};
    \draw (main.south) -- ++(0,-8pt);
    \draw (main.north) ++(-12pt,0) -- ++(0,8pt);
    \draw (main.north) ++(0pt,0) -- ++(0,8pt);
    \draw (main.north) ++(12pt,0) -- ++(0,8pt);
  }\,
  \left( \d(g_{x\pm \frac12} \otimes f_{x} \otimes f_{x \pm 1})\right)
  \\ = 
    \,\tikz[baseline=(main.base)]{
    \node[draw,rectangle,inner sep=2pt] (main) {$
        \,\tikz[baseline=(basepoint)]{ 
    \path (0,3pt) coordinate (basepoint) (0,1pt) node[dot] {} (-8pt,11pt) node[dot] {};
    \draw[](0,-9pt) -- (0,1pt);
    \draw[](0,1pt) -- (-8pt,11pt);
    \draw[](-8pt,11pt) -- (-12pt,21pt);
    \draw[](-8pt,11pt) -- (0pt,21pt);
    \draw[](0,1pt) -- (12pt,21pt);
  }\,
  -
  \,\tikz[baseline=(basepoint)]{ 
    \path (0,3pt) coordinate (basepoint) (0,1pt) node[dot] {} (8pt,11pt) node[dot] {};
    \draw[](0,-9pt) -- (0,1pt);
    \draw[](0,1pt) -- (8pt,11pt);
    \draw[](0,1pt) -- (-12pt,21pt);
    \draw[](8pt,11pt) -- (0pt,21pt);
    \draw[](8pt,11pt) -- (12pt,21pt);
  }\,
    $};
    \draw (main.south) -- ++(0,-8pt);
    \draw (main.north) ++(-12pt,0) -- ++(0,8pt);
    \draw (main.north) ++(0pt,0) -- ++(0,8pt);
    \draw (main.north) ++(12pt,0) -- ++(0,8pt);
  }\,
  \left(
    -g_{x\pm \frac12} \otimes (g_{x-\frac12} - g_{x+\frac12}) \otimes f_{x \pm 1} - g_{x\pm \frac12} \otimes f_{x} \otimes (g_{x \pm 1 - \frac12} - g_{x \pm 1 + \frac12})
  \right)
  \\ = \pm \left( \mp \frac1{12} \right)g_{x \pm \frac12} \mp \left( \pm \frac16 \right) g_{x \pm \frac12} = -\frac14 g_{x\pm \frac12}.
  \end{multline*}
  \begin{multline*}
    \left[\d,
  \,\tikz[baseline=(main.base)]{
    \node[draw,rectangle,inner sep=2pt] (main) {$
        \,\tikz[baseline=(basepoint)]{ 
    \path (0,3pt) coordinate (basepoint) (0,1pt) node[dot] {} (-8pt,11pt) node[dot] {};
    \draw[](0,-9pt) -- (0,1pt);
    \draw[](0,1pt) -- (-8pt,11pt);
    \draw[](-8pt,11pt) -- (-12pt,21pt);
    \draw[](-8pt,11pt) -- (0pt,21pt);
    \draw[](0,1pt) -- (12pt,21pt);
  }\,
  -
  \,\tikz[baseline=(basepoint)]{ 
    \path (0,3pt) coordinate (basepoint) (0,1pt) node[dot] {} (8pt,11pt) node[dot] {};
    \draw[](0,-9pt) -- (0,1pt);
    \draw[](0,1pt) -- (8pt,11pt);
    \draw[](0,1pt) -- (-12pt,21pt);
    \draw[](8pt,11pt) -- (0pt,21pt);
    \draw[](8pt,11pt) -- (12pt,21pt);
  }\,
    $};
    \draw (main.south) -- ++(0,-8pt);
    \draw (main.north) ++(-12pt,0) -- ++(0,8pt);
    \draw (main.north) ++(0pt,0) -- ++(0,8pt);
    \draw (main.north) ++(12pt,0) -- ++(0,8pt);
  }\,
  \right]
  \left( f_{x} \otimes g_{x\pm \frac1{12}} \otimes f_{x}\right) 
  = 
  \,\tikz[baseline=(main.base)]{
    \node[draw,rectangle,inner sep=2pt] (main) {$
        \,\tikz[baseline=(basepoint)]{ 
    \path (0,3pt) coordinate (basepoint) (0,1pt) node[dot] {} (-8pt,11pt) node[dot] {};
    \draw[](0,-9pt) -- (0,1pt);
    \draw[](0,1pt) -- (-8pt,11pt);
    \draw[](-8pt,11pt) -- (-12pt,21pt);
    \draw[](-8pt,11pt) -- (0pt,21pt);
    \draw[](0,1pt) -- (12pt,21pt);
  }\,
  -
  \,\tikz[baseline=(basepoint)]{ 
    \path (0,3pt) coordinate (basepoint) (0,1pt) node[dot] {} (8pt,11pt) node[dot] {};
    \draw[](0,-9pt) -- (0,1pt);
    \draw[](0,1pt) -- (8pt,11pt);
    \draw[](0,1pt) -- (-12pt,21pt);
    \draw[](8pt,11pt) -- (0pt,21pt);
    \draw[](8pt,11pt) -- (12pt,21pt);
  }\,
    $};
    \draw (main.south) -- ++(0,-8pt);
    \draw (main.north) ++(-12pt,0) -- ++(0,8pt);
    \draw (main.north) ++(0pt,0) -- ++(0,8pt);
    \draw (main.north) ++(12pt,0) -- ++(0,8pt);
  }\,
  \left( \d(f_{x} \otimes g_{x\pm \frac12} \otimes f_{x})\right)
  \\ = 
  \,\tikz[baseline=(main.base)]{
    \node[draw,rectangle,inner sep=2pt] (main) {$
        \,\tikz[baseline=(basepoint)]{ 
    \path (0,3pt) coordinate (basepoint) (0,1pt) node[dot] {} (-8pt,11pt) node[dot] {};
    \draw[](0,-9pt) -- (0,1pt);
    \draw[](0,1pt) -- (-8pt,11pt);
    \draw[](-8pt,11pt) -- (-12pt,21pt);
    \draw[](-8pt,11pt) -- (0pt,21pt);
    \draw[](0,1pt) -- (12pt,21pt);
  }\,
  -
  \,\tikz[baseline=(basepoint)]{ 
    \path (0,3pt) coordinate (basepoint) (0,1pt) node[dot] {} (8pt,11pt) node[dot] {};
    \draw[](0,-9pt) -- (0,1pt);
    \draw[](0,1pt) -- (8pt,11pt);
    \draw[](0,1pt) -- (-12pt,21pt);
    \draw[](8pt,11pt) -- (0pt,21pt);
    \draw[](8pt,11pt) -- (12pt,21pt);
  }\,
    $};
    \draw (main.south) -- ++(0,-8pt);
    \draw (main.north) ++(-12pt,0) -- ++(0,8pt);
    \draw (main.north) ++(0pt,0) -- ++(0,8pt);
    \draw (main.north) ++(12pt,0) -- ++(0,8pt);
  }\,
  \left( (g_{x - \frac12} - g_{x + \frac12})\otimes g_{x\pm \frac12} \otimes f_{x} - f_{x} \otimes g_{x\pm \frac12} \otimes (g_{x - \frac12} - g_{x + \frac12}) \right)
  \\ = \mp \left( \pm \frac1{12} \right) g_{x\pm \frac12}  \pm \left( \pm \frac1{12} \right) g_{x\pm \frac12} = 0.
  \end{multline*}
  \begin{multline*}
    \left[\d,
  \,\tikz[baseline=(main.base)]{
    \node[draw,rectangle,inner sep=2pt] (main) {$
        \,\tikz[baseline=(basepoint)]{ 
    \path (0,3pt) coordinate (basepoint) (0,1pt) node[dot] {} (-8pt,11pt) node[dot] {};
    \draw[](0,-9pt) -- (0,1pt);
    \draw[](0,1pt) -- (-8pt,11pt);
    \draw[](-8pt,11pt) -- (-12pt,21pt);
    \draw[](-8pt,11pt) -- (0pt,21pt);
    \draw[](0,1pt) -- (12pt,21pt);
  }\,
  -
  \,\tikz[baseline=(basepoint)]{ 
    \path (0,3pt) coordinate (basepoint) (0,1pt) node[dot] {} (8pt,11pt) node[dot] {};
    \draw[](0,-9pt) -- (0,1pt);
    \draw[](0,1pt) -- (8pt,11pt);
    \draw[](0,1pt) -- (-12pt,21pt);
    \draw[](8pt,11pt) -- (0pt,21pt);
    \draw[](8pt,11pt) -- (12pt,21pt);
  }\,
    $};
    \draw (main.south) -- ++(0,-8pt);
    \draw (main.north) ++(-12pt,0) -- ++(0,8pt);
    \draw (main.north) ++(0pt,0) -- ++(0,8pt);
    \draw (main.north) ++(12pt,0) -- ++(0,8pt);
  }\,
  \right]
  \left( f_{x} \otimes g_{x\pm \frac1{12}} \otimes f_{x\pm 1}\right) 
  = 
  \,\tikz[baseline=(main.base)]{
    \node[draw,rectangle,inner sep=2pt] (main) {$
        \,\tikz[baseline=(basepoint)]{ 
    \path (0,3pt) coordinate (basepoint) (0,1pt) node[dot] {} (-8pt,11pt) node[dot] {};
    \draw[](0,-9pt) -- (0,1pt);
    \draw[](0,1pt) -- (-8pt,11pt);
    \draw[](-8pt,11pt) -- (-12pt,21pt);
    \draw[](-8pt,11pt) -- (0pt,21pt);
    \draw[](0,1pt) -- (12pt,21pt);
  }\,
  -
  \,\tikz[baseline=(basepoint)]{ 
    \path (0,3pt) coordinate (basepoint) (0,1pt) node[dot] {} (8pt,11pt) node[dot] {};
    \draw[](0,-9pt) -- (0,1pt);
    \draw[](0,1pt) -- (8pt,11pt);
    \draw[](0,1pt) -- (-12pt,21pt);
    \draw[](8pt,11pt) -- (0pt,21pt);
    \draw[](8pt,11pt) -- (12pt,21pt);
  }\,
    $};
    \draw (main.south) -- ++(0,-8pt);
    \draw (main.north) ++(-12pt,0) -- ++(0,8pt);
    \draw (main.north) ++(0pt,0) -- ++(0,8pt);
    \draw (main.north) ++(12pt,0) -- ++(0,8pt);
  }\,
  \left( \d(f_{x} \otimes g_{x\pm \frac12} \otimes f_{x \pm 1})\right)
  \\ = 
  \,\tikz[baseline=(main.base)]{
    \node[draw,rectangle,inner sep=2pt] (main) {$
        \,\tikz[baseline=(basepoint)]{ 
    \path (0,3pt) coordinate (basepoint) (0,1pt) node[dot] {} (-8pt,11pt) node[dot] {};
    \draw[](0,-9pt) -- (0,1pt);
    \draw[](0,1pt) -- (-8pt,11pt);
    \draw[](-8pt,11pt) -- (-12pt,21pt);
    \draw[](-8pt,11pt) -- (0pt,21pt);
    \draw[](0,1pt) -- (12pt,21pt);
  }\,
  -
  \,\tikz[baseline=(basepoint)]{ 
    \path (0,3pt) coordinate (basepoint) (0,1pt) node[dot] {} (8pt,11pt) node[dot] {};
    \draw[](0,-9pt) -- (0,1pt);
    \draw[](0,1pt) -- (8pt,11pt);
    \draw[](0,1pt) -- (-12pt,21pt);
    \draw[](8pt,11pt) -- (0pt,21pt);
    \draw[](8pt,11pt) -- (12pt,21pt);
  }\,
    $};
    \draw (main.south) -- ++(0,-8pt);
    \draw (main.north) ++(-12pt,0) -- ++(0,8pt);
    \draw (main.north) ++(0pt,0) -- ++(0,8pt);
    \draw (main.north) ++(12pt,0) -- ++(0,8pt);
  }\,
  \left( (g_{x - \frac12} - g_{x + \frac12})\otimes g_{x\pm \frac12} \otimes f_{x\pm 1} - f_{x} \otimes g_{x\pm \frac12} \otimes (g_{x \pm 1 - \frac12} - g_{x \pm 1 + \frac12}) \right)
  \\ = \mp \left( \mp \frac1{12} \right) g_{x\pm \frac12}  \pm \left( \mp \frac1{12} \right) g_{x\pm \frac12} = 0.
  \end{multline*}
  The computations for inputs of the form $f\otimes f\otimes g$ agree with those for $g\otimes f \otimes g$, by symmetry.  All other matrix coefficients vanish either for degree reasons or for quasilocality reasons.  Direct calculation verifies that this proposed homotopy controlling associativity in fact generates the two-dimensional irrep of $\bS_{3}$.
  
  The homotopy controlling coassociativity should satisfy:
  $$ \hspace{-1in}
  \left[ \d,   \,\tikz[baseline=(main.base)]{
    \node[draw,rectangle,inner sep=2pt] (main) {$
    -  \,\tikz[baseline=(basepoint)]{ 
    \path (0,3pt) coordinate (basepoint) (8pt,1pt) node[dot] {} (0,11pt) node[dot] {};
    \draw[](-12pt,-9pt) -- (0,11pt);
    \draw[](0pt,-9pt) -- (8pt,1pt);
    \draw[](12pt,-9pt) -- (8pt,1pt);
    \draw[](8pt,1pt) -- (0,11pt);
    \draw[](0,11pt) -- (0,21pt);
  }\,
  -
  \,\tikz[baseline=(basepoint)]{ 
    \path (0,3pt) coordinate (basepoint) (-8pt,1pt) node[dot] {} (0,11pt) node[dot] {};
    \draw[](12pt,-9pt) -- (0,11pt);
    \draw[](-12pt,-9pt) -- (-8pt,1pt);
    \draw[](0pt,-9pt) -- (-8pt,1pt);
    \draw[](-8pt,1pt) -- (0,11pt);
    \draw[](0,11pt) -- (0,21pt);
  }\,
    $};
    \draw (main.north) -- ++(0,8pt);
    \draw (main.south) ++(-12pt,0) -- ++(0,-8pt);
    \draw (main.south) ++(0pt,0) -- ++(0,-8pt);
    \draw (main.south) ++(12pt,0) -- ++(0,-8pt);
  }\,
  \right] \dwant
  f_{x} \mapsto -\frac14 (g_{x-\frac12} - g_{x + \frac12}) \otimes (g_{x-\frac12} - g_{x + \frac12}) \otimes f_{x} + \frac14 f_{x}\otimes (g_{x-\frac12} - g_{x + \frac12})\otimes (g_{x-\frac12} - g_{x + \frac12}).
  \hspace{-1in}
  $$
  The reader is invited to check that the following proposed homotopy does in fact solve this equation, and transforms appropriately under the $\bS_{3}$-action:
  $$ \hspace{-1in}
   \,\tikz[baseline=(main.base)]{
    \node[draw,rectangle,inner sep=2pt] (main) {$
    -  \,\tikz[baseline=(basepoint)]{ 
    \path (0,3pt) coordinate (basepoint) (8pt,1pt) node[dot] {} (0,11pt) node[dot] {};
    \draw[](-12pt,-9pt) -- (0,11pt);
    \draw[](0pt,-9pt) -- (8pt,1pt);
    \draw[](12pt,-9pt) -- (8pt,1pt);
    \draw[](8pt,1pt) -- (0,11pt);
    \draw[](0,11pt) -- (0,21pt);
  }\,
  -
  \,\tikz[baseline=(basepoint)]{ 
    \path (0,3pt) coordinate (basepoint) (-8pt,1pt) node[dot] {} (0,11pt) node[dot] {};
    \draw[](12pt,-9pt) -- (0,11pt);
    \draw[](-12pt,-9pt) -- (-8pt,1pt);
    \draw[](0pt,-9pt) -- (-8pt,1pt);
    \draw[](-8pt,1pt) -- (0,11pt);
    \draw[](0,11pt) -- (0,21pt);
  }\,
    $};
    \draw (main.north) -- ++(0,8pt);
    \draw (main.south) ++(-12pt,0) -- ++(0,-8pt);
    \draw (main.south) ++(0pt,0) -- ++(0,-8pt);
    \draw (main.south) ++(12pt,0) -- ++(0,-8pt);
  }\,
  : f_{x} \mapsto \frac1{12} (g_{x-\frac12} - g_{x+\frac12}) \otimes f_{x} \otimes f_{x} - \frac16 f_{x} \otimes (g_{x-\frac12} - g_{x+\frac12}) \otimes f_{x} + \frac1{12}f_{x} \otimes f_{x} \otimes (g_{x-\frac12} - g_{x+\frac12}).\hspace{-1in}$$
  
  The homotopy controlling the Frobenius axiom should satisfy:
  $$ \left[ \d,
    \,\tikz[baseline=(main.base)]{
    \node[draw,rectangle,inner sep=2pt] (main) {$
          \,\tikz[baseline=(basepoint)]{ 
    \path (0,3pt) coordinate (basepoint) (0,1pt) node[dot] {} (0,11pt) node[dot] {};
    \draw[](-8pt,-9pt) -- (0,1pt);
    \draw[](8pt,-9pt) -- (0,1pt);
    \draw[](0,1pt) -- (0,11pt);
    \draw[](0,11pt) -- (-8pt,21pt);
    \draw[](0,11pt) -- (8pt,21pt);
  }\,
  +
  \,\tikz[baseline=(basepoint)]{ 
    \path (0,3pt) coordinate (basepoint) (8pt,1pt) node[dot] {} (-8pt,11pt) node[dot] {};
    \draw[](-8pt,-9pt) -- (-8pt,11pt);
    \draw[](8pt,-9pt) -- (8pt,1pt);
    \draw[](8pt,1pt) -- (-8pt,11pt);
    \draw[](-8pt,11pt) -- (-8pt,21pt);
    \draw[](8pt,1pt) -- (8pt,21pt);
  }\,
    $};
    \draw (main.south) ++(-6pt,0) -- ++(0pt,-8pt);
    \draw (main.south) ++(6pt,0) -- ++(0pt,-8pt);
    \draw (main.north) ++(-6pt,0) -- ++(0,8pt);
    \draw (main.north) ++(6pt,0) -- ++(0,8pt);
  }\,
  \right]
  \dwant
  \begin{cases}
  \displaystyle f_{x}\otimes f_{x} \mapsto -\frac14 f_{x} \otimes (g_{x - \frac12} + g_{x + \frac12}) ,
  & \quad
  \displaystyle f_{x} \otimes f_{x\pm 1} \mapsto \frac14 f_{x} \otimes g_{x \pm \frac12} ,
  \\[12pt]
  \displaystyle f_{x} \otimes g_{x\pm \frac12} \mapsto \mp \frac14 (g_{x- \frac12} - g_{x + \frac12})\otimes g_{x\pm \frac12} ,
  & \quad
  \displaystyle g_{x\pm \frac12} \otimes f_{x} \mapsto 0.
  \end{cases}
  $$
  All remaining matrix entries vanish  for reasons of either degree or quasilocality.  This homotopy illustrates that the composition of $0$-quasilocal operations might only be $1$-quasilocal.  One potential solution is:
  $$ \,\tikz[baseline=(main.base)]{
    \node[draw,rectangle,inner sep=2pt] (main) {$
          \,\tikz[baseline=(basepoint)]{ 
    \path (0,3pt) coordinate (basepoint) (0,1pt) node[dot] {} (0,11pt) node[dot] {};
    \draw[](-8pt,-9pt) -- (0,1pt);
    \draw[](8pt,-9pt) -- (0,1pt);
    \draw[](0,1pt) -- (0,11pt);
    \draw[](0,11pt) -- (-8pt,21pt);
    \draw[](0,11pt) -- (8pt,21pt);
  }\,
  +
  \,\tikz[baseline=(basepoint)]{ 
    \path (0,3pt) coordinate (basepoint) (8pt,1pt) node[dot] {} (-8pt,11pt) node[dot] {};
    \draw[](-8pt,-9pt) -- (-8pt,11pt);
    \draw[](8pt,-9pt) -- (8pt,1pt);
    \draw[](8pt,1pt) -- (-8pt,11pt);
    \draw[](-8pt,11pt) -- (-8pt,21pt);
    \draw[](8pt,1pt) -- (8pt,21pt);
  }\,
    $};
    \draw (main.south) ++(-6pt,0) -- ++(0pt,-8pt);
    \draw (main.south) ++(6pt,0) -- ++(0pt,-8pt);
    \draw (main.north) ++(-6pt,0) -- ++(0,8pt);
    \draw (main.north) ++(6pt,0) -- ++(0,8pt);
  }\,
  : 
  f_{x} \otimes g_{x \pm \frac12} \mapsto \mp \frac14 f_{x}\otimes g_{x\pm \frac12},
  $$
  with all other matrix entries vanishing.
  
  Turning now to the genus-one operations, we find ourselves faced with solving
  \begin{gather*}
   \left[ \d, \graphDbox\right]
  \dwant
  \begin{cases}
    \displaystyle f_{x}\otimes g_{x \pm \frac12} \mapsto \pm \frac1{12} g_{x\pm \frac12}, \\[12pt]
    \displaystyle g_{x \pm \frac12} \otimes f_{x} \mapsto \mp \frac1{12} g_{x \pm \frac12},
  \end{cases}
  \\
  \left[ \d, \graphAbox\right]
  \dwant 
  f_{x} \mapsto \frac1{12} (g_{x - \frac12} - g_{x + \frac12}) \otimes f_{x} + \frac1{12} f_{x} \otimes (g_{x - \frac12} - g_{x + \frac12}) .
  \end{gather*}
  The simplest solution is:
  $$ \graphDbox:
  g_{x+ \frac12}\otimes g_{x + \frac12} \mapsto - \frac1{12} g_{x+ \frac12}
  , \quad
  \graphAbox : f_{x} \mapsto \frac1{12} f_{x}\otimes f_{x}.
  $$
  
  Combining these, we see that
  $$\left[ \d,\graphBbox\right] \want -\frac1{12}\id \quad :\quad f_{x} \mapsto -\frac1{12} f_{x}, \quad g_{x+ \frac12} \mapsto -\frac1{12}g_{x+ \frac12}. $$
  But this is not exact.
\end{proof}


\section*{Acknowledgments} 
Gabriel C.\ Drummond-Cole provided enlightening discussions about both the general meaning and the onerous details of the calculations in this paper.  Many of those conversations occurred while I was a visitor at the IBS Center for Geometry and Physics (Korea), where I was also provided generous hospitality and a comfortable work environment.  The anonymous referee provided many valuable comments improving the content and exposition of this paper.  This work is supported by the NSF grant DMS-1304054.

%

\end{document}